\documentclass{article}
\usepackage{lmodern}
\usepackage[utf8]{inputenc}
\usepackage[T1]{fontenc}
\usepackage[english]{babel}
\usepackage{amsthm}
\usepackage{amsmath}
\usepackage{amssymb}
\usepackage{mathrsfs}
\usepackage{graphicx}
\usepackage{caption}
\usepackage{subcaption}
\usepackage{ucs}
\usepackage{bbm}
\usepackage{stmaryrd}
\usepackage[linktocpage]{hyperref}
\usepackage{geometry}
\newtheorem{theo}{Theorem}
\newtheorem{lem}{Lemma}
\newtheorem{prop}{Proposition}

\newtheorem{cor}{Corollary}

\author{Armand Riera}
\title{Isoperimetric inequalities in the Brownian plane\footnote{Supported by the ERC Advanced Grant 740943 {\sc GeoBrown}}}
\date{\small Universit\'e Paris-Sud}
\renewcommand{\leq}{\leqslant}
\renewcommand{\geq}{\geqslant}
\renewcommand{\epsilon}{\varepsilon}
\def\build#1_#2^#3{\mathrel{
\mathop{\kern 0pt#1}\limits_{#2}^{#3}}}

\begin{document}
\maketitle
\begin{abstract}
We consider the model of the Brownian plane, which is a pointed non-compact random metric space with the topology of the complex plane. The Brownian plane  can be obtained as the scaling limit in distribution of the uniform infinite planar triangulation or the uniform infinite planar quadrangulation and is conjectured to be the universal scaling limit of many others random planar lattices. We establish sharp bounds on the probability of having a short cycle separating the ball of radius $r$ centered at the distinguished point from infinity. Then we prove a strong version of the spatial Markov property of the Brownian plane. Combining our study of short cycles with this strong spatial Markov property we obtain sharp isoperimetric bounds for the Brownian plane.
\end{abstract}
\tableofcontents
\section{Introduction}\label{sec:int}
In recent years, much work and energy have been devoted to the study of discrete and continuous random geometry in dimension 2. In this paper we will  study the Brownian plane $\mathcal{M}_{\infty}$, which appears as the scaling limit in distribution of the uniform infinite planar  quadrangulation $Q_{\infty}$, in the local Gromov-Hausdorff sense and can also be interpreted as a Brownian map with infinite volume, see \cite{Plane}. The Brownian plane is a random pointed and weighted boundedly compact length space homeomorphic to $\mathbb{C}$ and  is conjectured to be the universal scaling limit of other discrete models. The case of the uniform infinite planar triangulation of type I has been treated in \cite{Bud}. We also mention that the Brownian plane is closely related to the Liouville quantum gravity surface called the quantum cone, see \cite[Corollary 1.5]{MS}.
\\
\\
The spaces $Q_{\infty}$ and $\mathcal{M}_{\infty}$ have a distinguished point, also called the root. Our first goal is to understand the probability of having a short injective cycle separating the ball of radius $r$ centered at the root in $\mathcal{M}_{\infty}$ from infinity. This will allow us to deduce isoperimetric inequalities for the Brownian plane. These results can then be extended to other models such as the Brownian sphere and the infinite Brownian disk.
The study of short separating cycles starts in random planar geometry with the paper \cite{Krikun}, where Krikun gave a construction of $Q_{\infty}$ as the local limit of large finite planar quadrangulations. He also proved the existence, for every $r\in \mathbb{N}^{*}$, of cycles separating the ball of radius $r$ of $Q_{\infty}$ from infinity having length of order $r$. Krikun conjectured that it is not possible to find  separating cycles with length of order smaller than $r$. In  \cite{Cycle},  Le Gall and Leh\'ericy confirmed Krikun's conjecture by proving that for every $\delta>0$, there exists a constant $c_{\delta}>0$ such that for every  $r\in \mathbb{N}^{*}$:
\[\mathbb{P}(L_{r}(Q_{\infty})< \epsilon r)<c_{\delta}\epsilon^{2-\delta}\]where $L_{r}(Q_{\infty})$ stands for the infimum of the lengths of injective cycles disconnecting the ball of radius $r$ of $Q_{\infty}$ from infinity. They also proved that the probability $\mathbb{P}(L_{r}(Q_{\infty})> u r)$ decreases exponentially fast when $u$ goes to infinity.
\\
\\  
This work can be seen as a continuous counterpart of this study. We are aiming at similar results for the Brownian plane. Thanks to the geometric properties of $\mathcal{M}_{\infty}$ we get  optimal results in the continuous setting. Since the Brownian plane is expected to be the universal scaling limit of random lattices such as the UIPQ and the UIPT, it is likely that these sharper results also have analogs for discrete models. Let us present our results more precisely. It should also be possible to adapt some of our techniques to the case of the UIPQ.
\\
\\
The Brownian plane $\mathcal{M}_{\infty}$ is equipped with a root, which we denote by $0$, a distance  $\Delta$ and a volume measure $|\cdot|$. The construction of $\mathcal{M}_{\infty}$  on a probability space $(\Omega, \mathcal{F},\Theta_{0})$ is recalled in Section \ref{subsection_Brownian_Plane}.
For every $r>0$, let $B_{r}(\mathcal{M}_{\infty})$ denote the closed ball of radius $r$ centered at $0$ in $\mathcal{M}_{\infty}$. For every path $\gamma:[t,t^{\prime}]\to \mathcal{M}_{\infty}$, we denote its length by $\Delta(\gamma)$ i.e.:
 \begin{equation}\label{def:distance}
 \Delta(\gamma):=\sup \limits_{t=t_{1}\leq t_{2},\ldots\leq t_{n}=t^{\prime}}\sum \limits_{i=1}^{n-1}\Delta\big(\gamma(t_{i}),\gamma(t_{i+1})\big)
 \end{equation}
 where the supremum is over all choices of the integer $n\geq 1$ and the finite sequence $t_{1}\leq t_{2}\leq\ldots\leq t_{n}$ satisfying $(t_{1},t_{n})=(t,t^{\prime})$. In this work a path has to be a continuous function. Moreover  we say that a path $\gamma:[t,t^{\prime}]\to \mathcal{M}_{\infty}$ is a separating cycle if:
\\
\\
$\bullet$  for every $t\leq s<s^{\prime}\leq t^{\prime}$ we have  $\gamma(s)= \gamma(s^{\prime})$ if and only if $(s,s^{\prime})=(t,t^{\prime})$;
\\
\\
$\bullet$ the distinguished point $0$ does not belong to the range of $\gamma$ and there exists $r>0$ such that for any path $\tilde{\gamma}:[s,s^{\prime}]\to \mathcal{M}_{\infty}$ with  $\tilde{\gamma}(s)=0$ and $\tilde{\gamma}(s^{\prime})\notin B_{r}(\mathcal{M}_{\infty})$ we have:
\[\gamma([t,t^{\prime}])\cap \tilde{\gamma}([s,s^{\prime}])\neq \emptyset.\]
We will say that a separating cycle $\gamma$ separates $B_{r}(\mathcal{M}_{\infty})$ from infinity if it takes values in the complement of $B_{r}(\mathcal{M}_{\infty})$.
Recall that $\mathcal{M}_{\infty}$ has a.s. the topology of $\mathbb{C}$ and consequently it has only one end.  So for every $r>0$, we can consider the hull of radius $r$, i.e. the complement of the unique  unbounded connected component of the complement of the closed ball of radius $r$ centered at the distinguished point. We denote the hull of radius $r$ by $B_{r}^{\bullet}(\mathcal{M}_{\infty})$. For every $r>0$ and any separating cycle $\gamma$ that separates $B_{r}(\mathcal{M}_{\infty})$ from infinity, an application of  Jordan's theorem shows that the path $\gamma$ has to take values in the complement of $B_{r}^{\bullet}(\mathcal{M}_{\infty})$. We say that such a path $\gamma$ separates  $B_{r}^{\bullet}(\mathcal{M}_{\infty})$ from infinity and we introduce the set $\mathcal{C}_{r}$ of all cycles  separating  $B_{r}^{\bullet}(\mathcal{M}_{\infty})$ from infinity, which is not empty since $B_{r}^\bullet(\mathcal{M}_{\infty})$ is bounded. Remark that any separating cycle $\gamma$ is in $\mathcal{C}_{r}$ for $r$ small enough and set:
\[L_{r}:=\inf \{\Delta(\gamma): \gamma \in \mathcal{C}_{r} \}. \]
 One of the benefits of working in the continuous setting is the fact that the Brownian plane is scale invariant in distribution, i.e. for every $r>0$, $(\mathcal{M}_{\infty},0,\Delta,|\cdot|)\overset{(d)}{=}(\mathcal{M}_{\infty},0,r\Delta, r^{4} |\cdot|)$ (see Section \ref{subsection_Brownian_Plane}). In particular, the scaling invariance implies that:
\[L_{r}\overset{(d)}{=} r L_{1} .\]
Therefore we will focus  on the variable $L_{1}$. 
We will prove the following result:
\begin{theo}\label{tail}
~\\
$\rm(i)$ We have
\[\limsup \limits_{u \to \infty} \frac{\log \big(\Theta_{0}(L_{1}>u)\big)}{u} \leq -\sup_{s>1}\frac{1}{2(s-1)}\log(\frac{s^{2}}{2s-1}) .\]
Consequently, $\Theta_{0}(L_{1}>u)$ decreases at least exponentially fast when $u$ goes to $\infty$. 
\\
\\
$\rm(ii)$ There exist two constants $0<c_{1}\leq c_{2}$ such that for every $\epsilon>0$:
\[ c_{1} (\epsilon^{2}\wedge 1)\leq \Theta_{0}(L_{1}<\epsilon) \leq c_{2} \epsilon^{2} . \]
\end{theo} 
\noindent It may be possible to get a discrete version of Theorem  \ref{tail} for the UIPQ by adapting our methods using the tree decomposition given in \cite{Discret}. This decomposition is the discrete analog of the construction of the Brownian plane that we will present in the preliminaries. Let us mention that our methods also allow us to obtain upper and lower bounds on the probability of the event $\{L_{1}<\epsilon\}$ under various conditionings. 
\\
\\
\\
In Section \ref{subsect_application} we prove a strong version of the spatial Markov property  of the Brownian plane, which has been
first derived in  \cite[Section 5.1]{Infinite_Spine}.  The statement of this property requires some notation and we give the
precise formulation of this property in Section \ref{subsect_application}.
 Combining Theorem \ref{tail} with this strong spatial Markov property we are able to study isoperimetric properties of the Brownian plane. 
Let us be more precise about this point.
\\
\\
We say that a closed subset $A$ of $\mathcal{M}_{\infty}$ is a (closed) Jordan domain if it is homeomorphic to the closed disk of the complex plane $\mathbb{C}$. 
Let $\mathcal{K}$ be the set of  all   Jordan domains of $\mathcal{M}_{\infty}$ whose interior contains the distinguished point of $\mathcal{M}_{\infty}$. For every $A\in\mathcal{K}$, we can define the length $\Delta(\partial A)$ of its boundary, as follows. We consider an injective cycle  $g:[0,1]\to \mathcal{M}_{\infty}$ such that  $g([0,1])=\partial A$ and we set: 
\[\Delta(\partial A) := \Delta(g) .\]
This definition does not depend on the parameterization $g$. We can now state our result concerning isoperimetric inequalities in the Brownian plane.
We will prove in Section \ref{seciso} that:
\begin{theo}\label{iso} For any nondecreasing function $f:\mathbb{R}_{+}\to (0,\infty)$:
\\
\\
$\rm(i)$ We have 
\[\inf \limits_{A\in \mathcal{K}} \frac{\Delta\big(\partial A\big)}{|A|^{\frac{1}{4}}} f(|\log(|A|)|)=0 \:,\:\Theta_{0}\text{-a.s.}\:,\:\text{if}\:\: \sum \limits_{m\in \mathbb{N}}\frac{1}{f(m)^{2}}=\infty.\]
\\
\\
$\rm(ii)$ We have
\[\inf \limits_{A\in \mathcal{K}} \frac{\Delta\big(\partial A\big)}{|A|^{\frac{1}{4}}} f(|\log(|A|)|)>0 \:,\:\Theta_{0}\text{-a.s.}\:,\:\text{if} \:\:\sum \limits_{m\in \mathbb{N}}\frac{1}{f(m)^{2}}<\infty.\]

\end{theo}
\noindent  Theorem \ref{iso} can be extended to the infinite volume Brownian disk (see Corollary \ref{isocor}) and the Brownian map (since the Brownian map and the Brownian Plane  are locally isometric \cite[Theorem 1]{Plane}).  In \cite{Cycle}, Le Gall and Leh\'ericy use  their study of short cycles to get an analog for the UIPQ of  Theorem \ref{iso} for the special case $f(x):=x^{\frac{3}{4}+\delta}$ for any $\delta>0$. We conclude this introduction by pointing out that the study of separating cycles appears naturally in other problems of random geometry; in the recent work \cite{Pants} the authors use a class of  separating cycles to obtain bijective enumerations of planar maps with three boundaries. They also discuss the statistics of the lengths of minimal separating loops in different discret models.

\section{Preliminaries}\label{preli}
The preliminaries are divided as follows. Section \ref{secsnake} gives a quick presentation of snake trajectories and the associated compact trees, we refer to \cite{ALG,DuquesneLeGall} for a more detailed description of these objects. Section \ref{subsection_Brownian_snake} presents the Brownian snake excursion, which is the  building block of the theory of Brownian geometry, and the special Markov property. Finally in Section \ref{subsect_coding_triple} and \ref{subsection_Brownian_Plane} we introduce the notion of a coding triple and give the construction of the Brownian plane and the infinite volume Brownian disk; these last sections follow \cite{Infinite_Spine}. Before starting the preliminaries, let us introduce some standard notation. 
\\
\\
Let  $(E,d)$ be a metric space.
\\
\\
$\bullet$ For $A\subset E$, we denote the closure (resp. the interior) of $A$ in $E$ by $\text{Cl}(A)$ (resp. $\text{Int}(A)$). Set $\partial A:=\text{Cl}(A)\setminus \text{Int}(A)$. 
\\
\\
$\bullet$ A path $\gamma$ on $E$ is a continuous function defined on an interval $I$ of $\mathbb{R}$ taking values in $E$. We say that the path $\gamma$ separates two subsets $A$ and $B$ of $E$ if the range of $\gamma$ does not intersect $A\cup B$ and if any path starting at $A$ and ending at $B$ intersects the range of $\gamma$. The path $\gamma$ is a geodesic on $E$ if for every $s,t\in I$, $d(\gamma(s),\gamma(t))=|s-t|$. 
\\
\\
$\bullet$ We denote the length of a path $\gamma$ by $d(\gamma)$. The definition of $d(\gamma)$ is the same as defined in \eqref{def:distance} replacing $\Delta$ by $d$. We say that $(E,d)$ is a length space if, for every $x,y\in E$, the distance $d(x,y)$ is the infimum of the quantities $d(\gamma)$ over all the paths $\gamma$ on $E$ starting at $x$ and ending at $y$.
\\
\\
$\bullet$ If $(E,d)$ is a length space and $U$ is a path-connected subset of $E$, the intrinsic distance induced by $d$ on $U$ is the distance $d_{U}$ on $U$ defined as follows:
\[\forall x,y\in U,\:d_{U}(x,y):=\inf\big\{d(\gamma):\:\gamma:[0,1]\to U\:\:\text{path with}\: (\gamma(0),\gamma(1))=(x,y)\big\}.\]
Remark that $d_{U}$ may take infinite values if $U$ is not an open subset of $E$.
\\
\\
$\bullet$ We say that a compact (resp. boundedly compact) metric space is weighted if it is given with a finite (resp. finite on compact sets) measure,
which is often called the volume measure.
We denote by $\mathbb{K}$ (resp. $\mathbb{K}_{\infty}$) the set of all isometry classes of pointed and weighted compact (resp. boundedly compact) metric spaces equipped with the  Gromov-Hausdorff-Prokhorov distance (resp. the local Gromov-Hausdorff-Prokhorov distance). Both $\mathbb{K}$ and
$\mathbb{K}_{\infty}$ are
 Polish spaces. 
\\
\\
Finally, we write $s\vee t:=\max(s,t)$, $s\wedge t:=\min(s,t)$ and by convention $\inf \emptyset:=\infty$. 
\subsection{Snake trajectories and labeled  trees}\label{secsnake}
Let $\mathcal{W}$ be the set of all continuous mappings $\text{w}:[0,\zeta_{\text{w}}]\to \mathbb{R}$, where $\zeta_{\text{w}}\geq 0$ is called  the lifetime of $\text{w}$. We will write $\widehat{\text{w}}=\text{w}(\zeta_{\text{w}})$ for the endpoint of $\text{w}$. For every  $x\in \mathbb{R}$, we identify $x$ with the map starting from $x$ with 0 lifetime. Set $\mathcal{W}_{x}:=\{\text{w} \in \mathcal{W}:~ \text{w}(0)=x\}$ and equip $\mathcal{W}$ with the distance:
\[d_{\mathcal{W}}(\text{w},\text{w}^{\prime})=|\zeta_{\text{w}}-\zeta_{\text{w}^{\prime}}|+\sup \limits_{t\geq 0}|\text{w}(t\wedge \zeta_{\text{w}})-\text{w}^{\prime}(t\wedge \zeta_{\text{w}^{\prime}})| .\]
Let $x\in \mathbb{R}$. A snake trajectory with initial point $x$ is a continuous mapping $\omega:s\mapsto \omega_{s}$ from $\mathbb{R}_{+}$ into $\mathcal{W}_x$ satisfying the following properties:
\\
\\
$\bullet$ $\omega_{0}=x$ and the quantity $\sigma(\omega):=\sup \{s\geq 0:~ \omega_{s}\neq x\}$ is finite. The quantity $\sigma(\omega)$ is called the lifetime of $\omega$. By convention $\sigma(\omega):=0$ if $\omega_{s}=x$ for every $s\geq 0$;
\\
$\bullet$ For every $ s, s^{\prime}\in \mathbb{R}_{+}$ with $s\leq s^{\prime}$, we have $\omega_{s}(t)= \omega_{s^{\prime}}(t)$ for every $t\leq \min \limits_{r\in[s,s^{\prime}]} \zeta_{\omega_{r}}$. This property is called the snake property.
\\
\\
We denote the set of all snake trajectories starting at $x$ by $\mathcal{S}_{x}$, and write $\mathcal{S}=\cup_{x\in\mathbb{R}}\mathcal{S}_{x}$ for the set of all snake trajectories. For every $\omega\in \mathcal{S}$ and $s\geq 0$, introduce the notation $W_{s}(\omega):=\omega_{s}$. The set $\mathcal{S}$ is equipped with the distance:
\[d_{\mathcal{S}}(\omega,\omega^{\prime}):=|\sigma(\omega)-\sigma(\omega^{\prime})|+\sup \limits_{s\geq 0}d_{\mathcal{W}}\big(W_{s}(\omega),W_{s}(\omega^{\prime})\big).\]
It is straightforward to verify that the space $(\mathcal{S},d_{\mathcal{S}})$ is a Polish space. To simplify notation, for every $\omega\in \mathcal{S}$, we set
\[\omega_{*}:=\inf \{\widehat{\omega}_{s}:~ s\geq 0 \}.\]
It will be important for our study to associate a compact $\mathbb{R}$-tree $\mathcal{T}_{\omega}$ with every snake trajectory $\omega$.
\\
\\
Let $\omega\in \mathcal{S}$ and define:
\[d_{\omega}(s,t):=\zeta_{\omega_{s}}+\zeta_{\omega_{t}}-2\inf \limits_{r\in[s\wedge t,s\vee t]} \zeta_{\omega_{r}}\]
for every $s,t\in [0,\sigma(\omega)]$. Since $s\mapsto \zeta_{\omega_{s}}$ is continuous, $d_{\omega}$ 
is a continuous pseudo-distance on $[0,\sigma(\omega)]$. We define an equivalence relation $\approx_{d_{\omega}}$ by setting $s\approx_{d_{\omega}}t$ if $d_{\omega}(s,t)=0$. The space $\mathcal{T}_{\omega}:=[0,\sigma(\omega)]/\approx_{d_{\omega}}$ equipped with the distance
induced by $d_{\omega}$ is a compact $\mathbb{R}$-tree.
Let $p_{\omega}:[0,\sigma(\omega)]\to \mathcal{T}_{\omega}$ be the canonical projection and let $V_{\omega}$ be  
the pushforward of Lebesgue measure on $[0,\sigma(\omega)]$ under $p_{\omega}$.
We view the tree $\mathcal{T}_{\omega}$ as a pointed  and weighted compact metric space, for which the volume 
measure is $V_\omega$ and the distinguished point  is $\rho_{\omega}:=p_{\omega}(0)$, which  is called the root of $\mathcal{T}_{\omega}$. For every $u\in \mathcal{T}_{\omega}$, set $\Lambda^{\omega}_{u}:=\widehat{\omega}_{t}$ where $t$ is any element of $p_{\omega}^{-1}(u)$. The quantity $\Lambda^{\omega}_{u}$ is well defined by the snake property and we interpret $\Lambda^{\omega}_{u}$ as a label assigned to $u$. The pair $\big(\mathcal{T}_{\omega},(\Lambda^{\omega}_{u})_{u\in\mathcal{T}_{\omega}}\big)$ is the labeled tree associated with the snake trajectory $\omega$. 
\\
\\
We will use the following standard nomenclature. 
Let $\mathcal{T}$ be a compact tree. The multiplicity of a point $a\in \mathcal{T}$ is the number of connected components of $\mathcal{T}\setminus\{a\}$. If the multiplicity of $a$ is $1$ (resp. $>2$), $a$ is called a leaf (resp. a branching point). 
\subsection{The Brownian snake excursion}\label{subsection_Brownian_snake}
To simplify notation, set $\widehat{W}_{s}(\omega)=\widehat{\omega}_{s}$ and $W_{*}(\omega)=\omega_{*}$ for every $\omega \in \mathcal{S}$.
Fix $x\in \mathbb{R}$. The Brownian snake excursion measure $\mathbb{N}_{x}$ is the unique $\sigma$-finite measure on $\mathcal{S}_{x}$ that satisfies the following properties:
\\
\\
$\bullet$ The distribution of $s\mapsto \zeta_{\omega_{s}}$ is the Itô measure of positive excursions of linear Brownian motion, with the normalization:
\[\forall \epsilon>0,\:\mathbb{N}_{x}(\sup\limits_{s\in[0,\sigma(\omega)]} \zeta_{\omega_{s}}>\epsilon)=\frac{1}{2\epsilon}~;\]
\\
$\bullet$ Conditionally on $(\zeta_{\omega_{s}})_{s\geq0}$, $(\widehat{W}_{s}(\omega))_{s\geq 0}$ is a Gaussian process with mean $x$ and covariance function:
\[\forall s,s^{\prime}\in [0,\sigma(\omega)],\:K(s,s^{\prime}):=\min\limits_{r\in [s\wedge s^{\prime},s\vee s^{\prime}]} \zeta_{\omega_{r}}.\]
Roughly speaking, conditionally on $(\zeta_{s})_{s\geq 0}$, the process $(W_{s})_{s\geq 0}$  evolves as follows. If $\zeta_{s}$ decreases, the path $W_{s}$ is shortened from its tip, while if $\zeta_{s}$ increases, the path $W_{s}$ is extended by adding "little pieces of linear Brownian motion" at its tip. We refer to \cite{Zurich} for a rigorous presentation.  \noindent For every $x,y\in \mathbb{R}$ with $x<y$ we have:
\begin{equation}\label{etoile}
\mathbb{N}_{y}\big(W_{*}<x\big)=\frac{3}{2(y-x)^{2}}
\end{equation}
see \cite[Chapter 6]{Zurich} for more details. 
To simplify notation, under $\mathbb{N}_{x}(d\omega)$ we will write $\sigma$ for $\sigma(\omega)$ and $W_{s}(t)$ for $\omega_{s}(t)$.
\\
\\
\textbf{Operations.}
We introduce a collection of elementary operations on $\mathcal{S}$.
\\
\\
$\bullet$ Translation: 
\\
For every snake trajectory $\omega$ and every $\lambda\in \mathbb{R}$, we will write $\omega+\lambda$ for the snake trajectory
\[(\omega+\lambda)_{s}(t):= \omega_{s}(t)+\lambda\,,\qquad 0\leq t\leq \zeta_{(\omega+\lambda)_{s}}:=\zeta_{\omega_{s}}.\]
By construction for every $x\in \mathbb{R}$ the pushforward measure of $\mathbb{N}_{x}$ under $\omega \mapsto \omega+\lambda$ is $\mathbb{N}_{x+\lambda}$. 
\\
\\
$\bullet$ Scaling:
\\ 
For every snake trajectory $\omega$ and every $\lambda\in \mathbb{R}_{+}^{*}$, we will write $\text{hom}_{\lambda}(\omega)$ for the snake trajectory defined by
\[ \text{hom}_{\lambda}(\omega)_{s}(t):= \lambda \omega_{s\lambda^{-4}}(t\lambda^{-2})\,,\qquad 0\leq t\leq \zeta_{\text{hom}_{\lambda}(\omega)_{s}}:=\lambda^{2}\zeta_{\omega_{s\lambda^{-4}}}.\]
It is also easy to deduce from the scaling property of Brownian motion that for every $x\in \mathbb{R}$ the pushforward measure of $\mathbb{N}_{x}$ under $\omega\mapsto \text{hom}_{\lambda}(\omega)$ is $\lambda^{2} \mathbb{N}_{\lambda x}$. We will call this property the scaling property of the Brownian snake excursion.
\\
\\
$\bullet$ Truncation:
\\
Let $(x,r)\in \mathbb{R}^{2}$ with $x> r$. For every $\text{w}\in \mathcal{W}_{x}$, let:
\[\text{hit}_{r}(\text{w}):=\inf\{t\in[0,\zeta_{\text{w}}]: \text{w}(t)= r\}\]
with the usual convention $\inf \emptyset=\infty$. 
 Consider $\omega\in\mathcal{S}_{x}$, and for every $s\geq 0$ set:
\[\eta_{s}^{(r)}(\omega):=\inf \Big\{t\geq 0: \int_{0}^{t}\mathbbm{1}_{\zeta_{\omega_{u}}\leq \text{hit}_{r}(\omega_{u})} \:du> s\Big\}.\]
The snake trajectory $\text{tr}_{r}(\omega)$ defined by
\[\forall s\geq 0,\:\big( \text{tr}_{r}(\omega)\big)_{s}:= \omega_{\eta_{s}^{(r)}(\omega)}\]
is called the truncation of $\omega$ at level $r$. See \cite[Proposition 10]{ALG}. Roughly speaking, 
$\text{tr}_{r}(\omega)$ is obtained by removing those paths $\omega$ that 
hit $r$ and then survive for a positive amount of time.
Let $\mathcal{Y}_{r}(\omega):=\sigma(\text{tr}_{r}(\omega))$ which can be interpreted as the time spent by $\omega$ before hitting $r$ and write  $\mathcal{H}^{x}_{r}$  for  the $\sigma$-field on $\mathcal{S}_{x}$ generated by $\text{tr}_{r}(W)$ and the class of all $\mathbb{N}_{x}$-negligible sets.

We now discuss the special Markov property of the Brownian snake excursion, which will be crucial in our study. 
The set:
\[\{s\geq 0:~ \text{hit}_{r}(W_{s})<\zeta_{s}\}\]
is open so it can be written as a union of disjoint open intervals $(a_{i},b_{i})_{i\in I}$ with $I$ an indexing set that may be empty. For every $i\in I$,  let $W^{(i)}$  be the  snake trajectory defined by:
\[W_{s}^{(i)}(t):=W_{(a_{i}+s)\wedge b_{i}}(\zeta_{a_{i}}+t)\:\: \text{for}\: 0\leq t\leq \zeta_{(a_{i}+s)\vee b_{i}}-\zeta_{a_{i}}\]
for every $s\geq 0$. 
By definition the snake trajectories $(W^{(i)})_{i\in I}$ are the excursions of $W$ below $r$. Note that  the information about the paths $W_{s}$ before hitting $r$ is contained in the sigma-field $\mathcal{H}^{x}_{r}$. The exit measure at level $r$ is the quantity:
\[\mathcal{Z}_{r}(\omega):=\liminf \limits_{\epsilon\downarrow 0}\frac{1}{\epsilon^{2}}\int_{0}^{\sigma}ds\mathbbm{1}_{\text{hit}_{r}(\omega_{s})=\infty,~\widehat{\omega}_{s}<r+\epsilon}.\]
The previous $\liminf$ is a well defined finite limit $\mathbb{N}_{x}$-a.e. (we refer to  \cite[Proposition 28]{Disks} for a proof) and it is $\mathcal{H}_{r}^{x}$-measurable by \cite[Proposition 2.3]{Laplacian_u_2}. 
We now can give a formal statement of the special Markov property.
\\
\\
\textbf{Special Markov property.} Let $x,r\in\mathbb{R}$, such that $x> r$. Under $\mathbb{N}_{x}$, conditionally on $\mathcal{H}^{x}_r$, the point measure:
\[\sum \limits_{i\in I} \delta_{W^{(i)}}(d\omega)\]
is Poisson with intensity $\mathcal{Z}_{r}\mathbb{N}_{r}(d\omega)$.
\\
\\
We refer to \cite[Corollary 21]{subor} for a proof.
It will be useful to note that for $r^{\prime}<r<x$, if we replace $\mathbb{N}_{x}(d\omega)$ by $\mathbb{N}_{x}(d\omega\:|\: W_{*}>r^{\prime})$, the last statement remains valid up to the replacement of $\mathcal{Z}_{r}\mathbb{N}_{r}(d\omega)$ by $\mathcal{Z}_{r}\mathbb{N}_{r}(d\omega \cap\: \{W_{*}>r^{\prime}\})$. The Laplace transform of $\mathcal{Z}_{r}$ is given by:
\begin{equation}\label{LaplaceZ}
\mathbb{N}_{x}\big(1-\exp(-\lambda\mathcal{Z}_{r})\big)=\Big(\lambda^{-\frac{1}{2}}+\sqrt{\frac{2}{3}}(x-r)\Big)^{-2}
\end{equation}
for every $\lambda\geq 0$. See e.g. formula (6) in \cite{Hull}. Remark that the limit when $\lambda$ goes to $\infty$ gives formula \eqref{etoile}.
\subsection{Coding triples and metric spaces}\label{subsect_coding_triple}~
\\
\\
\textbf{Infinite spine coding triples.}
An infinite spine coding triple is a triple $(\text{w},\mathfrak{N}^{+},\mathfrak{N}^{-})$ such that:
\begin{description}
\item[\rm(i)] $\text{w}:\mathbb{R}_{+}\to \mathbb{R}$ is a continuous function;
\item[\rm(ii)] $\mathfrak{N}^{+}=\sum \limits_{i\in I}\delta_{(t_{i},\omega^{i})}$ and $\mathfrak{N}^{-}=\sum \limits_{i\in J}\delta_{(t_{i},\omega^{i})}$ are point measures on $(0,\infty)\times \mathcal{S}$ ($I$ and $J$ are two disjoint indexing sets) and for every $i\in I\cup J$, $\omega^{i}\in\mathcal{S}_{\text{w}(t_{i})}$;
\item[\rm(iii)] the numbers $(t_{i})_{i\in I\cup J}$ are distinct;
\item[\rm(iv)] the functions 
\[u\mapsto \beta_{u}^{+}:=\sum \limits_{i\in I}\mathbbm{1}_{t_{i}\leq u}\sigma(\omega^{i}),\:\:\:u\mapsto \beta_{u}^{-}:=\sum \limits_{i\in J}\mathbbm{1}_{t_{i}\leq u}\sigma(\omega^{i})\]
take finite values, are monotone increasing on $\mathbb{R}_{+}$, and tend to $\infty$ at $\infty$;
\item[\rm(v)] for every $t>0$ and $\epsilon>0$:
\[\#\big\{i\in I\cup J:~t_{i}\leq t\:\:\text{and}\:\:\sup\limits_{s\in[0,\sigma(\omega^{i})]}|\hat{\omega}^{i}_{s}-\text{w}_{t_{i}}|>\epsilon\big\}<\infty.\]
\end{description}
We  define a scaling operation for coding triples as follows; for every $\lambda>0$
\[ \text{hom}_{\lambda}\Big(\text{w},\sum \limits_{i\in I}\delta_{(t_{i},\omega^{i})},\sum \limits_{i\in J}\delta_{(t_{i},\omega^{i})}\Big):=\Big(\lambda\text{w}(\cdot/\lambda^{2}),\sum \limits_{i\in I}\delta_{(\lambda^2 t_{i},\text{hom}_{\lambda}(\omega^{i}))},\sum \limits_{i\in J}\delta_{(\lambda^2 t_{i},\text{hom}_{\lambda}(\omega^{i}))}\Big).\]
An infinite spine coding triple belongs to the space $\mathcal{C}(\mathbb{R}_{+},\mathbb{R})\times M(\mathcal{S})\times M(\mathcal{S})$, where $M(\mathcal{S})$ stands for the space of all $\sigma$-finite measures $\mu$ on $(0,\infty)\times\mathcal{S}$ putting no mass on the 
set $\{(t,\omega):\sigma(\omega)=0\}$ and satisfying 
$\mu\big([0,t]\times \{\omega \in \mathcal{S}:~\sigma(\omega)>\delta\}\big)<\infty$, 
for every $t\geq 0$ and $\delta>0$. We equip the space $M(\mathcal{S})$ with the distance:
$$d_{M(\mathcal{S})}(\mu,\mu^{\prime}):=\sum \limits_{n\geq 0} d_{\text{Pro}}\big(\mu(\cdot \cap \mathcal{S}_{(n)}),\mu'(\cdot\cap \mathcal{S}_{(n)}))\wedge 2^{-n},$$
where 
$\mathcal{S}_{(n)}=[0,2^n]\times\{\omega\in\mathcal{S}:\sigma(\omega)> 2^{-n}\}$,
and $d_{\text{Pro}}$ stands for the Prokhorov metric inducing the weak topology on finite measures on $\mathbb{R}_{+}\times \mathcal{S}$. \\
We also equip $\mathcal{C}(\mathbb{R}_{+},\mathbb{R})\times M(\mathcal{S})\times M(\mathcal{S})$ with the product metric and the associated Borel sigma-field. 
\\
\\
Let $(\text{w},\mathfrak{N}^{+},\mathfrak{N}^{-})$ be an infinite spine coding triple. We now introduce the infinite  tree $\mathcal{T}_{\infty}$ associated with $(\text{w},\mathfrak{N}^{+},\mathfrak{N}^{-})$. For every $i \in I \cup J$, let $(\zeta^{i}_{s})$ be the lifetime process associated with $\omega^{i}$ and $\sigma^{i}:=\sigma(\omega^{i})$. We write $\mathcal{T}^{i}$ for the tree coded by $\zeta^{i}$, i.e.~$\mathcal{T}^i=\mathcal{T}_{\omega^{i}}$ , and $p_{\zeta^{i}}$ for the canonical projection from $[0,\sigma^{i}]$ onto $\mathcal{T}^{i}$. The tree $\mathcal{T}_{\infty}$ can be defined from the disjoint union:
\[ [0,\infty) \cup \Big(\bigcup_{i \in I\cup J}  \mathcal{T}^{i} \Big) \]
by identifying the point $t_{i}$ of $[0,\infty)$ with $p_{\omega^{i}}(0)$ (that is, the root of $\mathcal{T}^{i}$) for every $i\in I\cup J$. The set $[0,\infty)$ is called the spine of $\mathcal{T}_{\infty}$. 
\begin{figure}[!h]
 \begin{center}
    \includegraphics[height=7.5cm,width=8.5cm]{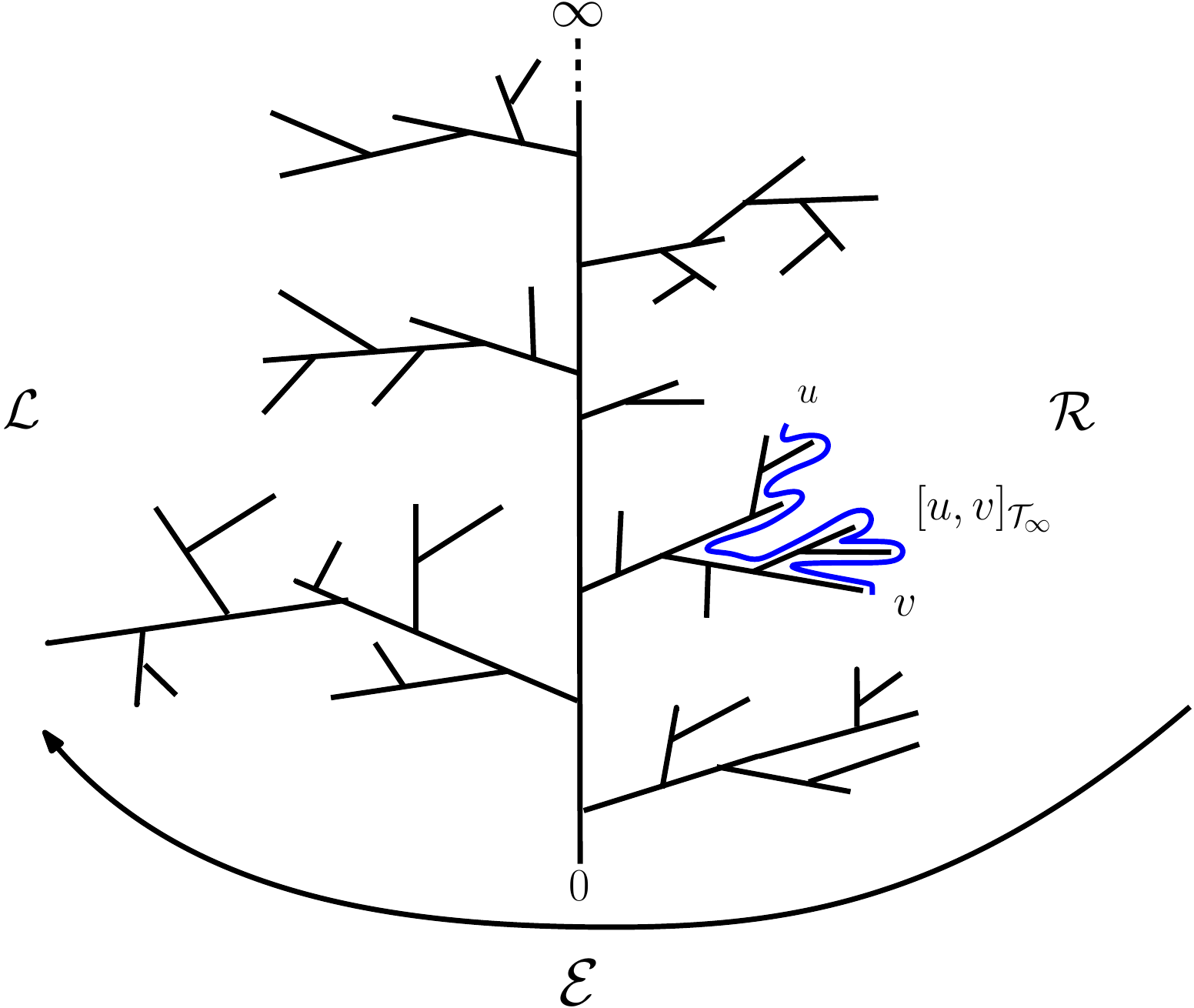}
 \caption{\label{fig:tree} A representation of the tree $\mathcal{T}_{\infty}$.}
 \end{center}
 \end{figure}
We equip $\mathcal{T}_{\infty}$ with a natural distance  $d_{\mathcal{T}_{\infty}}$ as follows. The restriction of $d_{\mathcal{T}_{\infty}}$ to the spine  is the Euclidean distance in $[0,\infty)$ and the restriction on each tree $\mathcal{T}^{i}$ is the tree distance $d_{\omega^{i}}$. If $u \in \mathcal{T}^{i}$ and $t \in [0,\infty)$, we take $d_{\mathcal{T}_{\infty}}(u,t)=d_{\omega^{i}}(u, p_{\omega^{i}}(0))+ |t_{i}-t|$. If $u \in \mathcal{T}^{i}$ and $v \in \mathcal{T}^{j}$ with $i \neq j$, we take $d_{\mathcal{T}_{\infty}}(u,v)=d_{\omega^{i}}(u, p_{\omega^{i}}(0))+ |t_{i}-t_{j}|+ d_{\omega^{j}}(v, p_{\omega^{j}}(0)) $. Then $(\mathcal{T}_{\infty},d_{\infty})$ is a (non-compact) $\mathbb{R}$-tree. We can also assign a label $\Lambda_{u}$, to each $u$ in  $\mathcal{T}_{\infty}$ as follows. If $t \in [0,\infty)$, we take $\Lambda_{t}:=\text{w}(t)$. If $u \in \mathcal{T}^{i}$, we take $\Lambda_{u}:=\Lambda_{u}^{\omega^{i}}$. In particular, we have $\big(\mathcal{T}^{i},(\Lambda_{u})_{u\in\mathcal{T}^{i}}\big)=\big(\mathcal{T}_{\omega^{i}},(\Lambda^{\omega^{i}}_{u})_{u\in\mathcal{T}_{\omega^{i}}}\big)$ for every $i\in I\cup J$. Moreover using property \rm(v), one checks that the mapping $u\mapsto \Lambda_{u}$ is continuous on $\mathcal{T}_{\infty}$. Finally, we can define a natural volume measure $V_{\mathcal{T}_{\infty}}$ on $\mathcal{T}_{\infty}$ as follows, $V_{\mathcal{T}_{\infty}}$ gives no mass to the spine and its restriction to $\mathcal{T}^{i}$ is $V_{\omega^{i}}$.
\\
\\
Roughly speaking, $\mathcal{T}_{\infty}$ is obtained by gluing the trees $\mathcal{T}^{i}$ along the spine and keeping their labels. It will be important for our purposes to equip $\mathcal{T}_{\infty}$ also with an order structure inherited from the coding triple. We define the left side of $\mathcal{T}_{\infty}$ as the subset:
\[ \mathcal{L}:=[0,\infty) \cup \Big(\bigcup_{i \in I}  \mathcal{T}^{i} \Big)\]
where again the point $t_{i}$ is identified  with $p_{\omega^{i}}(0)$ for $i\in I$, and we define the right side
$\mathcal{R}$ in the same way by replacing $I$ by $J$. 
Remark that $\mathcal{L}\cap \mathcal{R}=[0,\infty)$.  We write $\beta_{u-}^{+}$ and $\beta_{u-}^{-}$ for the left limits of $\beta^{+}$ and  $\beta^{-}$ at $u$ (and we take $\beta_{0-}^{+}= \beta_{0-}^{-}=0$ as convention). Note that if $u$ is a discontinuity point of $\beta^{+}$ then there is a unique $i \in I$ such that $t_{i}=u$ and $\beta_{u}^{+}-\beta_{u-}^{+}=\sigma^{i}$ (and the same property is true for $\beta^{-}$ replacing $I$ by $J$).
\\
\\
We define the exploration process $\mathcal{E}^{+}$ of the left side of $\mathcal{T}_{\infty}$ as follows.
\\
For every $s\geq 0$, there is a unique $u$ such that $\beta_{u-}^{+}\leq s \leq \beta_{u}^{+}$. Then if there is an index $i\in I$ such that $t_{i}=u$, set
\[\mathcal{E}_{s}^{+}:=p_{\omega^{i}}(s-\beta_{t_{i}-}^{+})\]
and if there is no such $i$, simply set $\mathcal{E}_{s}^{+}:=u$. We define similarly the exploration of the right side $\mathcal{E}^{-}$ by replacing $\beta^{+}$ by $\beta^{-}$ and $I$ by $J$.
Finally, let $\mathcal{E}$ be the function from $\mathbb{R}$ onto $\mathcal{T}_{\infty}$ defined by:
\[\mathcal{E}_{s}:=\begin{cases}
 \mathcal{E}_{s}^{+}\:& \text{if}\: s\geq 0 \\
  \mathcal{E}_{-s}^{-}\:& \text{if}\: s\leq 0
  \end{cases}\]
Remark that $\mathcal{E}$ is continuous and the volume measure on $\mathcal{T}_{\infty}$ is the pushforward of Lebesgue measure on $\mathbb{R}$ under the mapping $s\mapsto \mathcal{E}_{s}$. Moreover the left side of  $\mathcal{T}_{\infty}$ is $\{\mathcal{E}_{s}:~ s\geq 0\}$ and the right side is $\{\mathcal{E}_{s}:~ s\leq 0\}$. This exploration process allows us to define a notion of interval on $\mathcal{T}_{\infty}$. By convention,  for every  $s,t\in \mathbb{R}$ with $s<t$ we set $[t,s]:=(-\infty,s]\cup [t, \infty)$.
For every $u,v \in \mathcal{T}_{\infty}$ with $u\neq v$, let $[s,t]$ be the smallest interval such that $\mathcal{E}_{s}=u$ and $\mathcal{E}_{t}=v$. It is easy to check from the definition that there is always a smallest such interval. We put:
\[ [u,v]_{\mathcal{T}_{\infty}}:=\{\mathcal{E}_{r}:~ r\in [s,t]\}. \]
If $u=v$, take $[u,v]_{\mathcal{T}_{\infty}}=\{u\}$. 
Note that $[u,v]_{\mathcal{T}_{\infty}}\neq [v,u]_{\mathcal{T}_{\infty}}$ as long as $u \neq v$.  See Figure \ref{fig:tree} for an illustration.
\\
\\
By analogy with the case of compact trees, for every $u\in \mathcal{T}_{\infty}$,  the multiplicity of $u$ is the number of connected components of $\mathcal{T}_{\infty}\setminus \{u\}$. We will say that $u$ is a leaf (resp. a branching point) if its multiplicity is $1$ (resp. greater than $2$). Remark that:
\\
\\
$\bullet$ $0$ is the only leaf belonging to the spine.
\\
$\bullet$ The branching points belonging to the spine are the points $(t_{i})_{i\in I\cup J}$.
\\
$\bullet$ For every $i\in I\cup J$, the multiplicity of $a\in \mathcal{T}^{i}\setminus\{t_{i}\}$ in $\mathcal{T}_{\infty}$ is its multiplicity in $\mathcal{T}^{i}$. 
\\
\\
Finally, for every $u,v\in\mathcal{T}_{\infty}$, we denote  the unique geodesic segment of $\mathcal{T}_{\infty}$ connecting $u$ and $v$ by $\llbracket u,v \rrbracket_{\mathcal{T}_{\infty}}$. We write $u\preceq v$ for $u,v\in \mathcal{T}_{\infty}$ if and only if $u\in \llbracket 0,v \rrbracket_{\mathcal{T}_{\infty}}$. In this case we say that $u$ is an ancestor of $v$.
We also write $\llbracket u,\infty\llbracket_{\mathcal{T}_\infty}$ for the range of the unique geodesic from 
$u$ to $\infty$ in $\mathcal{T}_\infty$. 
\\
\\
\textbf{From coding triples to metric spaces.}
Let $(\text{w},\mathfrak{N}^{+},\mathfrak{N}^{-})$ be a coding triple and let $\big(\mathcal{T}_{\infty},(\Lambda_{v})_{v\in\mathcal{T}_{\infty}}\big)$ be the associated labeled tree. We make the following assumption:

\[
(H_{1}):\begin{cases}
\text{for every}\:\: v\in\mathcal{T}_{\infty},\:\: \Lambda_{v}\geq 0;\\
\text{if}\:\: \Lambda_{v}= 0\:\: \text{then}\:\: v\:\: \text{is a leaf;}\\
 \Lambda_{0}=0;\\
 \Lambda_{\mathcal{E}_{t}}\rightarrow \infty\:\:\text{as}\:\: |t|\to \infty.
\end{cases}\]
\noindent Set $\mathcal{T}_{\infty}^{\circ}:=\{v\in\mathcal{T}_{\infty}:\: \Lambda_{v}>0\}$ and $\partial \mathcal{T}_{\infty}:=\mathcal{T}_{\infty}\setminus\mathcal{T}_{\infty}^{\circ}$. Remark that $\mathcal{T}_{\infty}^{\circ}$ is path connected and dense in $\mathcal{T}_{\infty}$ by $(H_{1})$. The last assumption in $(H_{1})$  implies that $\inf \limits_{[u,v]_{\mathcal{T}_{\infty}}}\Lambda$ is attained for every interval $[u,v]_{\mathcal{T}_{\infty}}$ of $\mathcal{T}_{\infty}$.
\\
\\
For every $u,v\in \mathcal{T}_{\infty}$ set:
\[\Delta^{\circ}(u,v):=\begin{cases}
 \Lambda_{u}+\Lambda_{v}-2\max\big(\inf \limits_{[u,v]_{\mathcal{T}_{\infty}}}\Lambda,\inf \limits_{[v,u]_{\mathcal{T}_{\infty}}}\Lambda\big)\:& \text{if}\: \max\big(\inf \limits_{[u,v]_{\mathcal{T}_{\infty}}}\Lambda,\inf \limits_{[v,u]_{\mathcal{T}_{\infty}}}\Lambda\big)>0 \\
  \infty\:& \text{otherwise.}
  \end{cases}\]
We then let 
\begin{equation}
\label{Def-Delta}
\forall u,v\in \mathcal{T}_{\infty}^{\circ},\:\Delta(u,v):= \inf \limits_{u_{1}=u,u_{2},\ldots
,u_{n}=v} \sum \limits_{i=1}^{n-1} \Delta^{\circ}(u_{i},u_{i+1})
\end{equation}
 where the infimum is over all choices of the integer $n\geq 1$ and of the finite sequence $u_{1},\ldots,u_{n}$ of elements of $\mathcal{T}_{\infty}$ verifying $u_{1}=u$ and $u_{n}=v$. Using the continuity of $u\mapsto \Lambda_{u}$ one verifies that the mapping $(u,v)\mapsto \Delta(u,v)$ takes finite values and is continuous on $\mathcal{T}_{\infty}^{\circ}\times \mathcal{T}_{\infty}^{\circ}$. Since $\Delta^{\circ}(u,v)\geq  |\Lambda_{u}-\Lambda_{v}|$, we have for every $u,v \in \mathcal{T}_{\infty}^{\circ}$
\begin{equation}\label{major}
 \Delta(u,v)\geq |\Lambda_{u}-\Lambda_{v}|.
\end{equation}
It is important to remark that $\Delta$ defines a pseudo-distance on $\mathcal{T}_{\infty}^{\circ}$. From now on we make the extra assumption that:
\\
\\
$(H_{2})$: The map $(u,v)\mapsto \Delta(u,v)$ has a continuous extension to $\mathcal{T}_{\infty}\times \mathcal{T}_{\infty}$ 
\\
\\
and we consider this continuous extension in what follows. For simplicity we keep the notation $\Delta$ for this continuous extension, which defines a pseudo-distance on $\mathcal{T}_{\infty}$. The associated equivalence relation is defined by $u\approx v$ iff $\Delta(u,v)=0$.  By abuse of notation, we write $\mathcal{T}_{\infty}/\Delta$ for $\mathcal{T}_{\infty}/\approx$. Note that the definition of $u\approx v$
makes sense for $u,v\in \mathcal{T}_\infty^\circ$ even if $(H_{2})$ does not hold and so we can still consider the space $\mathcal{T}_{\infty}^{\circ}/\approx$
in that case. We denote the canonical projection by $\Pi:\mathcal{T}_{\infty}\to\mathcal{T}_{\infty}/\Delta$  and, for every $x\in\mathcal{T}_{\infty}/\Delta$, we set $\Lambda_{x}:=\Lambda_{u}$ where $u$ is any preimage of $x$ under $\Pi$ (remark that the definition is unambiguous by \eqref{major}). We write $|\cdot|$ for the pushforward of $V$ under $\Pi$, which defines a volume measure on $\mathcal{T}_{\infty}/\Delta$, and for simplicity we write $0$ for the equivalence class of $0$ in $\mathcal{T}_{\infty}/\Delta$.
The metric space $(\mathcal{T}_{\infty}/\Delta,0,\Delta,|\cdot|)$ is a weighted locally compact length space which is pointed at $0$, and we have:
\begin{equation}\label{distance_boundary}
 \Delta(x,\Pi(\partial \mathcal{T}_{\infty}))=\Lambda_{x} 
\end{equation}
for every $x\in \mathcal{T}_{\infty}/\Delta$.
We refer \cite[Subsection 4.1]{Infinite_Spine} for a proof of these two facts. 
For every $r\geq 0$, we write $B_{r}(\mathcal{T}_{\infty}/\Delta)$ for the set of all points $x\in \mathcal{T}_{\infty}/\Delta$ with $\Delta(x,\Pi(\partial \mathcal{T}_{\infty}))\leq r$. By \eqref{distance_boundary}:
\[B_{r}(\mathcal{T}_{\infty}/\Delta)=\{x\in \mathcal{T}_{\infty}/\Delta:\:\Lambda_{x}\leq r\}.\]
It will also be useful to introduce for every $r>0$, the set $\mathcal{T}_{\infty}^{r}$  of all points $u\in \mathcal{T}_{\infty}$ such that $\Lambda_{u}\geq r$ and $\Lambda_{v}>r$ for every $v\in [\![u,\infty[\![_{\mathcal{T}_{\infty}}\setminus\{u\}$. Remark that $\mathcal{T}_{\infty}^{r}$ is an $\mathbb{R}$-tree. We define:
\[\mathcal{T}_{\infty}^{r,\circ}:=\big\{u\in  \mathcal{T}_{\infty}:\:\inf\limits_{[\![u,\infty[\![_{\mathcal{T}_{\infty}}}\Lambda>r\big\}\]
and we let the  "boundary" $\partial \mathcal{T}_{\infty}^{r}$ be the set of all points $u\in \mathcal{T}_{\infty}$ such that $\Lambda_{u}= r$ and $\Lambda_{v}>r$ for every $v\in [\![u,\infty[\![_{\mathcal{T}_{\infty}}\setminus\{u\}$. Set $\check{B}_{r}^{\bullet}(\mathcal{T}_{\infty}/\Delta):=\Pi(\mathcal{T}_{\infty}^{r})$, $\check{B}_{r}^{\circ}(\mathcal{T}_{\infty}/\Delta):=\Pi(\mathcal{T}_{\infty}^{r,\circ})$ and:
\begin{equation}\label{equation_hull_open}
    B_{r}^{\bullet}(\mathcal{T}_{\infty}/\Delta):=\Pi\Big(\big\{u\in \mathcal{T}_{\infty}:~\inf\limits_{[\![u,\infty[\![_{\mathcal{T}_{\infty}}} \Lambda\leq r\big\}\Big)=\Pi\Big(\mathcal{T}_{\infty}\setminus\mathcal{T}_{\infty}^{r,\circ}\Big).
\end{equation}
When there is no ambiguity, we will remove $\mathcal{T}_{\infty}/\Delta$ from the notation and write $B_{r}^{\bullet},\check{B}_{r}^{\bullet}$ and $\check{B}_{r}^{\circ}$ instead. In the next section we explain the geometric meaning of these sets and we will see that the notation $B_{r}^{\bullet}$ is consistent with the one used in the introduction to designate the hull of the Brownian plane.

\subsection{The Brownian plane and the infinite volume Brownian disk}\label{subsection_Brownian_Plane}
In this section we give the construction of  the Brownian plane and the infinite volume Brownian disk    from random infinite spine coding triples. We also list some useful geometric properties of the Brownian plane.
\subsubsection{The Brownian plane}
\label{subsub-Plane}

We now consider a triple $(X,\mathfrak{L},\mathfrak{R})$ such that:
\\
$\bullet$ $X=(X_{t})_{t\geq 0}$ is a nine-dimensional Bessel process started from $0$;
\\
$\bullet$ Conditionally on $X$, $\mathfrak{L}$ and $\mathfrak{R}$ are independent Poisson point measures on $\mathbb{R}_{+}\times \mathcal{S}$ with intensity:
\[2dt\mathbb{N}_{X_{t}}(d\omega\cap\{\omega_{*}>0\}).\]
It is easy to verify that $(X,\mathfrak{L},\mathfrak{R})$ is a.s.~a coding triple in the sense of Section \ref{subsect_coding_triple}, and the root
of $\mathcal{T}_\infty$ is the only point with zero label. We may assume that $(X,\mathfrak{L},\mathfrak{R})$ is defined on the canonical space $\mathcal{C}(\mathbb{R}_{+},\mathbb{R})\times M(\mathcal{S})\times M(\mathcal{S})$ under the probability measure $\Theta_{0}$. As previously, we write $(\mathcal{T}_{\infty},(\Lambda_{v})_{v\in\mathcal{T}_{\infty}})$ for the associated infinite labeled tree. We note that $(X,\mathfrak{L},\mathfrak{R})$ satisfies assumptions $(H_{1},H_{2})$,
see \cite[Section 4.2]{Infinite_Spine}. In fact, since the root of $\mathcal{T}_\infty$
is the only point with zero label, it is possible to define directly the continuous extension of $\Delta$
to $\mathcal{T}_\infty\times\mathcal{T}_\infty$, just replacing $\Delta^\circ$ in formula  \eqref{Def-Delta} by 
$$\Delta^{\circ,\prime}(u,v):= \Lambda_{u}+\Lambda_{v}-2\max\big(\inf \limits_{[u,v]_{\mathcal{T}_{\infty}}}\Lambda,\inf \limits_{[v,u]_{\mathcal{T}_{\infty}}}\Lambda\big).$$
See
\cite[Section 4.2]{Infinite_Spine} for more details. The Brownian plane is the space $ (\mathcal{T}_{\infty}/\Delta,0,\Delta,|\cdot|)$ under $\Theta_{0}$. To simplify notation, we denote this space, which is an element of $\mathbb{K}_{\infty}$,  by $\mathcal{M}_{\infty}$. Remark that since $\partial \mathcal{T}_{\infty}=\{0\}$ we have $\Lambda_{x}=\Delta(0,x)$ for every $x\in \mathcal{M}_{\infty}$. Moreover, for every $\lambda>0$, the pushforward of $\Theta_{0}$ under $\text{hom}_{\lambda}$ is $\Theta_{0}$. Consequently, the space $ (\mathcal{T}_{\infty}/\Delta,0,\lambda \Delta,\lambda^{4}|\cdot|)$ is also distributed as a Brownian plane. Another important property is that , $\Theta_{0}$-a.s., we have 
\\
\\
$(F)$: For every $u,v\in \mathcal{T}_{\infty}$, with $u\neq v$, we have $\Delta(u,v)=0$ if and only if $\Delta^{\circ}(u,v)=\Delta^{\circ,\prime}(u,v)=0$. Moreover if  $u\neq v$ and $\Delta(u,v)=0$ then $u$ and $v$ must be leaves.
\\
\\
This fact is a classical result in Brownian geometry.  The first part of $(F)$ is derived in \cite[Section 3.2]{Hull}. The second part follows from the first part and the known results  for the Brownian map (see \cite[Lemma 3.2]{GallPaulin}).   By formulas (16) and (17) in \cite{Hull}, the set $ B_{r}^{\bullet}=B_{r}^{\bullet}(\mathcal{T}_{\infty}/\Delta)$
defined in \eqref{equation_hull_open} coincides with the hull of radius $r$ of $\mathcal{M}_{\infty}$ as defined in the introduction (Section \ref{sec:int}), $\check{B}_{r}^{\circ}$ is the complement of the hull, and  $\check{B}_{r}^{\bullet}$ is the closure of $\check{B}_{r}^{\circ}$. We also have 
\begin{equation}\label{injective_boundary}
    \partial B_{r}^{\bullet}=\partial \check{B}_{r}^{\bullet}=\Pi\big(\partial  \mathcal{T}_{\infty}^{r}\big)
\end{equation}
which is the range of an injective cycle (see the proof of \cite[Theorem 31]{Infinite_Spine} for more details). We will equip the hull $B_{r}^{\bullet}=\Pi\big(\mathcal{T}_{\infty}\setminus\mathcal{T}_{\infty}^{r,\circ}\big)$ with the distance $\Delta^{(r)}$ defined as follows. First set
\begin{equation}
\label{Def-Delta-Hull}
\forall u,v\in \mathcal{T}_{\infty}\setminus\mathcal{T}_{\infty}^{r,\circ},\:\Delta^{(r)}(u,v):=\build{\inf_{u=u_{1},u_{2},\ldots
u_{n}=v}}_{u_2,\ldots,u_{n-1}\in \mathcal{T}_{\infty}\setminus\mathcal{T}_{\infty}^{r,\circ}}^{} \sum_{i=1}^{n-1}\Delta^{\circ,\prime}(u_{i},u_{i+1}).
\end{equation}
By $(F)$, we see that for every $u,v\in \mathcal{T}_{\infty}\setminus\mathcal{T}_{\infty}^{r,\circ}$ we have $\Delta(u,v)=0$ iff  $\Delta^{(r)}(u,v)=0$. In particular, we can define $\Delta^{(r)}$ on the hull $B^{\bullet}_{r}$ taking for every $x,y\in B_{r}^{\bullet}$,  $\Delta^{(r)}(x,y):=\Delta^{(r)}(u,v)$ where $u,v\in \mathcal{T}_{\infty}\setminus\mathcal{T}_{\infty}^{r,\circ}$ are any elements such that $(\Pi(u),\Pi(v))=(x,y)$.  By definition $\Delta^{(r)}$ is a distance on $B_{r}^{\bullet}$ and it is not hard to verify that the restriction of  $\Delta^{(r)}$ on $\text{Int}(B_{r}^{\bullet})$ coincides with the intrinsic metric on $\text{Int}(B_{r}^{\bullet})$ viewed
as a subset of the metric space $(\mathcal{T}_\infty/\Delta,\Delta)$ (one can directly adapt the proof of \cite[Lemma 30]{Infinite_Spine}). In other words, $\Delta^{(r)}$ is the continuous extension to $B_{r}^{\bullet}$ of the intrinsic metric on $\text{Int}(B_{r}^{\bullet})$. In what follows, we will always view $B^\bullet_r$
as a (random) pointed and weighted compact metric space for the metric $\Delta^{(r)}$ (the volume
measure is obviously the restriction of the volume measure on $\mathcal{M}_\infty$ and the distinguished 
point is the same as in $\mathcal{M}_\infty$). 
\\
\\
\textbf{Exit measures.}
We now introduce the exit measures of the infinite tree $\mathcal{T}_{\infty}$. For every $a\geq 0$, set
\begin{equation}\label{tau}
    \tau_{a}:= \sup \{t\geq 0: X_{t}\leq a\}
\end{equation}
which is $\Theta_{0}$-a.s. finite since $X_{t}\rightarrow \infty$,  $\Theta_{0}$-a.s., when $t\to \infty$. We take $\tau_{\infty}:=\infty$ by convention.  For every $0\leq s\leq t\leq \infty$, introduce the point measures $\mathfrak{L}^{s,t}$ and $\mathfrak{R}^{s,t}$ on $\mathbb{R}_{+}\times \mathcal{S}$ defined as follows:
\[\int \Phi(\ell,\omega)\mathfrak{L}^{s,t}(d\ell\:d\omega):=\int_{\tau_{s}}^{\tau_{t}} \Phi(\ell-\tau_{s},\omega)\mathfrak{L}(d\ell\:d\omega)\]
and
\[\int \Phi(\ell,\omega)\mathfrak{R}^{s,t}(d\ell\: d\omega):=\int_{\tau_{s}}^{\tau_{t}} \Phi(\ell-\tau_{s},\omega)\mathfrak{R}(d\ell\:d\omega).\]
By the time reversal property of Bessel processes the process $(X_{(\tau_{t}-\ell)\vee 0})_{\ell\geq 0}$ is a Bessel process of dimension $-5$ started from $t$ stopped when it hits $0$ (see  \cite[Theorem 2.5]{Wil}). 
Applying this property with $t$ replaced by $t^{\prime}>t$, we get that $(X_{(\tau_{t}-\ell)\vee 0})_{\ell\geq 0}$ and $(X_{(\tau_{t}+\ell)})_{\ell\geq 0}$ are independent. Consequently, for every $0< t< \infty$:
\[\big((X_{(\tau_{t}+\ell)})_{\ell\geq 0},\mathfrak{L}^{t,\infty},\mathfrak{R}^{t,\infty}\big)\:\:\text{and}\:\:\big((X_{(\tau_{t}-\ell)\vee 0})_{\ell\geq 0}),\mathfrak{L}^{0,t},\mathfrak{R}^{0,t}\big)\:\:\text{are independent}.\]
We call this property the spine independence property of $\mathcal{T}_{\infty}$.  For every
 $0<r\leq s \leq t$ set:
\[ Z_{r}^{s,t}:=\int \mathcal{Z}_{r}(\omega)\mathfrak{R}^{s,t}(d\ell d\omega)+\int \mathcal{Z}_{r}(\omega)\mathfrak{L}^{s,t}(d\ell d\omega)\]
which is the total  exit measure at level $r$ accumulated by the snakes glued on $[\tau_{s},\tau_{t}]$.
To simplify notation, write $Z_{r}:=Z_{r}^{r,\infty}$. The proof of \cite[Lemma 4.2]{Hull} gives the following formula, for every $\lambda\geq 0$:
\begin{equation}\label{Zabc}
    \Theta_{0}\big(\exp(-\lambda Z_{r}^{s,t})\big)=\Big(\frac{t}{s}\Big)^{3}\cdot \Big(\frac{s-r+(r^{-2}+\frac{2}{3}\lambda)^{-\frac{1}{2}}}{t-r+(r^{-2}+\frac{2}{3}\lambda)^{-\frac{1}{2}}}\Big)^{3}.
\end{equation}
\noindent Consequently, computing the limit when $\lambda$ goes to infinity, we obtain:
\begin{prop}\label{corZabc0}
For every $0\leq r\leq s\leq t<\infty$:
\begin{equation}
    \Theta_{0}(Z_{r}^{s,t}=0)=\big(\frac{t}{s}\big)^{3} \cdot \big(\frac{s-r}{t-r}\big)^{3}.
\end{equation}
\end{prop} 
The special Markov property of the Brownian snake excursion implies that conditionally on $Z_{r}^{s,t}$ the excursions outside $r$ of the snake trajectories $\omega^{i}$ with $t_{i}\in[\tau_{s},\tau_{t}]$ are distributed as the atoms of  a Poisson point measure with intensity:
\[Z_{r}^{s,t} \mathbb{N}_{r}(\:d\omega \:\cap\{\omega_{*}>0\}).\]
We will use this property throughout the article. It will be also useful to note that the Laplace transform of $Z_{r}$ can be deduced from formula \eqref{Zabc} taking the limit when $t$ goes to infinity with $s=r$. More precisely,  for every $r>0$ and $\lambda\geq 0$ we get:
\begin{equation}\label{eq_Laplace_Z}
\Theta_{0}\big(\exp(-\lambda Z_{r})\big)=\big(1+\frac{2\lambda r^{2}}{3}\big)^{-\frac{3}{2}}.
\end{equation}
Equivalently $Z_{r}$ follows a Gamma distribution with parameter $\frac{3}{2}$ and mean $r^{2}$.
The previous formula appears already in \cite[Proposition 1.2]{Hull}, which also shows that 
$Z:=(Z_{r})_{r\geq 0}$ has a c\`adl\`ag modification, with only negative jumps, and from now on
we consider this modification. 
Furthermore, \cite[Proposition 4.3]{Hull}  states  that, for every $0\leq r\leq s$ and $\lambda\geq 0$, we have 
\begin{align}\label{eq_Z_sachant_Z}
\Theta_{0}\big(\exp(-\lambda Z_{r})\:|Z_{s}\big)=&\Big(\frac{s}{r+(s-r)(1+\frac{2\lambda r^{2}}{3})^{\frac{1}{2}}}\Big)^{3}\\
&\cdot \exp\Big(-\frac{3}{2}Z_{s}(\frac{1}{(s-r+(\frac{2\lambda}{3}+r^{-2})^{-\frac{1}{2}})^{2}}-\frac{1}{s^{2}})\Big)\nonumber.
\end{align}
We conclude this Subsection by giving a geometric interpretation of $Z_{r}$. One can derive from  \cite[Proposition 8]{Growth} that
\begin{lem}\label{lem:appen}
$\Theta_0$-a.s.~, for every $r>0$ we have:
\begin{equation}\label{geoZ}
Z_{r}=\lim \limits_{\epsilon \downarrow 0} \frac{1}{\epsilon^{2}} |\check{B}^{\circ}_{r}\cap B_{r+\epsilon}|.
\end{equation}
\end{lem}
For the sake of completeness we give a proof of Lemma \ref{lem:appen}, but we postpone it to the Appendix below to avoid to weigh down the preliminaries.
It will be important for us to know that the convergence holds simultaneously for every $r>0$.
Roughly speaking, $Z_{r}$ represents the length or perimeter (in a generalized sense) of $\partial B^{\bullet}_{r}$. 

\subsubsection{The infinite volume Brownian disk}\label{IBD}
We keep the assumptions and notation of the preceding Subsection.
Let $r>0$ and set $\widetilde{X}^{(r)}_{t}:=X_{\tau_{r}+t}-r$. Let us also introduce the point measures $\widetilde{\mathfrak{R}}_{r}$ and $\widetilde{\mathfrak{L}}_{r}$ on $\mathbb{R}_{+}\times \mathcal{S}$ defined by:
\[\int \Phi(t,\omega)\widetilde{\mathfrak{L}}_{r}(dtd\omega):=\int\Phi(t,\text{tr}_{0}(\omega-r))\mathfrak{L}^{r,\infty}(dt\:d\omega)\]
and
\[\int \Phi(t,\omega)\widetilde{\mathfrak{R}}_{r}(dtd\omega):=\int \Phi(t,\text{tr}_{0}(\omega-r))\mathfrak{R}^{r,\infty}(dt\:d\omega).\]
One easily checks that the triple $(\widetilde{X}^{(r)},\widetilde{\mathfrak{L}}_{r},\widetilde{\mathfrak{R}}_{r})$ is a random infinite spine coding triple satisfying $(H_{1})$. Moreover \cite[Proposition 6]{Infinite_Spine} shows that there exists a unique collection of probability measures $(\Theta_{z})_{z>0}$ on the space of coding triples such that for every $r>0$:
\begin{equation}\label{def_theta_z}
    \Theta_{0}\big(g(Z_{r})F(\widetilde{X}^{(r)},\widetilde{\mathfrak{L}}_{r},\widetilde{\mathfrak{R}}_{r})\big)=\frac{3^{\frac{3}{2}}}{\sqrt{2\pi}r}\int_{0}^{\infty}dz z^{\frac{1}{2}}\exp(-\frac{3}{2r^{2}} z) g(z) \Theta_{z}(F).
\end{equation}
and the pushforward of $\Theta_{z}$ by $\text{hom}_{\lambda}$ is $\Theta_{\lambda^{2} z}$ (for every $z,\lambda>0$).
In other words, conditionally on $Z_{r}=z$, the distribution of $(\widetilde{X}^{(r)},\widetilde{\mathfrak{L}}_{r},\widetilde{\mathfrak{R}}_{r})$ is $\Theta_{z}$. It is crucial that the preceding conditional distribution does not depend on $r$. Furthermore by \cite[Lemma 16]{Infinite_Spine} an infinite coding triple distributed according to $\Theta_{z}$ satisfies a.s.~$(H_{1},H_{2})$. Consequently we can consider the associated metric space and according to \cite[Proposition 38]{Infinite_Spine} this space is the infinite volume Brownian disk with perimeter $z$. The infinite volume Brownian disk is a random element of $\mathbb{K}_{\infty}$ and is a.s. homeomorphic to the complement of the open unit disk in the complex plane. It can also be obtained as scaling limit of random planar lattices with a boundary (see \cite{BMR}). The boundary of the infinite volume Brownian disk is the set of points that have no neighborhood homeomorphic to the (open) disk. The infinite volume Brownian disk also satisfies a scale invariance property. More precisely  since the pushforward of $\Theta_{z}$ by $\text{hom}_{\lambda}$ is $\Theta_{\lambda^{2} z}$,  if $(E,\rho_{E},\Delta_{E},|\cdot|_{E})$ is an infinite volume Brownian disk with perimeter $z$, then $(E,\rho_{E},\lambda \Delta_{E},\lambda^{4}|\cdot|_{E})$ is an infinite volume Brownian disk with perimeter $\lambda^{2} z$. \\
\\
We now explain the geometric interpretation of  $(\widetilde{X}^{(r)},\widetilde{\mathfrak{L}}_{r},\widetilde{\mathfrak{R}}_{r})$ and the implications of \eqref{def_theta_z} for the Brownian plane. First observe that the labeled tree associated with $(\widetilde{X}^{(r)},\widetilde{\mathfrak{L}}_{r},\widetilde{\mathfrak{R}}_{r})$ can be identified with $\big(\mathcal{T}_{\infty}^{r},(\Lambda_{v}-r)_{v\in \mathcal{T}_{\infty}^{r}}\big)$, see
the beginning of the proof of Theorem 29 in \cite{Infinite_Spine}. For every $r>0$,  let $\check{\Delta}^{(r)}$ be the intrinsic
distance induced by $\Delta$ on $\check{B}_{r}^{\circ}$ and also write  $|\cdot|_{\check{\Delta}^{(r)}}$ for the restriction of the volume measure $|\cdot|$ to $\check{B}_{r}^{\bullet}$.  The following lemma is then a consequence of  \cite[Lemma 30]{Infinite_Spine} and the identification of
$\big(\mathcal{T}_{\infty}^{r},(\Lambda_{v}-r)_{v\in \mathcal{T}_{\infty}^{r}}\big)$ with the labeled tree associated with $(\widetilde{X}^{(r)},\widetilde{\mathfrak{L}}_{r},\widetilde{\mathfrak{R}}_{r})$.

\begin{lem}\label{lem:H_2}
$\Theta_{0}$-a.s., for every $r>0$ such that $(\widetilde{X}^{(r)},\widetilde{\mathfrak{L}}_{r},\widetilde{\mathfrak{R}}_{r})$ satisfies $(H_{2})$ the following properties hold:
\\
\\
$\rm(i)$ The intrinsic distance $\check{\Delta}^{(r)}$ has a unique continuous extension to $\check{B}_{r}^{\bullet}$;
\\
\\
$\rm(ii)$ The space $\check{B}_{r}^{\bullet}$ equipped with this continuous extension of $\check{\Delta}^{(r)}$, the measure $|\cdot|_{\check{\Delta}^{(r)}}$ and the distinguished point $\Pi(\tau_{r})$ coincides as an element of $\mathbb{K}_{\infty}$ with the metric space associated with $(\widetilde{X}^{(r)},\widetilde{\mathfrak{L}}_{r},\widetilde{\mathfrak{R}}_{r})$.
\end{lem}

By \eqref{def_theta_z}, the coding triple  $(\widetilde{X}^{(r)},\widetilde{\mathfrak{L}}_{r},\widetilde{\mathfrak{R}}_{r})$ satisfies $(H_{2})$, $\Theta_{0}$-a.s., for every fixed $r>0$, and thus properties $\rm(i)$ and $\rm(ii)$ hold $\Theta_0$-a.s. when $r>0$ is fixed. However, we point out that we do not claim that $(\widetilde{X}^{(r)},\widetilde{\mathfrak{L}}_{r},\widetilde{\mathfrak{R}}_{r})$ satisfies $(H_{2})$ simultaneously for every $r>0$.  Consequently it is not
clear whether $\check{\Delta}^{(r)}$ has a unique continuous extension to $\check{B}_{r}^{\bullet}$ simultaneously for every $r>0$, a.s.
\\
\\
On the other hand, we saw in Section \ref{subsub-Plane} that the hull $B^\bullet_r$
(equipped with the distance $\Delta^{(r)}$ defined in \eqref{Def-Delta-Hull}) can be viewed as a random element of 
the space $\mathbb{K}$. In the next statement, we also view $\check B^\bullet_r$
as an element of $\mathbb{K}_\infty$ as explained in property $\rm(ii)$ of Lemma \ref{lem:H_2}. 
The following theorem is essentially a reformulation of Theorems 29 and 31 in \cite{Infinite_Spine}. 

\begin{theo}\label{spatial_fix_level}
Let $r>0$. Then, conditionally on $Z_{r}=z$, the coding triple $(\widetilde{X}^{(r)},\widetilde{\mathfrak{L}}_{r},\widetilde{\mathfrak{R}}_{r})$ is distributed according to $\Theta_{z}$ and  is independent of $B_{r}^{\bullet}$. 
Consequently, conditionally on  $Z_{r}=z$, the space  $\check{B}_{r}^{\bullet}$ is an infinite Brownian disk with perimeter $z$
and is independent of $B_{r}^{\bullet}$. 
\end{theo}

The fact that conditionally on $Z_{r}=z$ the coding triple $(\widetilde{X}^{(r)},\widetilde{\mathfrak{L}}_{r},\widetilde{\mathfrak{R}}_{r})$ is distributed according to $\Theta_{z}$ is just a reformulation of \eqref{def_theta_z}. Using property \rm(ii) of Lemma \ref{lem:H_2}, it follows that the conditional
distribution of $\check{B}_{r}^{\bullet}$ knowing that $Z_r=z$ is the law of the infinite volume Brownian disk with perimeter $z$.
The conditional
 independence of $\check{B}_{r}^{\bullet}$ and $B^\bullet_r$ given $Z_r$ is stated in \cite[Theorem 31]{Infinite_Spine},
 and the slightly stronger conditional independence of $(\widetilde{X}^{(r)},\widetilde{\mathfrak{L}}_{r},\widetilde{\mathfrak{R}}_{r})$
 and $B^\bullet_r$ is also established at the end of the proof of this result.

We will refer to the last assertion of Theorem \ref{spatial_fix_level} as the spatial Markov property of $\mathcal{M}_\infty$. 
 In Section \ref{subsect_application} below, we will extend this property to the case of a random level $r$.   

\section{Separating cycles}
In most of this section, we argue under $\Theta_0$ and we use the following notation:
\[B_{r,s}^{\circ}:=\text{Int}(B_{s}^{\bullet}\setminus B_{r}^{\bullet})\:\:\text{and}\:\:B_{r,s}^{\bullet}:=\text{Cl}(B_{r,s}^{\circ}).\]
for every $r,s\in (0,\infty)$ with $r<s$. 
Our first goal is to study the quantity:
\begin{equation}\label{def:L_r,s}
L_{r,s}=\inf \{\Delta(g):~g:[0,1]\to B_{r,s}^{\circ}\:\:\text{cycle separating }B^\bullet_r\text{ from }\infty\}. 
\end{equation}
Since $\mathcal{M}_{\infty}$ has the topology of the complex plane $\mathbb{C}$, the quantity $L_{r,s}$ is well defined. Actually by construction it only depends on $B_{r,s}^{\circ}$ and the intrinsic distance induced by $\Delta$ on $B_{r,s}^{\circ}$.  Let us briefly justify the measurability of
the random variable $L_{r,s}$. We consider a dense sequence $(a_n:n\in\mathbb{N})$ in $\mathcal{M}_\infty$. Given $\alpha>0$, we observe that $L_{r,s}< \alpha$
if and only if, for some $\delta>0$, the following holds for every $\epsilon>0$: There exists a finite sequence $a_{n_1},a_{n_2},\ldots,a_{n_p},a_{n_{p+1}}=a_{n_1}$ such that $\Delta(a_{n_i},(B^\circ_{r,s})^c)\geq \delta$, $\Delta(a_{n_i},a_{n_{i+1}})<\epsilon$ for every $1\leq i\leq p$, and
$$\sum_{i=1}^p \Delta(a_{n_i},a_{n_{i+1}})<\alpha-\delta,$$
and such that for any other sequence $a_{m_1},\ldots,a_{m_q}$ with $a_{m_1}\in B^\bullet_r$, $a_{m_q}\in\check B^\bullet_s$, and
$\Delta(a_{m_j},a_{m_{j+1}})<\epsilon$ for every $1\leq j\leq q-1$, there exist $i\in\{1,\ldots,p\}$ and $j\in\{1,\ldots,q\}$ such that
$\Delta(a_{n_i},a_{m_j})<\epsilon$. 

\vspace{0.2cm}
\subsection{Geometric properties}
We argue under $\Theta_{0}$ and use the notation introduced in
Section \ref{subsect_coding_triple}. In particular, $(\mathcal{E}_s)_{s\in\mathbb{R}}$ is the exploration function
of the tree $\mathcal{T}_\infty$. Our goal here is to identify a subclass of separating cycles and then to show that we can restrict our study to this collection of paths which is easier to study. 
\\
\\
For every $t\geq 0$, set $p_{\infty}^{(\ell)}(t):=\mathcal{E}_{\sup\{s\in \mathbb{R}:\: \Lambda_{s}=t\}}$ and $p_{\infty}^{(r)}(t):=\mathcal{E}_{\inf\{s\in \mathbb{R}:\: \Lambda_{s}=t\}}$. Remark that $p_{\infty}^{(\ell)}$ (resp. $p_{\infty}^{(r)}$) takes values in $\mathcal{L}$ (resp. $\mathcal{R}$) since $\mathcal{E}_{0}=0$. By definition for every $s\leq t$, we have 
\[\inf \limits_{[p^{(\ell)}_{\infty}(s),p^{(\ell)}_{\infty}(t)]_{\mathcal{T}_{\infty}}}\Lambda=\inf \limits_{[p^{(r)}_{\infty}(t),p^{(r)}_{\infty}(s)]_{\mathcal{T}_{\infty}}}\Lambda=s.\]
Consequently we have
\begin{equation}\label{geo_p}
    \Delta^{\circ}\big(p_{\infty}^{(\ell)}(t),p_{\infty}^{(\ell)}(s)\big)=\Delta^{\circ}\big(p_{\infty}^{(r)}(t),p_{\infty}^{(r)}(s)\big)=|t-s|
\end{equation}
for every $s,t>0$. Moreover knowing that for every $t\geq 0$ and $u\in[p_{\infty}^{(\ell)}(t),p_{\infty}^{(r)}(t)]_{\mathcal{T}_{\infty}}$ we have $\Lambda_{u}\geq t$, we get that  $\Delta^{\circ}(p_{\infty}^{(\ell)}(t),p_{\infty}^{(r)}(t))=0$ for every $t>0$. We write $$\gamma_{\infty}(t):=\Pi\big(p_{\infty}^{(\ell)}(t)\big)=\Pi\big(p_{\infty}^{(r)}(t)\big)\,,\qquad t\in [0,\infty).$$
By \eqref{major} and \eqref{geo_p}, $\gamma_{\infty}$ is a geodesic path connecting $0$ and $\infty$. It can be shown that this geodesic path is the unique geodesic path connecting $0$ and $\infty$ (see \cite[Proposition 15]{Plane}) but we will not use this result in this work. To simplify notation set $P^{(\ell)}:=p_{\infty}^{(\ell)}(\mathbb{R}_{+})$, $P^{(r)}:=p_{\infty}^{(r)}(\mathbb{R}_{+})$ and $P:=P^{(\ell)}\cup P^{(r)}$. Remark that $\Pi(P)$ is the range of $\gamma_{\infty}$.
\\
\\
We define the left (resp. right) side of $\mathcal{M}_{\infty}$ as the subset $\Pi(\mathcal{L})$ (resp. $\Pi(\mathcal{R})$). 
\begin{lem}\label{Topological_lemma} 
The following properties hold $\Theta_{0}$-a.s.
\\
\\
$\rm(i)$ The maps $\mathbb{R}_{+}\ni t\mapsto \Pi(t)$ and $ \mathbb{R}_{+}\ni t\mapsto \gamma_{\infty}(t)$ are injective. Moreover $\Pi([0,\infty))\cap \Pi(P)=\{0\}$.
\\
\\
$\rm(ii)$ The sets $\text{Int}(\Pi(\mathcal{L}))$ and $\text{Int}(\Pi(\mathcal{R}))$ are the connected components of the complement of $\Pi([0,\infty))\cup \Pi(P)$. 
\end{lem}

\begin{proof}~
\\
\\
$\rm(i)$ Since  $\mathbb{R}_{+}\ni t\mapsto \gamma_{\infty}(t)$ is a geodesic path it has to be injective. Moreover, as the only leaf on the spine $[0,\infty)$ is $0$, we can apply $(F)$ to deduce that $t\in \mathbb{R}_{+}\mapsto \Pi(t)$ is also injective and  that $\Pi([0,\infty))\cap \Pi(P)\subset\{0\}$. 
\\
\\
$\rm(ii)$ As a simple consequence of $(F)$, a point $x$ belongs to the boundary of $\Pi(\mathcal{L})$ iff it belongs to $\Pi([0,\infty))$ or to $\Pi(P^{(\ell)})=\Pi(P)$, and similarly if $\mathcal{L}$ is replaced by $\mathcal{R}$. Consequently:
\[\text{Int}\big(\Pi(\mathcal{L})\big)=\Pi(\mathcal{L})\setminus\big(\Pi([0,\infty))\cup \Pi(P)\big)\:\:,\:\:\text{Int}\big(\Pi(\mathcal{R})\big)=\Pi(\mathcal{R})\setminus\big(\Pi([0,\infty))\cup \Pi(P)\big).\]
Thanks again to $(F)$ we have $\text{Int}\big(\Pi(\mathcal{L})\big)\cap \text{Int}\big(\Pi(\mathcal{R})\big)=\emptyset$. Since $\mathcal{M}_{\infty}$ has the topology of the complex plane and  $\mathcal{M}_{\infty}\setminus \big(\Pi([0,\infty))\cup \Pi(P)\big)$ is the union of $\text{Int}\big(\Pi(\mathcal{L})\big)$ and $\text{Int}\big(\Pi(\mathcal{R})\big)$ the desired result follows.
\end{proof}
\noindent Let us introduce the subclass of separating cycles that will play an important role. We define the set $\mathcal{A}$  of all paths $\gamma:[t,t^{\prime}]\mapsto \mathcal{M}_{\infty}$ such that:
\begin{itemize}
\item[$\bullet$]
$\gamma(t)=\gamma(t^{\prime})$ is in $\Pi(P)$ and $\gamma$ does not hit $0$;
\item[$\bullet$]
 For every $t\leq s<s^{\prime}\leq t^{\prime}$, we have $\gamma(s)=\gamma(s^{\prime})$ if and only if $(s,s^{\prime})=(t,t^{\prime})$;
 \item[$\bullet$] There exist two times $t_{1}\leq t_{2}$ in $[t,t']$, such that $\gamma(t_{1}),\gamma(t_{2})\in \Pi([0,\infty))$, $(\gamma(t))_{t\in[t_1,t_2]}$
does not intersect $\Pi(P)$, and 
$\gamma(s)\in \Pi(\mathcal{R})$ (resp. $\gamma(s)\in \Pi(\mathcal{L})$) for every $s\in[t,t_{1}]$  (resp. for every $s\in[t_{2},t']$).
\end{itemize}

\begin{figure}[!h]
 \begin{center}
    \includegraphics[height=6.5cm,width=10cm]{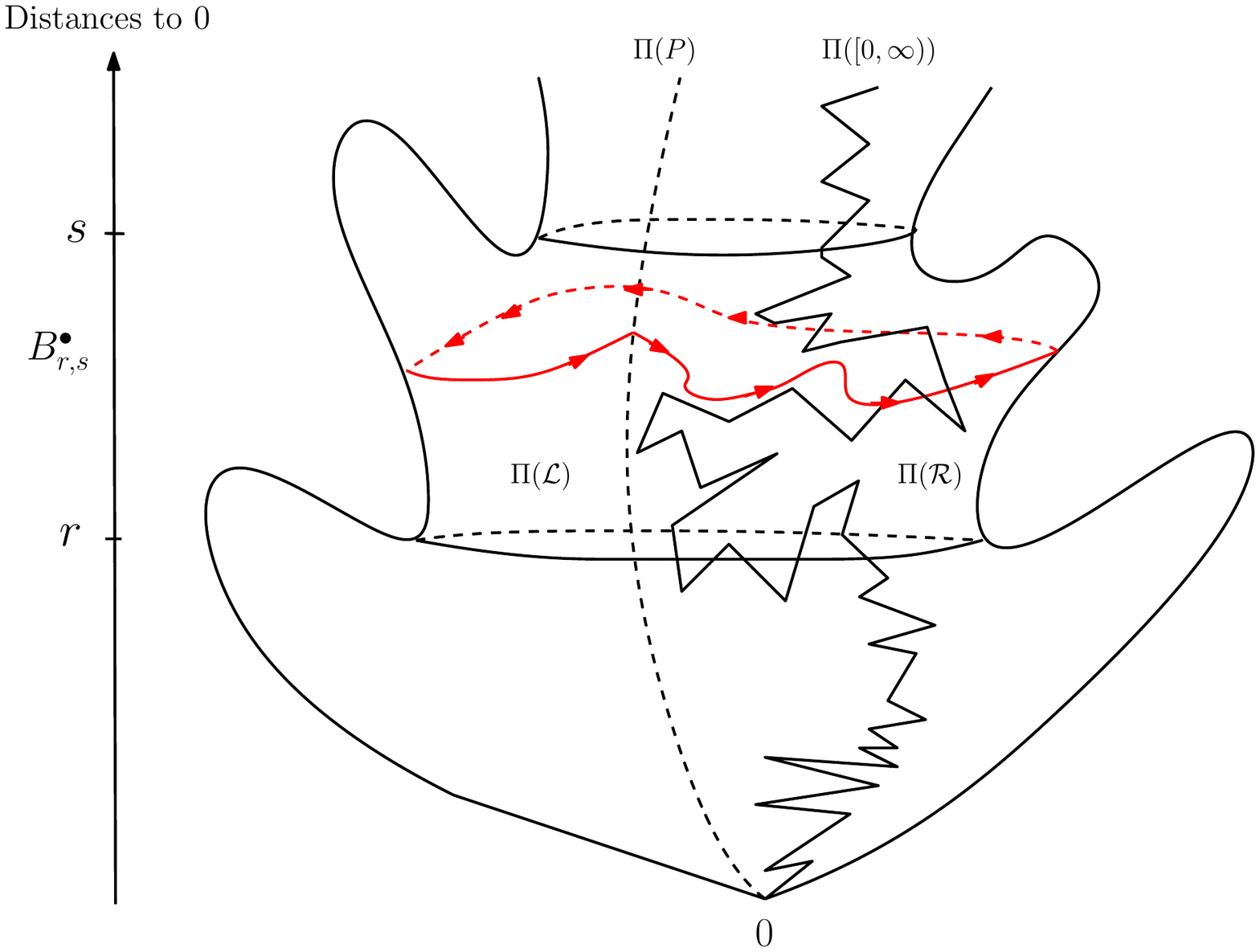}
 \caption{\label{fig:A}Cactus representation of the Brownian plane. The vertical distances represent the distances to the root $0$. In red a path of $\mathcal{A}$ taking values in $B_{r,s}^{\circ}$.}
 \end{center}
 \end{figure}
Using Lemma \ref{Topological_lemma} and the fact that $\mathcal{M}_{\infty}$ is homeomorphic to $\mathbb{C}$, one easily verifies that any path $\gamma$ in $\mathcal{A}$ is a separating cycle. So to give an upper bound for $L_{r,s}$ for $r<s$, it is sufficient to construct a path $\gamma\in \mathcal{A}$ taking values in $B_{r,s}^{\circ}$.  See Figure \ref{fig:A} for an illustration. In the next lemma we explain why we can restrict our attention to the subclass $\mathcal{A}$.

\begin{lem}\label{sepa_A}
$\Theta_{0}$-a.s.~for every separating cycle $\gamma$ there exists a path $\gamma^{\prime}$ in $\mathcal{A}$ such that $\Delta(\gamma^{\prime})\leq \Delta(\gamma)$. Moreover if $\gamma$ takes values in $B_{r,s}^{\circ}$ then the path  $\gamma^{\prime}$ also takes values in $B_{r,s}^{\circ}$.
\end{lem}

\begin{proof}
Let $(\gamma(t))_{t\in[0,1]}$ be a separating cycle. Since $\gamma$ does not hit $0$, there exist $r<s$ such that $\gamma$ takes values in $B_{r,s}^{\circ}$. In what follows we fix $r<s$ such that  $\gamma$ stays in $B_{r,s}^{\circ}$. Our goal is to show that there exists $\gamma^{\prime}\in \mathcal{A}$ taking values in $B_{r,s}^{\circ}$ such that $\Delta(\gamma^{\prime})\leq \Delta(\gamma)$. First notice that since the path $(\gamma_{\infty}(t))_{t\geq 0}$ connects $0$ and $\infty$, the range of $\gamma$ has to intersect  $\Pi(P)=\gamma_{\infty}\big(\mathbb{R}_{+}\big)$. Without loss of generality we may and will assume that $\gamma(0)=\gamma(1) \in \Pi(P)$.  Let $(t_{i},t_{i}^{\prime})_{i\in \mathcal{I}}$ be  the connected components of  $\{t\in[0,1]:\:\gamma(t)\notin \Pi(P)\}$ and to simplify notation set $\gamma^{i}:=\gamma_{[t_{i},t_{i}^{\prime}]}$. Remark that  $\gamma^{i}$ hits $\Pi(P)$ only at times $t_{i}$ and $t_{i}^{\prime}$. In particular, since  $\gamma$ does not hit $0$ we can use Lemma \ref{Topological_lemma} to obtain that for every $i\in \mathcal{I}$ there exists $\epsilon>0$ such that:
\[\forall t\in [0,\epsilon],\:\gamma^{i}(t_{i}+t)\in\Pi(\mathcal{R})\:\:\text{or}\:\:\forall t\in [0,\epsilon],\:\gamma^{i}(t_{i}+t)\in\Pi(\mathcal{L}).\]
In the first case we say that $\gamma^{i}$ starts in $\Pi(\mathcal{R})$ and in the second case that $\gamma^{i}$ starts in $\Pi(\mathcal{L})$. Similarly  by Lemma \ref{Topological_lemma} there exists $\epsilon>0$ such that:
\[\forall t\in [0,\epsilon],\:\gamma^{i}(t_{i}^{\prime}-t)\in\Pi(\mathcal{R})\:\:\text{or}\:\:\forall t\in [0,\epsilon],\:\gamma^{i}(t_{i}^{\prime}-t)\in\Pi(\mathcal{L}).\]
In the first case we say that $\gamma^{i}$ ends in $\Pi(\mathcal{R})$ and in the second case that $\gamma^{i}$ ends in $\Pi(\mathcal{L})$. Then since $\gamma$ is a separating cycle we claim that
\\
\\ 
$(C)$: there exists $i\in \mathcal{I}$ such that $\gamma^{i}$ starts in $\Pi(\mathcal{R})$ and ends in  $\Pi(\mathcal{L})$, or  starts in $\Pi(\mathcal{L})$ and ends in  $\Pi(\mathcal{R})$.
\\
\\
Let us explain why $(C)$ holds. 
Thanks to Lemma \ref{Topological_lemma}, we can find a homeomorphism $h:\mathcal{M}_{\infty}\to \mathbb{C}$ such that $h(\Pi(\mathbb{R}_{+}))=\mathbb{R}_{-}$, $h(\gamma_{\infty}(\mathbb{R}_{+}))=\mathbb{R}_{+}$, and
$h(\Pi(\mathcal{L}))$ and $h(\Pi(\mathcal{R}))$ are the upper and lower half-planes. In particular, since $h(0)=0$, the injective cycle $h\circ\gamma$ separates $0$ from infinity in the complex plane $\mathbb{C}$. Let $(\rho(t),\theta(t))_{t\in[0,1]}$ be two continuous functions from $[0,1]$ into $\mathbb{R}_{+}$ such that $h(\gamma(t))=\rho(t)\exp(i\theta(t))$. Since $\gamma(0)\in \gamma_{\infty}(\mathbb{R}_{+})$, we can take $\theta(0)=0$ which determines in a unique way the pair $(\rho(t),\theta(t))_{t\in[0,1]}$. Since $h\circ\gamma$ separates $0$ from infinity in $\mathbb{C}$, an application of Jordan's theorem shows that we must have $\theta(1)\in\{2\pi,-2\pi\}$. Suppose that $\theta(1)=2\pi$ for
definiteness and set $s_0:=\sup\{t\in[0,1]:\theta(0)=0\}$
and $s_1=\inf\{t\in[s_0,1]:\theta(t)=2\pi\}$. Necessarily there exists $i$ such that $\gamma^{i}=(\gamma(r))_{r\in[s_0,s_1]}$, and the ``excursion''   $\gamma^{i}$ starts in $\Pi(\mathcal{R})$ and ends in  $\Pi(\mathcal{L})$ or  conversely.
\\
\\
Let us derive the lemma from the claim $(C)$. Up to replacing $\gamma$ by $(\gamma(1-t))_{t\in[0,1]}$, we can assume
that there exists $i\in \mathcal{I}$ such that $\gamma^{i}$ starts in $\Pi(\mathcal{R})$ and ends in $\Pi(\mathcal{L})$. Since $\gamma(0),\gamma^{i}(t_{i}),\gamma^{i}(t_{i}^{\prime})\in \Pi(P)$, we can consider the geodesic path  $g$ (resp. $g^{\prime}$) taking values in $\Pi(P)$ starting at $\gamma(0)$ and ending at  $\gamma(t_{i})$ (resp. starting at $\gamma(t_{i}^{\prime})$ and ending at $\gamma(1)=\gamma(0)$). The concatenation of $g^{\prime}$, $\gamma^{i}$ and $g$ gives a path $\gamma^{\prime}$ in $\mathcal{A}$, which is shorter than $\gamma$. Moreover as $\gamma$ takes values in $B_{r,s}^{\circ}$ we have $\gamma(0),\gamma^{i}(t_{i}),\gamma^{i}(t_{i}^{\prime})\in \gamma_{\infty}\big((r,s)\big)$ and consequently $g$ and $g^{\prime}$ takes values in $\gamma_{\infty}\big((r,s)\big)$. We conclude that the concatenation of $g^{\prime}$, $\gamma^{i}$ and $g$ takes values in  $B_{r,s}^{\circ}$.
\end{proof}
\subsection{Lower bound for the tail of $L_{1}$ near $0$}\label{bound_L_r_s}
In this section, for every $0<r<s$, we construct an explicit path in $\mathcal{A}$ taking values in $B_{r,s}^{\circ}$. This gives an upper bound for $L_{r,s}$. We will use this bound to obtain Theorem \ref{tail} $\rm(i)$ and the lower bound for point $\rm(ii)$ of Theorem \ref{tail}.
\\
\\
We start with some notation. Let $u,v\in \mathcal{T}_{\infty}$ and let $t,t'\in \mathbb{R}$ be chosen in a unique way
so that $\mathcal{E}_t=u$, $\mathcal{E}_{t'}=v$ and $[t,t']$ is as small as possible (recall
our special convention for $[t,t']$ when $t>t'$). Suppose that $t\leq t'$. Recall
that $[u,v]_{\mathcal{T}_\infty}=\{\mathcal{E}_r:r\in[t,t']\}$. Let $M_{u,v}:=\inf \limits_{[u,v]_{\mathcal{T}_{\infty}}}\Lambda$ and, for every $0\leq r\leq \Lambda_{u}-M_{u,v}$, set
\[\gamma_{u,v}(r):=\Pi\Big(\mathcal{E}_{\inf\big\{r^{\prime}\in[t,t']:\:\Lambda_{r^{\prime}}=\Lambda_{u}-r\big\}}\Big)\]
and for every $\Lambda_{u}-M_{u,v}< r\leq \Lambda_{u}+\Lambda_{v}-2M_{u,v}$
\[\gamma_{u,v}(r):=\Pi\Big(\mathcal{E}_{\sup\big\{r^{\prime}\in[t,t']:\:\Lambda_{r^{\prime}}=r+2M_{u,v}-\Lambda_{u}\big\}}\Big).\]
By construction $\Delta^{\circ}(\gamma_{u,v}(r_{1}),\gamma_{u,v}(r_{2}))=|r_{1}-r_{2}|$ as soon as $r_{1},r_{2}\in (0,\Lambda_{u}-M_{u,v})$ or $r_{1},r_{2}\in (\Lambda_{u}-M_{u,v}~,~\Lambda_{u}+\Lambda_{v}-2 M_{u,v})$. In particular, applying \eqref{major}, we deduce that the restriction of $\gamma_{u,v}$ to $[0,\Lambda_{u}-M_{u,v}]$ or $[\Lambda_{u}-M_{u,v},\Lambda_{u}+\Lambda_{v}-2M_{u,v}]$ is a geodesic. Hence $\gamma_{u,v}$ is a  path with length $\Lambda_{u}+\Lambda_{v}-2\inf \limits_{[u,v]_{\mathcal{T}_{\infty}}}\Lambda$. Remark that the range of $\gamma_{u,v}$ is contained in $\Pi([u,v]_{\mathcal{T}_{\infty}})$. In particular, if $u,v\in \mathcal{R}$ (resp. $u,v\in \mathcal{L}$)  the range of $\gamma_{u,v}$ is contained in $\Pi(\mathcal{R})$ (resp. $\Pi(\mathcal{L})$) since $[u,v]_{\mathcal{T}_\infty}=\{\mathcal{E}_r:r\in[t,t']\}$. 
\\
\\
Finally, for $r<s$ set:
\[K_{r}^{s}:=\{u\in \partial \mathcal{T}_{\infty}^{s}:\:\text{there exists}\: v\in \mathcal{T}_{\infty}\:\text{such that} \:\:u\in [\![v,\infty[\![_{\mathcal{T}_{\infty}} \:\text{and}\:\: \Lambda_{v}\leq r\}.\]
The continuity of the label function $u\mapsto \Lambda_{u}$ from $\mathcal{T}_{\infty}$ into $\mathbb{R}$ and the last assumption in $(H_{1})$ imply that the set $K_{r}^{s}$ is finite $\Theta_{0}$-a.s. Recall the definition \eqref{tau} of $(\tau_{a})_{a\geq 0}$ and remark that $\tau_{s}$ is the unique element of $K_{r}^{s}$ belonging to the spine $[0,\infty)$.
Write $N_{r}^{s}:=\# K_{r}^{s}-1$ for the number of elements of $K_{r}^{s}$ not belonging to the spine.
\begin{prop}\label{L_domine_par_N}
For every $r<s$, $\Theta_{0}\text{-a.s.}$, we have:
\begin{equation}\label{L_tilde_L_1}
  L_{r,s}\leq 2 (N_{r}^{s}+1) (s-r).  
\end{equation}
\end{prop}

\begin{proof}
Denote the elements of $K_{r}^{s}$ by $u_{1},\ldots,u_{N_{r}^{s}+1}$ in such a way that $$\inf\{t\in \mathbb{R}:~\mathcal{E}_{s}=u_{i}\}<\inf\{t\in \mathbb{R}:~\mathcal{E}_{s}=u_{j}\}$$ if $i<j$. Recall that $\tau_{s}$ is the only point of the spine in $K_{r}^{s}$. Let $1\leq k\leq N_{r}^{s}+1$ be the unique index such that $u_{k}=\tau_{s}$. Fix $i\in\{1,\ldots,N_{r}^{s}+1\}$ with $i\ne k$.  Next set, for every $i\in \{1,\ldots,N_{r}^{s}+1\}$,
\[f_{i}:=\inf \{t\in \mathbb{R}:\:\mathcal{E}_{t}=u_{i}\}\:\:\text{and}\:\:\ell_{i}:=\sup \{t\in \mathbb{R}:\:\mathcal{E}_{t}=u_{i}\}.\]
Since $u_{i}$ cannot be a leaf we have $f_{i}<\ell_{i}$.  Moreover,  the sequence $(f_{i})_{1\leq i\leq n}$
is increasing. Remark that, for every $1\leq i\leq N_{r}^{s}$, we have $\ell_{i}<f_{i+1}$ and $[\ell_{i},f_{i+1}]$ is the smallest interval 
such that $\mathcal{E}_{\ell_i}=u_i$ and $\mathcal{E}_{f_{i+1}}=u_{i+1}$. Therefore we can consider 
the path $\gamma_{u_i,u_{i+1}}$ as defined before the proposition. By construction, labels of points of the
form $\mathcal{E}_t$ with $t\in [\ell_{i},f_{i+1}]$ are greater than $r$, and it follows that the length of the
path $\gamma_{u_i,u_{i+1}}$ is smaller than $2(s-r)$. Finally set
\[R=\min\big(\inf_{(-\infty,f_{1}]}\Lambda;\inf_{[\ell_{N_{r}^{s}+1},\infty)}\Lambda\big)\]
which is greater than $r$ by construction. Set $u_{0}:=p_{\infty}^{(r)}(R)$ and $u_{N_{s}^{r}+2}:=p_{\infty}^{(\ell)}(R)$. Note in 
particular that $\Pi(u_0)=\Pi(u_{N_{s}^{r}+2})$. 
 Again we can consider the paths $\gamma_{u_{0},u_{1}}$ and $\gamma_{u_{N_{r}^{s}+1},u_{N_{r}^{s}+2}}$. Let $\gamma$  be the 
 cycle obtained by concatenating the paths $(\gamma_{u_{i},u_{i+1}})_{0\leq i\leq N_{r}^{s}+1}$. It is straightforward to
 verify using property $(F)$ that $\gamma$ is an injective cycle. The paths $\gamma_{u_{0},u_{1}}$ and $\gamma_{u_{N_{r}^{s}+1},u_{N_{r}^{s}+2}}$ have length $s-R\leq s-r$ and the paths $(\gamma_{u_{i},u_{i+1}})_{1\leq i\leq N_{r}^{s}}$ have length smaller than $2(s-r)$. Hence the length
 of $\gamma$ is smaller than $2 (N_{r}^{s}+1) (s-r)$. Moreover, by
 a preceding remark, the range of $\gamma_{u_{i},u_{i+1}}$ is contained in $\Pi(\mathcal{R})$ when $i<k$ and contained in  $\Pi(\mathcal{L})$ when $i\geq k$. Consequently, if $t_0$ is the time at which $\gamma$ visits the point $\Pi(u_k)$, we have $\gamma(t)\in \Pi(\mathcal{R})$ when $t\leq t_0$
 and $\gamma(t)\in\Pi(\mathcal{L})$ when $t\geq t_0$. We conclude that the path $\gamma$ is in $\mathcal{A}$ and in particular $\gamma$
 is a separating path. Finally, by construction the path $\gamma$ visits only points $v$ such that $\Lambda_v\in(r,s]$. Since 
$\gamma(0)= \gamma_\infty(R)$ does not belong to $B^\bullet_r$, it follows that $\gamma$ takes values in $B_{s}\setminus B_{r}^{\bullet}$. But now remark that $\gamma$ hits the boundary $\partial B_{s}$ only at the times at which it visits $\Pi(u_{1}),\ldots,\Pi(u_{N_{r}^{s}+1})$. Since $N_{r}^{s}$ is, $\Theta_{0}$-a.s., finite we deduce that $\gamma$ hits the boundary $\partial B_{s}$ a finite number of times. Consequently, by an approximation procedure, for every $\epsilon>0$ we can find  $\gamma^{\prime}\in \mathcal{A}$ taking values in $\text{Int}(B_{s}^{\bullet})\setminus B_{r}^{\bullet}\subset B_{r,s}^{\circ}$ such that $\Delta(\gamma^{\prime})<\Delta(\gamma)+\epsilon$. Consequently by the definition of $L_{r,s}$ as an infimum \eqref{def:L_r,s} we deduce that 
$$  L_{r,s}\leq \Delta(\gamma)\leq 2 (N_{r}^{s}+1) (s-r).  $$
\end{proof}

\noindent The proposition shows that it is enough to control $N_{r}^{s}$ in order to get an upper bound for $L_{r,s}$. Moreover since the Brownian plane is scale invariant we have:
\[N_{r}^{s} \overset{(d)}{=} N_{1}^{\frac{s}{r}}.\]
So we consider only $N_{1}^{s}$ with $s>1$. Thanks to \eqref{eq_Laplace_Z}, the law of $N_{1}^{s}$ can be easily determined:
\begin{prop}\label{LaplaceN}
For $s>1$ and $\lambda\geq 0$ we have
\[\Theta_{0}\big(\exp(-\lambda N_{1}^{s})\big)=\Big( 1+(1-\exp(-\lambda))\frac{2s-1}{(s-1)^{2}}\Big)^{-\frac{3}{2}}\]
\end{prop}
\begin{proof}
Let $s>1$, and write $\mathfrak{L}^{s,\infty}:=\sum_{i\in I_{s}}\delta_{s_{i}, \omega^{i}}\:\:;\:\:\mathfrak{R}^{s,\infty}:=\sum_{i\in J_{s}}\delta_{s_{i}, \omega^{i}}.$
Let $\tilde{I}_{s}$ (resp. $\tilde{J}_{s}$) be  the set of all indices $i\in I_{s}$ (resp. $j\in J_{s}$)  such that $\omega_{*}^{i}< s$. For every $i\in \tilde{I}_{s}\cup \tilde{J}_{s}$,  write $(\omega^{i,k})_{k\in \mathbb{N}}$ for the excursions of $\omega^{i}$ below $s$.
By the special Markov property, conditionally on $Z_{s}$, 
$$\sum \limits_{i\in \tilde{I}_{s}\cup \tilde{J}_{s}}\sum \limits_{k\in \mathbb{N}} \delta_{\omega^{i,k}}$$
is a Poisson point measure with intensity $Z_{s} \mathbb{N}_{s}(\cdot \cap \{W_{*}>0\})$. Moreover by definition:
\[N_{1}^{s}=\#\{(i,k)\in (\tilde{I}_{s}\cup \tilde{J}_{s})\times\mathbb{N}:\: \omega_{*}^{i,k}\leq 1 \}.\]
So conditionally on $Z_{s}$,  $N_{1}^{s}$ is distributed as a Poisson variable with intensity
$Z_{s} \mathbb{N}_{s}(0<W_{*}\leq 1)$. We can then apply \eqref{etoile} to obtain:
\begin{equation}\label{0<W<1}
\mathbb{N}_{s}(0<W_{*}\leq 1)=\frac{3}{2}(\frac{1}{(s-1)^{2}}-\frac{1}{s^{2}}) = \frac{3}{2} \frac{2s-1}{s^{2}(s-1)^{2}}.
\end{equation}
Using \eqref{eq_Laplace_Z} we get that for $\lambda\geq 0$:
\begin{align*}
\Theta_{0}\big(\exp (-\lambda N_{1}^{s})\big)&= \mathbb{E}[\exp(-(1-\exp(-\lambda))Z_{s}\mathbb{N}_{s}(0< W_{*}\leq 1) )\big) \\
&= \Big(1+(1-\exp(-\lambda))\frac{2s-1}{(s-1)^{2}}\Big)^{-\frac{3}{2}}.
\end{align*}
\end{proof}

\noindent Let us list some immediate properties of $N_{1}^{s}$.

\begin{lem}\label{propN}~\\
$\rm(i)$ For every $s>1$ we have $\Theta_{0}(N_{1}^{s}=0)= (\frac{s-1}{s})^{3}$.
\\
$\rm(ii)$ The law of $(s-1)^{2}N_{1}^{s}$ under $\Theta_{0}$  converges weakly to a Gamma distribution with parameter $\frac{3}{2}$ and mean $3/2$ when $s\downarrow 1$. Furthermore, $N^s_1$ tends to $0$  as $s\to\infty$, $\Theta_{0}$-a.s.
\\
$\rm(iii)$ For every  $s>1$ and  $q<\log(\frac{s^{2}}{2s-1})$ there exists a constant $C_{q}$ such that for all $r>0$:
\[\Theta_{0}(N_{1}^{s}>r)<C_{q} \exp(-qr).\]
Furthermore:
\[\limsup \limits_{u \to \infty} \frac{\log \big(\Theta_{0}(N_{1}^{s}>u)\big)}{u} = -\log(\frac{s^{2}}{2s-1}).\]
$\rm(iv)$ For every $s>1$ we have $\Theta_{0}\big(N_{1}^{s}\big)=\frac{3}{2} \frac{2s-1}{(s-1)^{2}}$.
\end{lem}

\begin{proof}~\\
$\rm(i)$ By Proposition \ref{LaplaceN}, we have
\begin{align*}
\Theta_{0}(N_{1}^{s}=0)&=\lim \limits_{\lambda \to \infty} \Theta_{0}\big(\exp(-\lambda N_{1}^{s})\big)=\big(\frac{s-1}{s}\big)^{3}
\end{align*}
(we can also use the fact that $\Theta_{0}(N_{1}^{s}=0)=\Theta_{0}(Z_{1}^{s,\infty}=0)$ and Proposition \ref{corZabc0}).
\\
\\
$\rm(ii)$ Again using Proposition \ref{LaplaceN}, we obtain:
\[\Theta_{0}\big(\exp(-\lambda(s-1)^{2} N_{1}^{s})\big) = \Big(1+(1-\exp(-\lambda(s-1)^{2}))\frac{2s-1}{(s-1)^{2}}\Big)^{-\frac{3}{2}}.\]
When $s$ goes to 1 the Laplace transform converges to $\lambda \mapsto(1+\lambda)^{-\frac{3}{2}}$ which is the Laplace transform of a Gamma distribution with parameter $3/2$ and mean $3/2$. The fact that $N^s_1$ tends to $0$  as $s\to\infty$, $\Theta_{0}$ a.s., immediately follows
from the property $\Lambda_{\mathcal{E}_t}\longrightarrow \infty$ as $|t|\to\infty$. 
\\
\\
$\rm(iii)$ For every $\lambda>0$ we have:
\[
\Theta_{0}\Big(\exp(\lambda N_{1}^{s})\Big)=\Theta_{0}\Big(\exp\big(Z_{s}(\exp(\lambda)-1) \mathbb{N}_{s}(0< W_{*}\leq 1)\big) \Big)\] 
because, conditionally on $Z_{s}$ , the variable  $N_{1}^{s}$  is distributed as a Poisson random variable with intensity $Z_{s} \mathbb{N}_{s}(0<W_{*}\leq 1)$. But $Z_{s}$ is a Gamma random variable with parameter $3/2$ and mean $s^{2}$ so the previous expectation is finite if and only if $$(\exp(-\lambda)-1)\mathbb{N}_{s}(0<W_{*}<1)<3/(2 s^{2})$$ or equivalently $\lambda<\log(\frac{s^{2}}{2s-1})$. The first part of $\rm(iii)$ then follows from 
the Markov inequality. On the other hand if we had:
\[\limsup \limits_{u \to \infty} \frac{\log \big(\Theta_{0}(N_{1}^{s}>u)\big)}{u} \leq \alpha <-\log(\frac{s^{2}}{2s-1})\]
this would contradict $\mathbb{E}[\exp(\log(\frac{s^{2}}{2s-1})N_{1}^{s})]=\infty$. This gives the second assertion of $\rm(iii)$.
\\
\\
$\rm(iv)$ We use again the fact that, under $\Theta_{0}$, conditionally on $Z_{s}$, $N_{1}^{s}$ is distributed as a Poisson random variable with intensity $Z_{s} \mathbb{N}_{s}(0<W_{*}\leq 1)$. We have:
\begin{align*}
\Theta_{0}\big(N_{1}^{s}\big)&=\mathbb{N}_{s}(0<W_{*}\leq 1) \Theta_{0}\big(Z_{s}\big)=\frac{3}{2} \frac{2s-1}{s^{2}(s-1)^{2}} \Theta_{0}\big(Z_{s}\big)= \frac{3}{2} \frac{2s-1}{(s-1)^{2}}
\end{align*} 
where in the second equality we use  \eqref{0<W<1} and in the last equality we use the fact that $Z_{s}$ has mean $s^{2}$. 
\end{proof}
\noindent As a direct consequence we derive Theorem \ref{tail} $\rm(i)$ from Lemma \ref{propN}.
\begin{proof}[Proof of Theorem \ref{tail} \rm(i)]
By \eqref{L_tilde_L_1}, for every $s>1$ we have  $L_{1}\leq 2(s-1)(N_{1}^{s}+1)$ $\Theta_{0}$-a.s.
Let $s>1$.  Lemma \ref{propN} gives that :
\begin{align*}
\limsup \limits_{u \to \infty} \frac{\log \big(\Theta_{0}(L_{1}>u)\big)}{u} &\leq \limsup \limits_{u \to \infty} \frac{\log \big(\Theta_{0}(N_{1}^{s}>\frac{u}{2(s-1)}-1)\big)}{u}\\
&\leq -\frac{1}{2(s-1)}\log(\frac{s^{2}}{2s-1} ).
\end{align*}
Since this holds  for every $s>1$, we obtain:
\[\limsup \limits_{u \to \infty} \frac{\log \big(\Theta_{0}(L_{1}>u)\big)}{u}\leq-\sup\limits_{s>1}\frac{1}{2(s-1)}\log(\frac{s^{2}}{2s-1} ).\]
\end{proof}
The rest of this section is devoted to the proof of the lower bound appearing in Theorem \ref{tail} $\rm(ii)$. The proof relies again in \eqref{L_tilde_L_1} but in a more technical way.  We state the following slightly stronger result:
\begin{prop}\label{borneinf}
There exists a positive constant, $c_{1}$, such that for every $\epsilon\in[0,1]$ and $r>0$:
\[\Theta_{0}\big(L_{r,3r}<\epsilon r)\geq c_{1} \epsilon^{2}.\]
\end{prop}
 
\noindent The factor $3$ is arbitrary and we will see in the proof that it can be replaced by any constant greater than $1$. It will be useful in what follows to note that for every $0<r<t$,  we have $\{N_{r}^{t}=0\}=\{Z_{r}^{t,\infty}=0\}\:$ $\Theta_{0}$-a.s. 
\noindent We are going to deduce Proposition \ref{borneinf}  from \eqref{L_tilde_L_1} and the following result:
\begin{lem}\label{lemborneinf}
There exists a positive constant $c_{1}$ such that for every $r>0$ and $m\in \mathbb{N}^{*}$:
\[\Theta_{0}\Big(\:\bigcup \limits_{i=0}^{m-1} \{ N_{(m+i)r}^{(m+i+1)r}=0\}\Big) \geq \frac{c_{1}}{m^{2}}\]
\end{lem}
\begin{proof}
By the scaling invariance of $\mathcal{M}_{\infty}$ we can take  $r=1$. For $m\geq 2$, we have:
\begin{align}\label{eq:N_m}
\Theta_{0}\Big(\:\bigcup \limits_{i=0}^{m-1} \{N_{m+i}^{m+i+1}=0\}\Big)&=\Theta_{0}(N_{m}^{m+1}=0)\\
&+\sum \limits_{k=0}^{m-2} \Theta_{0}\Big(\{N_{m+k+1}^{m+k+2}=0\} \cap \bigcap \limits_{i=0}^{k} \{N_{m+i}^{m+i+1}>0\} \Big).\nonumber
\end{align}

Moreover, for every $k \in \{0,\ldots m-2\}$:
\begin{align}\label{N_m+k+1_m+k+2}
\Theta_{0}\Big(\{N_{m+k+1}^{m+k+2}=0\} \cap \bigcap \limits_{i=0}^{k} \{&N_{m+i}^{m+i+1}>0\}\Big)\\
&=\Theta_{0}\Big(\{Z_{m+k+1}^{m+k+2,\infty}=0\} \cap \bigcap \limits_{i=0}^{k} \{Z_{m+i}^{m+i+1,m+k+2}>0\}\Big) \nonumber \\
&=\Theta_{0} \big(Z_{m+k+1}^{m+k+2,\infty}=0\big)\: \Theta_{0}\Big(\bigcap \limits_{i=0}^{k} \{Z_{m+i}^{m+i+1,m+k+2}>0\}\Big)\nonumber
\end{align}
where the first equality comes from the fact that $N_{m+k+1}^{m+k+2}=0$, which is equivalent to $Z_{m+k+1}^{m+k+2,\infty}=0$, implies $Z_{t}^{m+k+2,\infty}=0$ for every $t\leq m+k+1$, so that, on the event $\{N_{m+k+1}^{m+k+2}=0\}$, we have $Z^{m+i+1,\infty}_{m+i}=Z^{m+i+1,m+k+2}_{m+i}$ for every $0\leq i\leq k$. The second equality in \eqref{N_m+k+1_m+k+2}  is a consequence of the spine independence property of $\mathcal{T}_{\infty}$ (see Section \ref{subsection_Brownian_Plane}). The idea now is to prove that for every integer $k\geq 0$:
\begin{equation}\label{Z_m_i}
\Theta_{0}\big(\bigcap \limits_{i=0}^{k} \{Z_{m+i}^{m+i+1,m+k+2}>0\}\big)\geq \prod \limits_{i=0}^{k} \Theta_{0}\big(Z_{m+i}^{m+i+1,m+k+2}>0\big).
\end{equation}
\noindent Let us explain how to obtain  this inequality. 
Let $k>0$, then:
\begin{align*}
\Theta_{0}\Big(\bigcap \limits_{i=0}^{k} \{Z_{m+i}^{m+i+1,m+k+2}>0\}\Big)&=  \Theta_{0}\Big(\bigcap \limits_{i=0}^{k-1} \{Z_{m+i}^{m+i+1,m+k+2}>0\}\Big) \\
& - \Theta_{0}\Big(\{Z_{m+k}^{m+k+1,m+k+2}=0\}\cap\bigcap \limits_{i=0}^{k-1} \{Z_{m+i}^{m+i+1,m+k+2}>0\}\Big) \\
&= \Theta_{0}\Big(\bigcap \limits_{i=0}^{k-1} \{Z_{m+i}^{m+i+1,m+k+2}>0\}\Big) \\
&- \Theta_{0}\Big(\{Z_{m+k}^{m+k+1,m+k+2}=0\}\cap\bigcap \limits_{i=0}^{k-1} \{Z_{m+i}^{m+i+1,m+k+1}>0\}\Big) 
\end{align*}
where the second equality is a consequence of the fact that $Z_{m+k}^{m+k+1,m+k+2}=0$ implies that for every $i<k$ $Z_{m+i}^{m+i+1,m+k+2}=Z_{m+i}^{m+i+1,m+k+1}$. We now can apply the spine independence property to obtain that $\Theta_{0}\Big(\bigcap \limits_{i=0}^{k} \{Z_{m+i}^{m+i+1,m+k+2}>0\}\Big)$ is equal to
\begin{align*}
\Theta_{0}\Big(\bigcap \limits_{i=0}^{k-1} \{Z_{m+i}^{m+i+1,m+k+2}>0\}\Big)-\Theta_{0}\big(Z_{m+k}^{m+k+1,m+k+2}=0\big)\Theta_{0}\Big(\bigcap \limits_{i=0}^{k-1} \{Z_{m+i}^{m+i+1,m+k+1}>0\}\Big).
\end{align*}
Now using the property $\{Z_{m+i}^{m+i+1,m+k+1}>0\}\subset \{Z_{m+i}^{m+i+1,m+k+2}>0\}$ for $i=0,\ldots,k-1$, we derive 
 \begin{align*}
 \Theta_{0}\Big(\bigcap \limits_{i=0}^{k} \{Z_{m+i}^{m+i+1,m+k+2}>0\}\Big)
 \geq \Theta_{0}\Big(\bigcap \limits_{i=0}^{k-1} \{Z_{m+i}^{m+i+1,m+k+2}>0\}\Big) \Theta_{0}\big(Z_{m+k}^{m+k+1,m+k+2}>0\big).
\end{align*}
We can then iterate this argument to obtain \eqref{Z_m_i}. By combining \eqref{N_m+k+1_m+k+2} and \eqref{Z_m_i} we deduce that:
\[\Theta_{0}\Big(\{N_{m+k+1}^{m+k+2}=0\} \cap \bigcap \limits_{i=0}^{k} \{N_{m+i}^{m+i+1}>0\}\Big)\geq \Theta_{0} \big(Z_{m+k+1}^{m+k+2,\infty}=0\big) \prod \limits_{i=0}^{k} \Theta_{0}\big(Z_{m+i}^{m+i+1,m+k+2}>0\big).\]
On the other hand, Proposition \ref{corZabc0} states that for $0<r<t<s$,
\[\Theta_{0}(Z_{r}^{t,s}=0)=\big(\frac{s}{t}\big)^{3}\big(\frac{t-r}{s-r}\big)^{3}\]
and taking the limit when $s$ goes to $\infty$, we obtain $\Theta_{0}(Z_{r}^{t,\infty}=0)=\big(\frac{t-r}{t}\big)^{3}$.
It follows that:
\begin{align*}
\Theta_{0}\big(\{N_{m+k+1}^{m+k+2}=0\} \cap \bigcap \limits_{i=0}^{k} \{&N_{m+i}^{m+i+1}>0\} \big)\geq
\frac{1}{(m+k+2)^{3}} \prod \limits_{i=0}^{k} \Big(1- \big(\frac{m+k+2}{m+i+1}\big)^{3} \frac{1}{(k+2-i)^3}\Big).
\end{align*}
Then,  for $m\geq 3$ and $k \in \{0,\ldots,m-2\}$ :
\begin{align*}
\prod \limits_{i=0}^{k} \Big(1- (\frac{m+k+2}{m+i+1})^{3} \frac{1}{(k+2-i)^3}\Big)&\geq \Big(1-\frac{1}{8}\big(\frac{m+k+2}{m+k+1}\big)^{3}\Big)\prod \limits_{i=0}^{k-1} \Big(1- (\frac{2m}{m})^{3} \frac{1}{(k+2-i)^3}\Big) \\ 
&\geq \Big(1-\frac{1}{8}\Big(\frac{5}{4}\Big)^{3}\Big)\prod \limits_{i=3}^{\infty} \big(1- \frac{8}{i^3}\big)
\end{align*}
which is a positive constant not depending on $m$. Let 
$\tilde{c}_{1}$ denote this constant. By applying the previous inequality to \eqref{eq:N_m}, we obtain that for every $m\geq 3$:
\begin{align*}
\Theta_{0}\Big(\bigcup \limits_{i=0}^{m-1} \{N_{m+i}^{m+i+1}=0\}\Big) &\geq \sum \limits_{k=0}^{m-2} \Theta_{0}\Big(\{N_{m+k+1}^{m+k+2}=0\} \cap \bigcap \limits_{i=0}^{k} \{N_{m+i}^{m+i+1}>0\} \Big)\\
&\geq\sum \limits _{k=0}^{m-2} \frac{\tilde{c}_{1}}{(m+k+2)^3}
\end{align*}
which gives us the lower bound in the lemma. 
\end{proof}
\noindent Proposition \ref{borneinf} follows now easily.
\begin{proof}[Proof of Proposition \ref{borneinf}]
By scaling we only need to prove the proposition for $r=1$. 
Let $\epsilon\in[0,1)$ and set  $s=\frac{\epsilon}{2}$ and $m=\lceil \frac{2}{\epsilon} \rceil$. Then the bound \eqref{L_tilde_L_1} and the fact that $L_{u^{\prime},v^{\prime}}\leq L_{u,v}$ if $[u^{\prime},v^{\prime}]\subset [u,v]$ give:
\[L_{1,3}\leq 2 \big(1+\min \limits_{i\in\{0,\ldots,m-1\}} N_{(m+i)s}^{(m+i+1)s} \big)s\]
$\Theta_{0}\text{-a.s.}$. Consequently
$\Theta_{0}(L_{1,3}\leq \epsilon)\geq \Theta_{0}\Big(\bigcup \limits_{i=0}^{m-1} \{ N_{(m+i)s}^{(m+i+1)s}=0\}\Big).$
 The desired result follows directly from  Lemma \ref{borneinf}. 
\end{proof}
\subsection{Upper bound for the tail of $L_{1}$ near $0$}
We are going to deduce the upper bound for Theorem \ref{tail} $\rm(ii)$ from the following result
\begin{prop}\label{exp}
For every $\delta>0$, there exists a constant $\alpha_{\delta}$ such that for every $r\geq 1$:
\[\Theta_{0}(L_{r,r+\delta}<1\:|\: Z_{r+2\delta})\leq  \exp(1-\alpha_{\delta} Z_{r+2\delta}) .\]
\end{prop}
We need to introduce some notation in order to prove Proposition \ref{exp}~. Let $u\in\mathcal{T}_{\infty}\setminus[0,\infty)$ such that $u$ is not a leaf. Set $f_{u}:=\inf\{t\in \mathbb{R}_{+}:\: \mathcal{E}_{t}=u\}$ and  $\ell_{u}:=\sup\{t\in \mathbb{R}_{+}:\: \mathcal{E}_{t}=u\}$. We consider the subtree $\mathcal{T}_{\infty}^{(u)}:=\{\mathcal{E}_{t}:\:t\in[f_{u},\ell_{u}]\}$ which is also equal to the set of all points $v\in \mathcal{T}_{\infty}$ such that $u\preceq v$. Remark that for every $v_{1},v_{2}\in\mathcal{T}_{\infty}^{(u)}$  there are two possibilities, either $[v_{1},v_{2}]_{\mathcal{T}_{\infty}}\subset \mathcal{T}_{\infty}^{(u)}$ and in this case $0\notin [v_{1},v_{2}]_{\mathcal{T}_{\infty}}$ or $0\in [v_{1},v_{2}]_{\mathcal{T}_{\infty}}$. Consequently $\Delta^{\circ}(v_{1},v_{2})$  only depends on the subtree $\big(\mathcal{T}_{\infty}^{(u)}, (\Lambda_{v})_{v\in \mathcal{T}_{\infty}^{(u)}}\big)$. For every $v,w\in \mathcal{T}_{\infty}^{(u)}$, set:
\begin{equation}\label{definition_tilde_Delta}
    \widetilde{\Delta}_{u}(v,w):=\mathop{\inf \limits_{v=v_{1},\ldots,v_{n}=w}}_{v_{1},\ldots,v_{n}\in \mathcal{T}_{\infty}^{(u)}}\sum\limits_{i=1}^{n-1} \Delta^{\circ}(v_{i},v_{i+1})
\end{equation}
 where the infimum is over all choices of the integer $n\geq 1$ and all the finite sequences $u_{0},\ldots,u_{n}$ of elements of $\mathcal{T}_{\infty}^{(u)}$ verifying $v_{1}=v$ and $v_{n}=w$. Remark that $\widetilde{\Delta}_{u}$ defines a continuous pseudo-distance on $\mathcal{T}^{(u)}_{\infty}$ (since $v\mapsto \Lambda_{v}$ is continuous)  and that $\Delta(v,w)\leq \widetilde{\Delta}_{u}(v,w)$ for every $v,w\in\mathcal{T}^{(u)}_{\infty}$. By the previous remark and since by $(F)$, for every $v,w\in\mathcal{T}^{(u)}_{\infty}$, we have $\Delta(v,w)=0$ if and only if $\Delta^{\circ}(v,w)=0$, we see that $\widetilde{\Delta}_{u}$ defines a distance on $\Pi(\mathcal{T}^{(u)}_{\infty})$ (and we keep the notation  $\widetilde{\Delta}_{u}$ for this distance).  To simplify notation, we introduce the set $A_{u}:=\Pi(\mathcal{T}_{\infty}^{(u)})$ and  the paths $\gamma^{(1)}_{u}$ and $\gamma^{(2)}_{u}$ defined as follows. For every $0\leq t\leq\Lambda_{u}- \min \limits_{v\in \mathcal{T}_{\infty}^{(u)}}\Lambda_v$, take:
\[\gamma^{(1)}_{u}(t):=\Pi\big(\inf\{r\in [f_{u},\ell_{u}]:\: \Lambda_{r}=\Lambda_{u}-t\}\big)\]
and 
\[\gamma^{(2)}_{u}(t):=\Pi\big(\sup\{r\in [f_{u},\ell_{u}]:\: \Lambda_{r}=\min \limits_{v\in \mathcal{T}_{\infty}^{(u)}}\Lambda_v+t\}\big).\]
By construction and \eqref{major}, $\gamma^{(1)}_{u}$ and $\gamma^{(2)}_{u}$ are two geodesic paths. We also observe that 
$\min_{v\in \mathcal{T}_{\infty}^{(u)}}\Lambda_v<\Lambda_u$ as a consequence of \cite[Lemma 3.2]{GallPaulin} (it is important
to notice that this holds simultaneously for all $u\in\mathcal{T}_{\infty}\setminus[0,\infty)$ such that $u$ is not a leaf).
Moreover, we have $\gamma^{(1)}_{u}\big(\Lambda_{u}- \min_{v\in \mathcal{T}_{\infty}^{(u)}}\Lambda_v\big)=\gamma^{(2)}_{u}(0)$ and $\gamma^{(1)}_{u}(0)=\gamma^{(2)}_{u}\big(\Lambda_{u}- \min_{v\in \mathcal{T}_{\infty}^{(u)}}\Lambda_v\big)$. From property $(F)$, we get that the concatenation of $\gamma^{(1)}_{u}$ and $\gamma^{(2)}_{u}$ is an injective cycle with length $2\Lambda_{u}-2 \min_{v\in \mathcal{T}_{\infty}^{(u)}}\Lambda_v$. We denote this path by $\gamma_{u}$. It will be useful to remark that $\gamma_{u}^{1}$ and $\gamma_{u}^{2}$ are also geodesic paths for $\widetilde{\Delta}_{u}$.  Recall the notation $\Delta_{A_{u}}$ for the intrinsic distance induced by $\Delta$ on $A_{u}$. 
\begin{lem}\label{lem_A_u}
$\Theta_{0}$-a.s. for every $u\in\mathcal{T}_{\infty}\setminus[0,\infty)$ such that $u$ is not a leaf, the set $A_{u}$ is homeomorphic to the closed unit disk of $\mathbb{C}$ and its boundary is the range of $\gamma_{u}$. Moreover we have:
\begin{equation}\label{Delta_A_u}
\Delta_{A_{u}}=\widetilde{\Delta}_{u}.    
\end{equation}
\end{lem}
The main interest of \eqref{Delta_A_u} is the fact that the function $\widetilde{\Delta}_{u}$ only depends on $\big(\mathcal{T}_{\infty}^{(u)}, (\Lambda_{v})_{v\in \mathcal{T}_{\infty}^{(u)}}\big)$. Recall the definition of $\Delta^{\circ}$  above \eqref{Def-Delta}, and remark that $\Delta^{\circ}$ on $\mathcal{T}_{\infty}^{(u)}$ does not change if we shift all the labels by any $a> -\min_{\mathcal{T}_{\infty}^{(u)}}\Lambda$. This gives  that  $\widetilde{\Delta}_{u}$ can also be defined from the labeled tree $\big(\mathcal{T}_{\infty}^{(u)}, (\Lambda_{v})_{v\in \mathcal{T}_{\infty}^{(u)}}+a\big)$ for any  $a> -\min_{\mathcal{T}_{\infty}^{(u)}}\Lambda$.
\begin{proof}
Let $u\in\mathcal{T}_{\infty}\setminus[0,\infty)$ such that $u$ is not a leaf. Since $\gamma_{u}$ is an injective cycle, Jordan's theorem implies that the complement of the range of $\gamma_{u}$ has two connected components, namely a bounded connected component $U_{1}$ and an unbounded connected component $U_{2}$. Moreover, the closure of $U_{1}$ is homeomorphic to the closed unit disk. 
The first assertion of the lemma then follows from the fact that $\text{Cl}(U_1)=A_u$, which is easy and left to the reader. Let us turn to the second part of the lemma. We start by showing that $\Delta_{A_{u}}(v,w)\leq \widetilde{\Delta}_{u}(v,w)$ for every $v,w\in \mathcal{T}_{\infty}^{(u)}$. Let $v,w\in \mathcal{T}_{\infty}^{(u)}$. Up to interchanging $v$ and $w$, we can suppose that 
\[\Delta^{\circ}(v,w)=\Lambda_{v}+\Lambda_{w}-2\min \limits_{[v,w]_{\mathcal{T}_{\infty}}}\Lambda,\]
and $v=\mathcal{E}_s$, $w=\mathcal{E}_t$ with $s\leq t$ and $[u,w]_{\mathcal{T}_{\infty}}=\{\mathcal{E}_r:s\leq r\leq t\}$.
We can then consider the path $\gamma_{v,w}$ introduce at the beginning of Section \ref{bound_L_r_s}. The length of $\gamma_{v,w}$ is $\Delta^{\circ}(v,w)$ and its range is a subset of $\Pi([v,w]_{\mathcal{T}_{\infty}})$. Since $[v,w]_{\mathcal{T}_{\infty}}\subset \mathcal{T}_{\infty}^{(u)}$, we deduce that $\gamma_{v,w}$ takes values in $A_{u}$. From the definition \eqref{definition_tilde_Delta} we obtain that $\Delta_{A_{u}}(v,w)\leq \widetilde{\Delta}_{u}(v,w)$ for every $v,w\in \mathcal{T}_{\infty}^{(u)}$.
\\
\\
Let us prove the reverse inequality. Let $\gamma:[0,1]\to A_{u}$ be a path. It is enough to show that $\Delta(\gamma)\geq \widetilde{\Delta}_{u}(\gamma(0),\gamma(1))$. We deal with two separate cases.\\
\\
$\bullet$ Case 1: We assume that, for every $t\in(0,1)$, $\gamma(t)\notin \partial A_{u}$. Let us start by showing that $\gamma$ is also continuous for $\widetilde{\Delta}_{u}$.
 In order to prove this, remark that the identity function $(A_{u},\widetilde{\Delta}_{u})\mapsto  (A_{u},\Delta)$ is a bijection and that it is also continuous, since $\Delta\leq \widetilde{\Delta}_{u}$. Moreover, as $\mathcal{T}_{\infty}^{(u)}$ is compact, the continuity of $\Pi$ implies that $A_{u}$ is  compact for the quotient topology. Since $\Delta$ and $\widetilde{\Delta}_{u}$ are continuous on  $\mathcal{T}_{\infty}^{(u)}\times \mathcal{T}_{\infty}^{(u)}$ we derive that $(A_{u},\widetilde{\Delta}_{u})$ and $(A_{u},\Delta)$ are both compact. So  the identity function $(A_{u},\widetilde{\Delta}_{u})\mapsto  (A_{u},\Delta)$ is a continuous bijection between compact spaces which implies that it is also an homeomorphism. We deduce that $\gamma$ is continuous for $\widetilde{\Delta}_{u}$. In particular, we have
$$\Delta(\gamma)=\lim \limits_{\epsilon \to 0} \Delta(\gamma|_{[\epsilon,1-\epsilon]})\:\:\text{and}\:\:\widetilde{\Delta}_{u}\big(\gamma(0),\gamma(1)\big)= \lim \limits_{\epsilon \to 0}\widetilde{\Delta}_{u}\big(\gamma(\epsilon),\gamma(1-\epsilon)\big).$$
By the previous display, we may and will restrict our attention to the case when $\gamma(t)\notin A_{u}$ for every $t\in [0,1]$. By compactness, the quantity $\delta:=\inf \limits_{(t,v)\in [0,1]\times \partial A_{u}} \Delta(\gamma(t),v)$ is positive. Let $n>1$ be an integer such that for every $0\leq i\leq n-1$:
\[\Delta\big(\gamma(\frac{i}{n}),\gamma(\frac{i+1}{n})\big)<\frac{\delta}{3}.\]
We are going to show that $\Delta\big(\gamma(\frac{i}{n}),\gamma(\frac{i+1}{n})\big)\geq \widetilde{\Delta}_{u}\big(\gamma(\frac{i}{n}),\gamma(\frac{i+1}{n})\big)$ for every $0\leq i\leq n-1$. By the triangle inequality this will imply that $\Delta(\gamma)\geq \widetilde{\Delta}_{u}(\gamma(0),\gamma(1))$. Fix $0\leq i\leq n-1$ and consider $u_{i},v_{i}\in \mathcal{T}_{\infty}$ such that $\big(\Pi(u_{i}),\Pi(v_{i})\big)=\big(\gamma(\frac{i}{n}),\gamma(\frac{i+1}{n})\big)$. Remark that  we must have $u_{i},v_{i}\in \mathcal{T}_{\infty}^{(u)}$ and recall that:
\[\Delta(u_{i},v_{i})=\inf \limits_{u_{i}=u_{i,1},u_{i,2}\ldots,u_{i,m}=v_i} \sum \limits_{k=1}^{m-1}\Delta^{\circ}(u_{i,k},u_{i,k+1})\]
where the infimum is over all choices of the integer  $m\geq 1$ and all finite sequence $u_{i,1},\ldots,u_{i,m}\in \mathcal{T}_{\infty}$ with $(u_{i,1},u_{i,m})=(u_{i},v_{i})$. Since $\Delta(u_{i},v_{i})<\delta/3$ we can restrict the previous infimum to finite sequence $u_{i,1},\ldots,u_{i,m}\in \mathcal{T}_{\infty}$ with $(u_{i,1},u_{i,m})=(u_{i},v_{i})$ such that:
\[\sum \limits_{k=1}^{m-1}\Delta^{\circ}(u_{i,k},u_{i,k+1})<\frac{\delta}{2}.\]
Consider  such a sequence $u_{i,1},\ldots,u_{i,m}\in \mathcal{T}_{\infty}$ and remark that $\Delta(u_{i,1},u_{i,k})<\frac{\delta}{2}$ for every $1\leq k\leq m$ by triangle inequality. This implies by the definition of $\delta$ that $u_{i,k}\in \mathcal{T}_{\infty}^{(u)}$ for every $1\leq k\leq m$. We conclude
from the definition of $\widetilde\Delta_u$ that:
\[\sum \limits_{k=1}^{m-1}\Delta^{\circ}(u_{i,k},u_{i,k+1})\geq \widetilde{\Delta}_{u}(u_{i},v_{i}).\]
Consequently, $\Delta\big(\gamma(\frac{i}{n}),\gamma(\frac{i+1}{n})\big)\geq \widetilde{\Delta}_{u}\big(\gamma(\frac{i}{n}),\gamma(\frac{i+1}{n})\big)$ for every $0\leq i\leq n-1$ and thus $\Delta(\gamma)\geq \widetilde{\Delta}_{u}\big(\gamma(0),\gamma(1)\big)$. 
\\
\\
$\bullet$ Case 2: We now assume that $\gamma$ hits $\partial A_u$. Let $s:=\inf \{r\in[0,1]:\:\gamma(r)\in \partial A_{u}\}$. 
Without loss of generality, we may assume that $\gamma(s)$ belongs to the range of $\gamma_{u}^{(1)}$. Let $s^{\prime}$ be the largest element of $[0,1]$ such that $\gamma(s^{\prime})$ is in the range of $\gamma_{u}^{(1)}$.  If $s^{\prime}=1$ or if $\gamma(t)\notin \partial A_{u}$ for every $s^{\prime}<t\leq 1$, we have:
\[\Delta(\gamma)=\Delta(\gamma|_{[0,s]})+\Delta(\gamma|_{[s,s^{\prime}]})+\Delta(\gamma|_{[s^{\prime},1]})\geq \widetilde{\Delta}_{u}(\gamma(0),\gamma(s))+\Delta(\gamma|_{[s,s^{\prime}]})+\widetilde{\Delta}_{u}(\gamma(s^{\prime}),\gamma(1))\]
since case $1$ can be applied to $\gamma|_{[0,s]}$ and $\gamma|_{[s^{\prime},1]}$. Moreover, since $\gamma_{u}^{(1)}$ is also a geodesic for $\widetilde{\Delta}_{u}$, we have  $\Delta(\gamma|_{[s,s^{\prime}]})\geq \Delta(\gamma(s),\gamma(s'))=\widetilde{\Delta}_{u}(\gamma(s),\gamma(s^{\prime}))$ and by the triangle inequality we obtain $\Delta(\gamma)\geq \widetilde{\Delta}_{u}(\gamma(0),\gamma(1))$. 
\\
\\
It remains to consider the case where $s^{\prime}<1$ and $\{t\in(s^{\prime},1]:~\gamma(t)\in\partial A_{u}\}$ is not empty. Let $t$ be the smallest element of $[s^{\prime},1]$ such that $\gamma(t)\in \partial A_{u}$. Then  $\gamma(t)$ belongs to the range of $\gamma_{u}^{(2)}$. Let $t^{\prime}$ 
be the largest element of $[t,1]$ such that $\gamma(t^{\prime})$ belongs to the range of $\gamma_{u}^{(2)}$. Then we get that:
\[\Delta(\gamma)=\Delta(\gamma|_{[0,s]})+\Delta(\gamma|_{[s,s^{\prime}]})+\Delta(\gamma|_{[s^{\prime},t]})+\Delta(\gamma|_{[t,t^{\prime}]})+\Delta(\gamma|_{[t^{\prime},1]}).\]
Since $\gamma_{u}^{(1)}$ and $\gamma_{u}^{(2)}$ are two geodesics paths for $\widetilde{\Delta}_{u}$ we have:
$$\Delta(\gamma|_{[s,s^{\prime}]})\geq \Delta\big(\gamma(s),\gamma(s^{\prime})\big)=\widetilde{\Delta}_{u}\big(\gamma(s),\gamma(s^{\prime})\big)$$
and
$$\Delta(\gamma|_{[t,t^{\prime}]})\geq  \Delta\big(\gamma(t),\gamma(t^{\prime})\big)=\widetilde{\Delta}_{u}\big(\gamma(t),\gamma(t^{\prime})\big) .$$
On the other hand,
$\gamma|_{[0,s]}$, $\gamma|_{[s^{\prime},t]}$ and $\gamma|_{[t^{\prime},1]}$ belong to case $1$. Consequently, we obtain $\Delta(\gamma)\geq \widetilde{\Delta}_{u}(\gamma(0),\gamma(1))$.
\end{proof}
Let us deduce Proposition \ref{exp} from Lemma \ref{lem_A_u}. 
\\
\\
\noindent{\it Proof of Proposition \ref{exp}.} 
Fix $\delta>0$ and $r\geq 1$. We want to give a lower bound for $L_{r,r+\delta}$. By Lemma \ref{sepa_A} it is enough to give a lower bound for $\Delta(\gamma)$ for every $\gamma\in \mathcal{A}$ taking values in $B_{r,r+\delta}^{\circ}$. Recall the notation of Section \ref{bound_L_r_s} and for every $u\in K_{r}^{r+2\delta}$ not belonging to the spine $[0,\infty)$ set:
$$M_{u}:=\min\limits_{v\in \mathcal{T}_{\infty}^{(u)}} \Lambda_{v}.$$
Let $u_{1},\ldots,u_{\tilde{N}_{r}^{r+2\delta}}$ be the elements of $K_{r}^{r+2\delta}$ that do not belong to the spine $[0,\infty)$ and such that $M_{u}>r-1/2$. By Lemma \ref{lem_A_u}, for every $1\leq k\leq  \tilde{N}_{r}^{r+2\delta}$ the set $A_{u_{k}}$ is homeomorphic to the closed unit disk. Since $\partial B_{r}^{\bullet}$ and $\partial B_{r+\delta}^{\bullet}$ are injective cycles (see  Section \ref{subsection_Brownian_Plane}), it is straightforward  to verify (using Jordan's theorem) that $A_{u_{k}}\cap B_{r,r+\delta}^{\bullet}$ is homeomorphic to the unit disk and its boundary is:
\[\big(A_{u_{k}}\cap \partial B_{r,r+\delta}^{\bullet}\big)\cup \gamma_{u_{k}}^{(1)}([\delta, 2\delta])\cup\gamma_{u_{k}}^{(2)}([r-M_{u}, r+\delta-M_{u}]).\]
This implies that any separating cycle $\gamma\in \mathcal{A}$ taking values in $B_{r,r+\delta}^{\circ}$ has to connect $\gamma_{u_{k}}^{(1)}([\delta, 2\delta])$ and $\gamma_{u_{k}}^{(2)}([r-M_{u}, r+\delta-M_{u}])$ while staying in $A_{u_{k}}$ see figure \ref{fig:borne} for an illustration. In other words, for every $1\leq k\leq \tilde{N}_{r}^{r+2\delta}$, there exist $t_{k}<t_{k}^{\prime}$ such that $\gamma|_{ [t_{k},t_{k}^{\prime}]}$ takes values in $A_{u_{k}}\cap B_{r,r+\delta}^{\circ}$, $\gamma(t_{k})\in  \gamma_{u_{k}}^{(1)}([\delta, 2\delta])$  and $\gamma(t_{k}^{\prime})\in \gamma_{u_{k}}^{(2)}([r-M_{u}, r+\delta-M_{u}])$. In particular:
\begin{align*}
\Delta(\gamma)&\geq \sum \limits_{k=1}^{\tilde{N}_{r}^{r+2\delta}}\Delta(\gamma|_{[t_{k},t_{k}^{\prime}]})\geq \sum \limits_{k=1}^{\tilde{N}_{r}^{r+2\delta}}\Delta_{A_{u_{k}}}(\gamma(t_{k}),\gamma(t_{k}^{\prime})).
\end{align*}
Set $$D_{k}:=\inf\big\{\Delta_{A_{u_{k}}}(x,y):\:(x,y)\in \gamma_{u_{k}}^{(1)}([\delta, 2\delta])\times\gamma_{u_{k}}^{(2)}([r-M_{u}, r+\delta-M_{u}]) \big\},$$
and note that $D_k>0$ by property $(F)$ (no point of $\gamma_{u_{k}}^{(1)}([\delta, 2\delta])$ can be identified with a point
of $\gamma_{u_{k}}^{(2)}([r-M_{u}, r+\delta-M_{u}])$). 
With this notation we have  $\Delta(\gamma)\geq \sum_{k=1}^{\tilde{N}_{r}^{r+2\delta}} D_{k}$ for every $\gamma\in \mathcal{A}$ taking values in $B_{r,r+\delta}^{\circ}$. Consequently, we obtain that:
\[L_{r,r+\delta}\geq \sum \limits_{k=1}^{\tilde{N}_{r}^{r+2\delta}} D_{k}.\]
\begin{figure}[!h]
 \begin{center}
     \includegraphics[height=4cm,width=7cm]{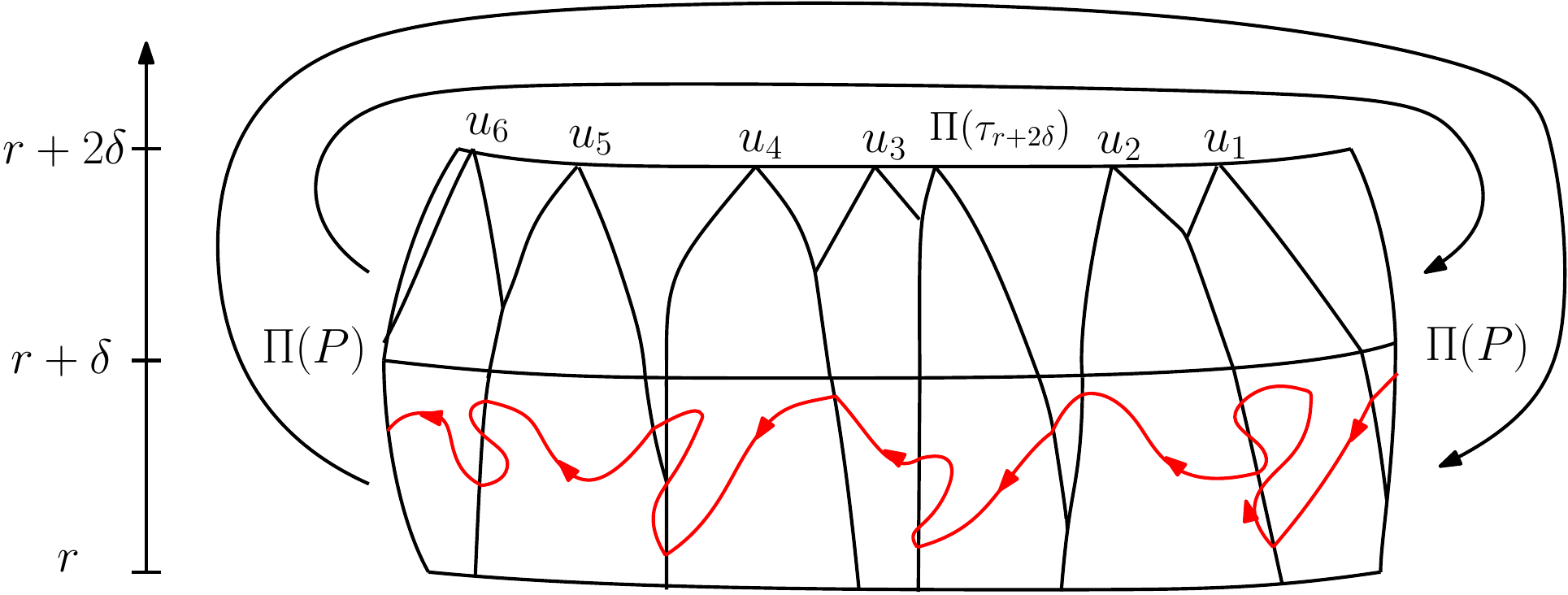}\:\:\:\:\:\:\:\:
  \includegraphics[height=2cm,width=5cm]{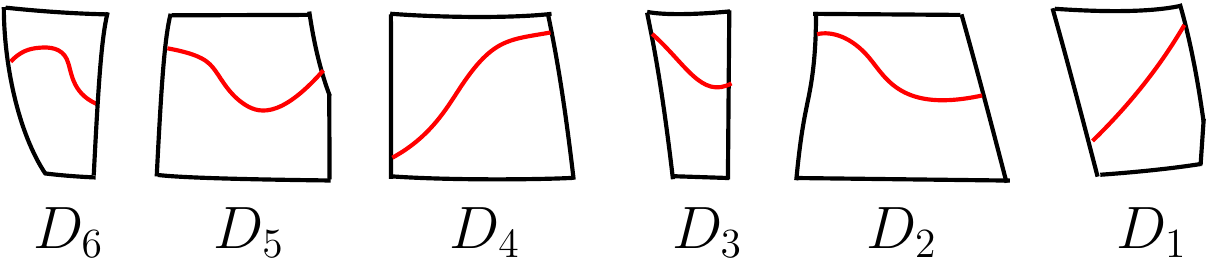}
 \caption{\label{fig:borne}Illustration of the inequality $L_{r,r+\delta}\geq \sum \limits_{k=1}^{\tilde{N}_{r}^{r+2\delta}} D_{k}$. The red path is an element of $\mathcal{A}$ taking values in $B_{r,r+\delta}^{\circ}$ and here $\tilde{N}_{r}^{r+2\delta}=6$.}
 \end{center}
 \end{figure}
 
To conclude we use the following claim:
\\
\\
$(C)$: Conditionally on $(Z_{r+2\delta},\tilde{N}_{r}^{r+2\delta})$, the variables $(D_{k})_{1\leq k\leq \tilde{N}_{r}^{r+2\delta}}$ are independent and identically distributed according to a distribution $\mu_\delta$ that does not depend on $r$ and is supported on $(0,\infty)$.
\\
\\
Before proving $(C)$, let us explain why  Proposition \ref{exp} follows. \\
We set $\chi_\delta:=\int \mu_\delta(dx)\,e^{-x}\in (0,1)$, and  then we have
\begin{align*}
\Theta_{0}(L_{r,r+\delta}<1\:|\: Z_{r+2\delta})&\leq \Theta_{0}\big(\sum\limits_{k=1}^{\tilde{N}_{r}^{r+2\delta}}D_{k} <1\:|\: Z_{r+2\delta}\big) \\
&\leq e\: \Theta_{0}\big(\exp(-\sum_{k=1}^{\tilde{N}_{r}^{r+2\delta}}D_{k})\:|\: Z_{r+2\delta} \big)\\
&= e\: \Theta\big(\chi_\delta^{\tilde{N}_{r}^{r+2\delta}}\:|\: Z_{r+2\delta} \big)
\end{align*}
by our claim $(C)$. By the special Markov property, conditionally on $Z_{r+2\delta}$, the variable  $\tilde{N}_{r}^{r+2\delta}$ is distributed as a Poisson variable with parameter $Z_{r+2\delta} \mathbb{N}_{r+2\delta}(r -1/2<W_{*}\leq r)$. It follows that
\[\Theta_{0}(L_{r,r+\delta}<1\:|\: Z_{r+2\delta})\leq e \exp\big(-Z_{r+2\delta} \mathbb{N}_{r+2\delta}(r -1/2<W_{*}\leq r)(1-\chi_\delta)\big)\]
and we obtain the desired result with 
\[\alpha_{\delta}:=\mathbb{N}_{r+2\delta}(r -1/2<W_{*}\leq r)(1-\chi)=\mathbb{N}_{2\delta}( 1/2<W_{*}\leq 1)(1-\chi_\delta).\]
Let us explain why $(C)$ holds in order to complete the proof.
Let $(\omega^{i})_{i\in I_{r+2\delta}\cup J_{r+2\delta}}$ be the atoms of $\mathfrak{L}^{r+2\delta,\infty}$ and $\mathfrak{R}^{r+2\delta,\infty}$. Let $ \tilde{I}_{r+2\delta}\cup \tilde{J}_{r+2\delta}$ be the set of indices $i\in I_{r+2\delta}\cup J_{r+2\delta}$ such that $\omega^{i}_{*}<r+2\delta$. For every $i\in\tilde{I}_{r+2\delta}\cup \tilde{J}_{r+2\delta}$ we write $(\omega^{i,n})_{n\in \mathbb{N}}$ for the excursions of $\omega^{i}$ outside $r+2\delta$. By construction:
\[\tilde{N}_{r}^{r+2\delta}:=\#\big\{(i,n)\in (\tilde{I}_{r+2\delta}\cup\tilde{J}_{r+2\delta})\times \mathbb{N}:\: r-1/2<\omega^{i,n}_{*}<r\big\}.\]
For every $1\leq k\leq \tilde{N}_{r}^{r+2\delta}$, there exists a unique $(i,n)$  with $r-1/2<\omega^{i,n}_{*}<r$ such that the labeled trees $\mathcal{T}_{\omega^{i,n}}$ and $\mathcal{T}_{\infty}^{(u_{k})}$ can be identified and we write $\omega^{k}$ instead of
$\omega^{i,n}$ to simplify notation. By the special Markov property, conditionally on $Z_{r+2\delta}$, the point measure:
\[\sum \limits_{k=1}^{\tilde{N}_{r}^{r+2\delta}}\delta_{\omega^{k}}\]
is a Poisson measure with intensity:
$$Z_{r+2\delta}\mathbb{N}_{r+2\delta}\big(\cdot \cap \{r-\frac{1}{2}<W_{*}<r\}\big).$$
So conditionally on $(Z_{r+2\delta},\tilde{N}_{r}^{r+2\delta})$, the sequence $(\omega^{k}-r+1)_{1\leq k\leq \tilde{N}_{r}^{r+2\delta}}$ is an i.i.d. sequence with common distribution $\mathbb{N}_{\delta}(\cdot\cap \{\frac{1}{2}<W_{*}<1\})$. In particular, this distribution does not depend on $r$. Moreover $\tilde{\Delta}_{u_{k}}$ depends only on the labeled tree $\mathcal{T}_{\infty}^{(u_{k})}=\mathcal{T}_{\omega^k}$ and the definition
\eqref{definition_tilde_Delta} shows that $\tilde{\Delta}_{u_{k}}$ is not affected if labels are shifted by $(-r+1)$. So $\tilde{\Delta}_{u_{k}}$ is also a function of $\omega^{k}-r+1$. Our claim $(C)$ follows since by Lemma \ref{lem_A_u}, we have $\Delta_{A_{u_{k}}}=\tilde{\Delta}_{u_{k}}$ for every $1\leq k\leq \tilde{N}_{r}^{r+2\delta}$. \hfill$\square$
\\
\\
We conclude this section with the proof of part $\rm(ii)$ of Theorem \ref{tail}.
\\
\\
\noindent{\it Proof of Theorem \ref{tail} $\rm(ii)$.} We want to show that
there exists  $c_{2}$, such that for every $\epsilon\geq 0$:
\[\Theta_{0}(L_{1}<\epsilon)\leq c_{2} \epsilon^{2} .\]
To do so fix $\epsilon\in(0,1/2)$ and remark that:
\begin{equation}\label{L_1_ep}
\{L_{1}<\epsilon\}\subset \bigcup \limits_{m=\lfloor \frac{1}{\epsilon} \rfloor-1}^{\infty}\{L_{m\epsilon,(m+3)\epsilon}<\epsilon\}\
\end{equation}
Let us explain why \eqref{L_1_ep} holds.  On the event $\{L_{1}<\epsilon\}$, let $\gamma$ be a separating cycle in $\check{B}_{1}^{\circ}$ such that $\Delta(\gamma)<\epsilon$. Since the sets $B^\bullet_{(m+1)\epsilon}\backslash B^\bullet_{m\epsilon}$, for $m\geq \lfloor \frac{1}{\epsilon}\rfloor$, cover 
$\check{B}_{1}^{\circ}$, we can find $m_0\geq \lfloor \frac{1}{\epsilon}\rfloor$ such that $\gamma(0)\in B^\bullet_{(m_0+1)\epsilon}\backslash B^\bullet_{m_0\epsilon}$. Then notice that $\Delta(\gamma(0),B_{(m_{0}-1)\epsilon}^{\bullet})\geq\epsilon$ and
$\Delta(\gamma(0),\check{B}_{(m_{0}+2)\epsilon}^{\bullet})\geq\epsilon$. Since the length of $\gamma$ is smaller than $\epsilon$, it follows that
the path $\gamma$ stays inside $B_{(m_{0}-1)\epsilon,(m_{0}+2)\epsilon}^{\circ}$, and consequently $\Delta(\gamma)\leq L_{(m_{0}-1)\epsilon, (m_{0}+2)\epsilon}$.
\\
\\
\eqref{L_1_ep} implies that:
\begin{align*}
\Theta_{0}(L_{1}<\epsilon)&\leq \sum \limits_{m=\lfloor \frac{1}{\epsilon} \rfloor-1}^{\infty} \Theta_{0}\big(L_{m\epsilon, (m+3)\epsilon}<\epsilon\big) =\sum \limits_{m=\lfloor \frac{1}{\epsilon} \rfloor-1}^{\infty} \Theta_{0}\big(L_{m,m+3}<1\big)
\end{align*}
where to obtain the right equality we use the scale invariance of $\mathcal{M}_{\infty}$. We now can apply Proposition \ref{exp} to obtain that there exists $\alpha>0$ such that :
\begin{align*}
\Theta_{0}(L_{1}<\epsilon)    
\leq e\sum \limits_{m=\lfloor \frac{1}{\epsilon} \rfloor-1}^{\infty} \Theta_{0}\big(\exp(-\alpha Z_{m+4})\big)=  e\sum \limits_{m=\lfloor \frac{1}{\epsilon} \rfloor-1}^{\infty} (1+\frac{2}{3}\alpha (m+4)^{2})^{-\frac{3}{2}}
\end{align*}
where we used \eqref{eq_Laplace_Z} in the last equality. The desired result follows.
\hfill$\square$
\subsection{Application to the infinite volume Brownian disk}\label{subsect_application}
The goal of this section is to extend Theorem \ref{spatial_fix_level} to random levels  and then to derive some properties of injective cycles separating the boundary from infinity in the infinite volume Brownian disk. Let us recall the notation  of Subsection \ref{IBD}, and in particular, the definition
of the coding triple $(\widetilde{X}^{(r)},\widetilde{\mathfrak{L}}_{r},\widetilde{\mathfrak{R}}_{r})$ for every $r\geq 0$. On the 
canonical space $\mathcal{C}(\mathbb{R}_{+},\mathbb{R})\times M(\mathcal{S})\times M(\mathcal{S})$, for every $r\geq 0$, let $\mathcal{F}_{r}$ be the $\sigma$-field generated by $B_{r}^{\bullet}$ (view as a random variable with values in $\mathbb{K}$ as explained
in Section \ref{subsub-Plane}) and the class of all $\Theta_{0}$-negligible sets. The approximation property \eqref{geoZ} implies that $Z_{r}$ is $\mathcal{F}_{r+}$-measurable, for every $r\geq 0$. We write $\rho_{r}:=\Pi(\tau_{r})$ for every $r\geq 0$, where $(\tau_{r})_{r\geq 0}$ is defined in \eqref{tau}.

\begin{theo}\label{spatial_random_level}
Let $T$ be a stopping time of the filtration $(\mathcal{F}_{r+})_{r\geq 0}$ such that  we have $0<T<\infty$, $\Theta_{0}$-a.s. Then conditionally on $Z_{T}=z$, the coding triple $(\widetilde{X}^{(T)},\widetilde{\mathfrak{L}}_{T},\widetilde{\mathfrak{R}}_{T})$ is distributed according to $\Theta_{z}$ and  is independent of $B_{T}^{\bullet}$.
Furthermore,
the intrinsic distance $\check{\Delta}^{(T)}$ on $\check{B}^\circ_T$ has a unique continuous extension to $\check{B}_{T}^{\bullet}$. The space $\check{B}_{T}^{\bullet}$ equipped with this continuous extension of $\check{\Delta}^{(T)}$, with the restriction of the volume measure and with the distinguished point $\rho_T$ coincides
(as an element of $\mathbb{K}_{\infty}$) with the metric space associated with $(\widetilde{X}^{(T)},\widetilde{\mathfrak{L}}_{T},\widetilde{\mathfrak{R}}_{T})$. In particular, conditionally on  $Z_{T}=z$, the space  $\check{B}_{T}^{\bullet}$ is an infinite Brownian disk with perimeter $z$ and is independent of $B_{T}^{\bullet}$. 
\end{theo}

\begin{proof}
Let $T$ be as in the statement of the theorem. 
Recall the notation $M(\mathcal{S})$ and the distance $d_{M(\mathcal{S})}$ introduced in Section \ref{subsect_coding_triple}. 
Let $F_1$ and $F_2$ be two bounded nonnegative measurable functions on the canonical space $\mathcal{C}(\mathbb{R}_{+},\mathbb{R})\times M(\mathcal{S})\times M(\mathcal{S})$. Assume that $F_1$ is $\mathcal{F}_{T+}$-measurable and that $F_2$
is continuous. We will show that
\begin{equation}\label{eq:indep}
\Theta_{0}\big(F_{1}\times F_{2}\big(X^{(T)},\widetilde{\mathfrak{L}}_{T},\widetilde{\mathfrak{R}}_{T}\big)\big)=\Theta_{0}\big(F_{1}\,\Theta_{Z_{T}}\big(F_{2}\big)\big).
\end{equation}
Remark that \eqref{eq:indep} implies that, conditionally on $Z_{T}=z$, the coding triple $(\widetilde{X}^{(T)},\widetilde{\mathfrak{L}}_{T},\widetilde{\mathfrak{R}}_{T})$ is distributed according to $\Theta_{z}$ and  is independent of $B_{T}^{\bullet}$ (the hull 
$B_{T}^{\bullet}$ is $\mathcal{F}_{T+}$ measurable, since the process $t\mapsto B^\bullet_t$ is adapted to $(\mathcal{F}_{t+})_{t\geq 0}$
and $T$ is a stopping time). In particular,  $(\widetilde{X}^{(T)},\widetilde{\mathfrak{L}}_{T},\widetilde{\mathfrak{R}}_{T})$ will a.s. verify $(H_{2})$. Then the different assertions of the theorem follow from Lemma \ref{lem:H_2}. It remains to establish \eqref{eq:indep}. For every integer $n\geq1$ we have:
\begin{equation}\label{eq:indep2}\Theta_{0}\big(F_{1}\times F_{2}\big(X^{(\frac{\lceil n T \rceil}{n})},\widetilde{\mathfrak{L}}_{\frac{\lceil n T \rceil}{n}},\widetilde{\mathfrak{R}}_{\frac{\lceil n T \rceil}{n}}\big)\big)=\sum \limits_{k=0}^{\infty}\Theta_{0}\big(F_{1}\,\mathbbm{1}_{\frac{k}{n}\leq T<\frac{k+1}{n}}F_{2}\big(X^{(\frac{k+1}{n})},\widetilde{\mathfrak{L}}_{\frac{k+1}{n}},\widetilde{\mathfrak{R}}_{\frac{k+1}{n}}\big)\big)\end{equation}
For every atom $(\ell,\omega)$ of $\mathfrak{R}$ or $\mathfrak{L}$ such that $\ell>\tau_T$, an application of \cite[Lemma 11]{ALG}
shows that $\text{tr}_{\frac{\lceil n T \rceil}{n}}(\omega)\to \text{tr}_{T}(\omega)$  as $n\to \infty$. Using also the fact that
$r\to \tau_{r}$ is c\`adl\`ag, 
we easily obtain that $\widetilde{\mathfrak{L}}_{\frac{\lceil n T \rceil}{n}}\to \widetilde{\mathfrak{L}}_{T}$ and $\widetilde{\mathfrak{R}}_{\frac{\lceil n T \rceil}{n}}\to \widetilde{\mathfrak{R}}_{T}$ when $n\to \infty$, with respect to the topology on $M(\mathcal{S})$. 
Since $F_{2}$ is bounded and continuous, we can take the limit when $n$ goes to $\infty$ to obtain:
\begin{equation}\label{eq:indep3}\lim \limits_{n\to \infty}\Theta_{0}\big(F_{1}\times F_{2}\big(X^{(\frac{\lceil n T \rceil}{n})},\widetilde{\mathfrak{L}}_{\frac{\lceil n T \rceil}{n}},\widetilde{\mathfrak{R}}_{\frac{\lceil n T \rceil}{n}}\big)\big)=\Theta_{0}\big(F_{1}\times F_{2}\big(X^{(T)},\widetilde{\mathfrak{L}}_{T},\widetilde{\mathfrak{R}}_{T}\big)\big).\end{equation}
On the other hand, for every $k\geq 0$, $F_{1}\,\mathbbm{1}_{\frac{k}{n}\leq T<\frac{k+1}{n}}$ is $\mathcal{F}_{\frac{k+1}{n}}$-measurable and
is thus equal, $\Theta_0$-a.s.~, to a measurable function of $B^\bullet_{\frac{k+1}{n}}$. Hence we can apply the spatial Markov property of Theorem \ref{spatial_fix_level} to obtain:
\begin{align*}
\sum_{k=0}^\infty\Theta_{0}\big(F_{1}\,\mathbbm{1}_{\frac{k}{n}\leq T<\frac{k+1}{n}}\,F_{2}\big(X^{(\frac{k+1}{n})},\widetilde{\mathfrak{L}}_{\frac{k+1}{n}},\widetilde{\mathfrak{R}}_{\frac{k+1}{n}}\big)\big)
&=\sum_{k=0}^\infty\Theta_{0}\Big(F_{1}\,\mathbbm{1}_{\frac{k}{n}\leq T<\frac{k+1}{n}}\Theta_{Z_{\frac{k+1}{n}}}\big(F_{2}\big)\Big)\\
&=\Theta_{0}\big(F_{1}\,\Theta_{Z_{\frac{\lceil n T \rceil}{n}}}\big(F_{2}\big)\big).
\end{align*}

Since the process $Z$ is c\`adl\`ag, we 
have $Z_{\frac{\lceil n T \rceil}{n}}\to Z_T$ as $n\to\infty$. Moreover, the fact that $F_{2}$ is bounded and  continuous and the scaling property of $(\Theta_{l})_{l>0}$ imply that the mapping $l\mapsto \Theta_{l}(F_{2})$ is also bounded and continuous. It follows that
\begin{align}\label{eq:indep4}\lim_{n\to\infty} \sum_{k=0}^\infty\Theta_{0}\big(F_{1}\,\mathbbm{1}_{\frac{k}{n}\leq T<\frac{k+1}{n}}\,F_{2}\big(X^{(\frac{k+1}{n})},\widetilde{\mathfrak{L}}_{\frac{k+1}{n}},\widetilde{\mathfrak{R}}_{\frac{k+1}{n}}\big)\big)&=
\lim \limits_{n\to \infty}\Theta_{0}\big(F_{1}\,\Theta_{Z_{\frac{\lceil n T \rceil}{n}}}\big(F_{2}\big)\big)\nonumber\\
&=\Theta_{0}\big(F_{1}\,\Theta_{Z_{T}}\big(F_{2}\big)\big).\end{align}
The identity \eqref{eq:indep} then follows by passing to the limit $n\to\infty$
in \eqref{eq:indep2}, using \eqref{eq:indep3} and \eqref{eq:indep4}.
\end{proof}

Let us state some direct consequences of Theorem \ref{spatial_random_level}~. For every $z>0$, set
\begin{equation}\label{T_z}
   T_{z}:=\inf\{r\geq 0:\: Z_{r}\geq z\} 
\end{equation}
which is a stopping time of the filtration $(\mathcal{F}_{r+})_{r\geq 0}$. Note that $0<T_{z}<\infty$, $\Theta_{0}$-a.s. Moreover since  $Z$ does not have positive jumps we have $Z_{T_{z}}=z$. Applying Theorem \ref{spatial_random_level} with $T=T_{z}$, we obtain the following result.

\begin{cor}\label{SpaMarkov_T_z}
Let $z>0$. Under $\Theta_{0}$,  $(\check{B}_{T_{z}}^{\bullet}, \rho_{T_{z}},\check{\Delta}^{T_{z}}, |\cdot|_{\check{\Delta}^{T_{z}}})$  is an infinite volume Brownian disk with perimeter $z$ and is  independent of $(B_{T_{z}}^{\bullet}, 0,\Delta^{T_{z}}, |\cdot|_{\Delta^{T_{z}}})$.
\end{cor}

The next goal is to extend the definition of process $Z$ under $\Theta_{z}$.
\noindent It will be useful to consider the Skorokhod space $\mathbb{D}(\mathbb{R}_{+},\mathbb{R})$ of càdlàg functions from $\mathbb{R}_{+}$ into $\mathbb{R}$. We write $(\xi_{t})_{t\geq 0}$ for the canonical process on $\mathbb{D}(\mathbb{R}_{+},\mathbb{R})$ and $(\mathcal{D}_{t})_{t\geq 0}$ for the canonical filtration. We introduce a probability measure  $\mathbb{P}$ on $\big(\mathbb{D}(\mathbb{R}_{+},\mathbb{R}),\mathcal{D}\big)$ such that under  $\mathbb{P}$, the process $\xi$ is distributed as a L\'evy process without positive jumps with Laplace exponent 
$\psi(\lambda):=\sqrt{\frac{8}{3}}\frac{\Gamma(\lambda+\frac{3}{2})}{\Gamma(\lambda)}$, i.e.:
\[\mathbb{E}[\exp(\lambda\xi_{1})]=\exp(\psi(\lambda))\,,\quad \forall \lambda>-\frac{3}{2},\]
where $\mathbb{E}$ stands for the expectation with respect to $\mathbb{P}$.  We refer to \cite[Lemma 2.1]{BBCK} for the existence of this L\'evy process. Since $\psi^{\prime}(0+)>0$, standard properties of Levy processes imply that $\xi$ drifts to $\infty$ (see for example \cite[Chapter VII]{BertoinBook}). 
We also introduce the time change:
\[\kappa(r):=\inf\big\{s\geq 0:~ \int_{0}^{s}\exp(\frac{1}{2}\xi_{t})\:dt\geq r\big\}\:.\]
Theorem 24 in \cite{Growth} states that the process $Z$ under $\Theta_0$ is a self-similar Markov process started at $0$ with index $\frac{1}{2}$ and  Laplace exponent $\psi$. In particular, the process $(Z_{T_{z}+t})_{t\geq 0}$ is distributed under $\Theta_{0}$ as:
\[\Big(z\exp\big(\xi_{\kappa(z^{-\frac{1}{2}}r)}\big)\Big)_{r\geq 0}\]
under $\mathbb{P}$.
As a consequence of \eqref{geoZ} and Corollary \ref{SpaMarkov_T_z}  we obtain:

\begin{cor}\label{Z_z}
Fix $z>0$.  Then, for every $r\geq 0$,
\[Z_{r}:=\lim \limits_{\epsilon\to 0}\epsilon^{-2}|\check{B}_{r}^{\bullet} \cap B_{r+\epsilon}|\]
exists $\Theta_z$-a.s. Moreover the process $(Z_{r})_{r\geq 0}$ has a c\`adl\`ag modification
under $\Theta_z$, which is distributed  as a $(\frac{1}{2},\psi)$-self-similar Markov process started at $z$, i.e.
\begin{equation}\label{Lamperti}
(Z_{r})_{r\geq 0}\overset{(d)}{=}\Big(z\exp\big(\xi_{\kappa(z^{-\frac{1}{2}}r)}\big)\Big)_{r\geq 0}\:
\end{equation}
where $\xi$ is distributed according to $\mathbb{P}$.
\end{cor}
Let $\mathcal{M}_{\infty}^{(z)}$ be the infinite volume Brownian disk with perimeter $z$ defined under the probability measure $\Theta_{z}$
as explained in Section \ref{IBD}. We say that $\gamma:[s,t]\to \mathcal{M}_{\infty}^{(z)}$ is a separating cycle if it is an injective continuous cycle that does not hit the boundary of  $\mathcal{M}_{\infty}^{(z)}$ and if any path connecting this boundary to $\infty$ has to cross the range of $\gamma$. 
We are going to use the "strong" spatial Markov property (Theorem \ref{spatial_random_level}) to study the separating cycles of the infinite volume Brownian disk. As in the previous sections, we consider $B_{r,s}^{\circ}=\text{Int}(B^\bullet_s\backslash B^\bullet_r)$ and
\[L_{r,s}=\inf \{\Delta(g): ~ g:[0,1]\to B_{r,s}^{\circ}\:\:\text{separating cycle}\},\]
 for every $0\leq r<s$, and we will now study $L_{r,s}$ under $\Theta_z$. 
To simplify notation we write $L_{r}:=L_{r,\infty}$ for every $r\geq 0$. Note that $L_{0}$ is the infimum of lengths of paths separating the boundary of $\mathcal{M}_{\infty}^{(z)}$ from infinity. 
We have the following analog of Theorem \ref{tail} for the infinite volume Brownian disk: 
\begin{prop}\label{cortail}
Fix  $z$ a positive real number.
\\
\\
$\rm(i)$ We have
\[\limsup \limits_{u \to \infty} \frac{\log \big(\Theta_{z}(L_{0}>u)\big)}{u} \leq -\sup_{s>1}\frac{1}{2(s-1)}\log(\frac{s^{2}}{2s-1}) .\]
Consequently, $\Theta_{z}(L_{0}>u)$ decreases at least exponentially fast when $u$ goes to $\infty$. 
\\
\\
$\rm(ii)$ There exist two constants $0<\tilde{c}_{1}\leq \tilde{c}_{2}$ such that:
\[\forall \epsilon\geq 0, \: \tilde{c}_{1} (1\wedge\epsilon^{2} )\leq \Theta_{z}(L_{0}<\epsilon) \leq \tilde{c}_{2} \epsilon^{2} .\]
\end{prop}
\begin{proof}
By scaling, it is enough to consider $z=1$. 
The spatial Markov property (Theorem \ref{spatial_fix_level}) and again a scaling argument show that the distribution of $Z_{1}^{-\frac{1}{2}} L_{1}$ under $\Theta_{0}$ coincides with the distribution of $L_{0}$ under $\Theta_{1}$ and moreover $Z_{1}^{-\frac{1}{2}} L_{1}$  is independent of $Z_{1}$
under $\Theta_0$. We obtain:
\[\Theta_{1}(L_{0}>u)\Theta_{0}(Z_{1}>1)=\Theta_{0}(L_{1}>Z_{1}^{\frac{1}{2}}u,\: Z_{1}>1 )\leq \Theta_{0}(L_{1}>u).\]
 Part $\rm(i)$ of the proposition then follows from Theorem \ref{tail}~.
\\
\\
Let us prove $\rm(ii)$. The upper bound is a direct consequence of the beginning of the proof, noting that for every $\epsilon\geq 0$:
\[\Theta_{1}(L_{0}<\epsilon)\Theta_{0}(Z_{1}<1)=\Theta_{0}(Z_{1}^{-\frac{1}{2}}L_{1}<\epsilon,\:Z_{1}<1)\leq \Theta_{0}(L_{1}<\epsilon)\]
so that, by Theorem \ref{tail}~, if $\tilde{c}_{2}:=c_{2}/\Theta_{0}(Z_{1}<1)\in(0,\infty)$ we have:
\begin{equation}\label{L_0_ep}
\Theta_{1}(L_{0}<\epsilon) \leq \tilde{c}_{2} \epsilon^{2}
\end{equation}
for every $\epsilon>0$. We argue in a similar way to obtain the lower bound. Let $\epsilon\leq 1$ and let $m_{0}\geq 1$ be an  integer. We have similarly:
\begin{align*}
\Theta_{0}(&L_{1}\leq \epsilon) =\Theta_{0}(Z_{1}^{\frac{1}{2}}Z_{1}^{-\frac{1}{2}}L_{1}\leq \epsilon)\\
&\leq \Theta_{0}(Z_{1}^{-\frac{1}{2}}L_{1}\leq m_{0} \epsilon,\:Z_{1}^{\frac{1}{2}}>\frac{1}{m_{0}})+ \sum_{m=m_{0}}^{\infty} \Theta_{0}\big(Z_{1}^{\frac{1}{2}}\in[\frac{1}{m+1}, \frac{1}{m}]\:,\: Z_{1}^{-\frac{1}{2}}L_{1}\leq (m+1)\epsilon\big).
\end{align*}
Now we can use the first observations of the proof and \eqref{L_0_ep} to get:
\begin{align*}
\Theta_{0}(L_{1}\leq \epsilon) &\leq  \Theta_{0}(Z_{1}^{-\frac{1}{2}}L_{1}\leq m_{0}\epsilon)+\tilde{c}_{2}\epsilon^{2}   \sum_{m=m_{0}}^{\infty} (m+1)^{2}\:\Theta_{0}(Z_{1}\in[(m+1)^{-2}, m^{-2}])\\
&= \Theta_{1}(L_{0}\leq m_{0}\epsilon)+\tilde{c}_{2}\epsilon^{2}   \sum_{m=m_{0}}^{\infty} (m+1)^{2}\:\Theta_{0}(Z_{1}\in[(m+1)^{-2}, m^{-2}])
\end{align*}
Under $\Theta_{0}$, the density of $Z_{1}$ is $\frac{3^{\frac{3}{2}}}{\sqrt{2\pi}}\sqrt{x}\exp(-\frac{3}{2}x)\:dx$, so for every $0<a<b$:
\[\Theta_{0}(Z_{1}\in[a, b])\leq \sqrt{\frac{6}{\pi}}(b^{\frac{3}{2}}-a^{\frac{3}{2}}) . \] 
Hence we can find a constant $c_{3}>0$, which does not depend on the choice of $m_0$, such that
\[\Theta_{0}(L_{1}\leq \epsilon)
\leq  \Theta_{1}(L_{0}\leq m_{0}\epsilon)+ c_{3} \epsilon^{2}   \sum_{m=m_{0}}^{\infty} \frac{1}{m^{2}}.\]
Then using Theorem \ref{tail}~, we get for every $\epsilon\in[0,1]$:
\[(c_{1}- c_3\sum_{m=m_{0}}^{\infty} \frac{1}{m^{2}})\,\epsilon^{2} \leq \Theta_{1}(L_{0}\leq m_{0}\epsilon) .\]
We obtain the lower bound in $\rm(ii)$ by choosing $m_{0}$ such that $\sum \limits_{m=m_{0}}^{\infty} m^{-2}<\frac{c_{1}}{c_{3}}.$
\end{proof}
\noindent Recall that  $0<\tilde{c}_{1}<\tilde{c}_{2}$ are the constants  appearing in Proposition \ref{cortail}. The end of this section is devoted to the proof of the following result which will be crucial for the proof of Theorem \ref{iso} $\rm(i)$. Before stating the result, observe that the definition 
\eqref{T_z} of $T_r$ for $r>0$ also makes sense under  $\Theta_z$ by Corollary \ref{Z_z}.

\begin{prop}\label{cycleIBD}
There exists $\tilde{c}_{3}>0$ such that, for every $r>2\tilde{c}_{2}/\tilde{c}_{1}$ and $\epsilon>0$,
\[\Theta_{1}\big(L_{0,T_{2r}}\leq\epsilon\big)\geq \tilde{c}_{3} (1\wedge\epsilon^{2}).\]
\end{prop}
The proof of Proposition \ref{cycleIBD} is based on the next lemma:

\begin{lem}\label{T2}
Let $z> 0$. Then, for every $A>z$ and $\beta<3$ we have:
\[\Theta_{z}\big(\sup \limits_{[0,\epsilon]}Z\geq A\big)=o(\epsilon^{\beta})\]
as $\epsilon \to 0$.
\end{lem}

\begin{proof}
By a scaling argument, it is enough to prove the lemma with $z=1$. Fix $A>1$.
Introduce the stopping time $T:=\inf \{t\geq 0: ~\xi_{t}\geq \log(A)\}$ which is finite $\mathbb{P}$-a.s.
By Corollary \ref{Z_z} for every $\epsilon>0$:
\[\Theta_{1}\big(\sup \limits_{t\in[0,\epsilon]} Z_{t}\geq A\big)=\mathbb{P}\big(T\leq \kappa(\epsilon)\big)=\mathbb{P}\big(\int_{0}^{T}\exp(\frac{1}{2}\xi_{r})\:dr\leq \epsilon\big).\]
Let $\alpha\in(0,1)$, we split $\mathbb{P}(\int_{0}^{T}\exp(\frac{1}{2}\xi_{r})\:dr\leq\epsilon)$ as follows:
\begin{equation}\label{split}
\mathbb{P}\big(\int_{0}^{T}\exp(\frac{1}{2}\xi_{r})\:dr\leq\epsilon\big)\leq\mathbb{P}(T\leq\epsilon^{\alpha})+\mathbb{P}\big(\int_{0}^{\epsilon^{\alpha}}\exp(\frac{1}{2}\xi_{r})\:dr\leq\epsilon\big)
\end{equation}
and we study each term separately. We need to estimate $\mathbb{P}(T<\delta)$ for $\delta>0$.
As $\xi$ is a Lévy process without positive jumps which drifts to $\infty$,  we have by standard properties of L\'evy processes
\[\mathbb{E}[\exp( -\psi(\lambda)T)]=\exp(-\lambda \log(A))\]
for every $\lambda>0$. See for example \cite[Chapter VII]{BertoinBook} for a proof. 
Remark that there exists $c>0$ such that, for every $\lambda>1$, we have $\psi(\lambda)<c\lambda^{\frac{3}{2}}$ and that an application of Markov's inequality gives:
\[\mathbb{P}(T<\delta)=\mathbb{P}(-\psi(\lambda)T>-\psi(\lambda)\delta)\leq \exp(\psi(\lambda)\delta-\lambda\log(A)).\]
So taking $\lambda=\delta^{-\frac{2}{3}}$ in the previous bound we obtain:
\[\mathbb{P}(T<\delta)=\underset{\delta\downarrow 0}{O}\big(\exp(-\delta^{-\frac{2}{3}}\log(A))\big)\:.\]
Consequently, for every $q>0$, $\mathbb{P}(T<\delta)=o(\delta^{q})$ as $\delta\downarrow 0$. Let us study the other term appearing in \eqref{split}. Fix $\beta\in (0,3)$. Again by using Markov's inequality we have:
\begin{align*}
\mathbb{P}\big(\int_{0}^{\epsilon^{\alpha}}\exp(\frac{1}{2}\xi_{r})\:dr\leq\epsilon\big)
&\leq \epsilon^{\beta} \mathbb{E}\Big[\big(\int_{0}^{\epsilon^{\alpha}}\exp(\frac{1}{2}\xi_{r})\:dr\big)^{-\beta}\Big].
\end{align*}
But then an application of Jensen inequality gives
\begin{align*}
\mathbb{P}\big(\int_{0}^{\epsilon^{\alpha}}\exp(\frac{1}{2}\xi_{r})\:dr\leq\epsilon\big)
&\leq \epsilon^{(1-\alpha)\beta-\alpha} \mathbb{E}\big[\int_{0}^{\epsilon^{\alpha}}\exp(-\frac{\beta}{2}\xi_{r})\:dr\big]\\
&=\epsilon^{(1-\alpha)\beta-\alpha} \frac{\exp(\psi(-\frac{\beta}{2})\epsilon^{\alpha})-1}{\psi(-\frac{\beta}{2})}.
\end{align*}
We obtain that $\mathbb{P}\big(\int_{0}^{\epsilon^{\alpha}}\exp(\frac{1}{2}\xi_{r})\:dr\leq\epsilon\big)=O(\epsilon^{(1-\alpha)\beta}) $
as $\epsilon \downarrow 0$. Since this is true for every $\beta\in(0,3)$ and $\alpha\in(0,1)$, the lemma follows. 
\end{proof}
\noindent Let us deduce Proposition \ref{cycleIBD} from Lemma \ref{T2}.
\\
\\
\noindent{\it Proof of Proposition \ref{cycleIBD}.} 
Fix $r>2 \tilde{c}_{2}/\tilde{c}_{1}\geq 2$.
Let $\gamma$ be a path separating the boundary of $\mathcal{M}_{\infty}^{(1)}$ from infinity. If $\gamma$ does not stay inside $B_{T_{2r}}^{\bullet}$ then it has to stay outside $B_{T_{r}}^{\bullet}$ or to connect $B_{T_{r}}^{\bullet}$ and $\check{B}_{T_{2r}}^{\bullet}$. Since the distance between $B_{T_{r}}^{\bullet}$ and  $\check{B}_{T_{2r}}^{\bullet}$ is $T_{2r}-T_{r}$ we have:
\[L_{0}\geq L_{0,T_{2r}} \wedge L_{T_{r}} \wedge (T_{2r}-T_{r})\:\:\Theta_{1}\text{-a.s.}\]
Consequently, for every $\epsilon>0$:
\begin{align*}
\Theta_{1}(L_{0}\leq \epsilon)&\leq \Theta_{1}\big(L_{0,T_{2r}}\leq \epsilon\big)+\Theta_{1}\big(L_{T_{r}}\leq \epsilon\big)+\Theta_{1}\big(T_{2r}-T_{r}\leq \epsilon\big).
\end{align*}
By Theorem \ref{spatial_random_level} and Corollary \ref{SpaMarkov_T_z}~, the distribution of $L_{T_{r}}$ under $\Theta_{1}$ is the distribution of $L_{0}$ under $\Theta_{r}$. Using a  scaling argument, we obtain:
\[\Theta_{1}\big(L_{T_{r}}\leq \epsilon\big)= \Theta_{r}\big(L_{0}\leq \epsilon\big)=\Theta_{1}\big( L_{0}\leq \frac{\epsilon}{\sqrt{r}}\big)\leq \tilde{c}_{2} \frac{\epsilon^{2}}{r} \]
and
\[\Theta_{1}\big(T_{2r}-T_{r}\leq \epsilon\big)=\Theta_r(T_{2r}\leq \epsilon)=\Theta_{1}\big(T_{2}\leq \frac{\epsilon}{\sqrt{r}}\big)=\Theta_{1}\big(\sup \limits_{t\in[0,\frac{\epsilon}{\sqrt{r}}]}Z_{t}\geq 2\big)=\underset{\epsilon\downarrow 0}{o}(\epsilon^{2})\]
where the last equality comes from Lemma \ref{T2} (taking $A=2$). 
We finally derive that:
\begin{align*}
\Theta_{1}\big(L_{0,T_{2r}}\leq \epsilon\big)&\geq \Theta_{1}\big(L_{0}\leq \epsilon\big)-\tilde{c}_{2}\frac{\epsilon^{2}}{r}+o(\epsilon^{2})
\geq \tilde{c}_{1}(1\wedge \epsilon^{2})-\tilde{c}_{2}\frac{\epsilon^{2}}{r}+o(\epsilon^{2})
\end{align*}
where in the second line we use Proposition \ref{cortail}~. Since  $r>2\tilde{c}_{2}/\tilde{c}_{1}$  we obtain the desired result.
\hfill $\square$
\section{Isoperimetric inequalities}\label{seciso}
\subsection{Preliminary results on the volume of the hulls}
This section is devoted to preliminary results about the volume of hulls. This will simplify some arguments in the derivation of Theorem \ref{iso}. We are going to use the following result \cite[Theorem 1.4]{Hull}
\begin{align} \label{eq_B_sachant_Z}
\Theta_{0}\big(\exp(-\lambda | B^{\bullet}_{r}|)\:|\: Z_{r}=l\big) 
= &r^{3}(2\lambda)^{\frac{3}{4}} \frac{\cosh((2\lambda)^{\frac{1}{4}}r)}{\sinh^{3}((2\lambda)^{\frac{1}{4}}r)}\nonumber\\
&\cdot \exp\Big(-l\big(\sqrt{\frac{\lambda}{2}}(3\coth^{2}((2\lambda)^{\frac{1}{4}}r)-2)-\frac{3}{2r^{2}}\big)\Big)
\end{align}
for every $\lambda> 0$. In particular, using \eqref{eq_Laplace_Z}, we obtain that for every $\lambda\geq 0$:
\begin{equation}\label{Laplace_Hull}
  \Theta_{0}\big(\exp(-\lambda | B^{\bullet}_{r}|) \big)=3^{\frac{3}{2}}\cosh((2\lambda)^{\frac{1}{4}}r)\:(\cosh^{2}((2\lambda)^{\frac{1}{4}}r)+2)^{-\frac{3}{2}}
\end{equation}
(this formula also appears in \cite{Hull}).

\begin{cor}\label{lemvolIBD}
There exists a constant $C>0$ such that for every $z>0$ and $r>0$:
\[\Theta_{z}\big(|B_{r}^{\bullet}|\big)\leq C (r+\sqrt{z})^{4}.\]
\end{cor}
\begin{proof}
Fix $z>0$ and $r>0$. First remark that, under $\Theta_{0}(\cdot\,|\,T_{z}\leq \sqrt{z})$,  the hull $B_{T_{z}+r}^{\bullet}$ is contained in $B_{\sqrt{z}+r}^{\bullet}$. So an application of Corollary \ref{SpaMarkov_T_z} gives:
\begin{align*}
\Theta_{0}\big(|B_{r+\sqrt{z}}^{\bullet}|\big)&\geq \Theta_{0}\big(|B_{r+\sqrt{z}}^{\bullet}| \mathbbm{1}_{T_{z}<\sqrt{z}}\big) \\
&\geq \Theta_{0}\big((|B_{T_{z}+r}^{\bullet}|-|B_{T_{z}}^{\bullet}|) \mathbbm{1}_{T_{z}<\sqrt{z}}\big)\\
&=\Theta_{z}\big(|B_{r}^{\bullet}|\big) \:\Theta_{0}(T_{z}<\sqrt{z}).
\end{align*}
On the other hand, we have $\{Z_{\sqrt{z}}>z\}\subset\{T_{z}\leq \sqrt{z}\}$. By scaling we also have $\Theta_{0}(Z_{\sqrt{z}}>z)=\Theta_{0}(Z_{1}>1)>0$ and $\Theta_{0}\big(|B_{r+\sqrt{z}}^{\bullet}|\big)=(r+\sqrt{z})^{4}\Theta_{0}\big(|B_{1}^{\bullet}|\big)$. Finally, it is easy to deduce from \eqref{Laplace_Hull} that $\Theta_{0}\big(|B_{1}^{\bullet}|\big)$ is finite. This gives  the statement of the corollary with $C=\Theta_{0}(|B_{1}^{\bullet}|)/\Theta_{0}(Z_{1}>1)$.
\end{proof}
We now give two lemmas that will be useful to control the fluctuations of the volume of  hulls in the Brownian plane.
\begin{lem}\label{moment}
For every $\beta\in \mathbb{R}$, we have $\Theta_{0}(|B^{\bullet}_{1}|^{\beta})<\infty$ if and only if $\beta<\frac{3}{2}$.
\end{lem}
\begin{proof}
To simplify notation write $F(\lambda):=\Theta_{0}(\exp(-\lambda|B^{\bullet}_{1}|))$ for every $\lambda\geq 0$.
Remark that for every $\lambda>0$, we have $F^{\prime \prime}(\lambda)=\Theta_{0}(|B^{\bullet}_{1}|^{2}\exp(-\lambda |B^{\bullet}_{1}|))$.
If  $\alpha >0$, we have:
\begin{align*}
\int_{\mathbb{R}_{+}} \lambda^{\alpha-1} F^{\prime \prime}(\lambda)\: d\lambda&=\Theta_{0}\Big(|B^{\bullet}_{1}|^{2}\int_{\mathbb{R}_{+}} \lambda^{\alpha-1} \exp(-\lambda|B^{\bullet}_{1}|)\: d\lambda\Big)=\Gamma(\alpha)\,\Theta_{0}\big(|B^{\bullet}_{1}|^{2-\alpha}\big) .
\end{align*}
From the explicit expression of $F$ given in \eqref{Laplace_Hull} we get:
\[F^{\prime \prime}(\lambda)=\frac{1}{9\sqrt{2}} \lambda^{-\frac{1}{2}}+O(1)\]
as $\lambda\downarrow0$, and
\[F^{\prime \prime}(\lambda)=O\big(\exp(-2(2\lambda)^{\frac{1}{4}})\big)\]
as $\lambda\uparrow \infty$.
So for $\alpha=\frac{1}{2}$, $\Theta_{0}(|B^{\bullet}_{1}|^{\frac{3}{2}})=\int_{\mathbb{R}_{+}} \lambda^{-\frac{1}{2}} F^{\prime \prime}(\lambda)\: d\lambda=\infty$ and for every $\alpha>\frac{1}{2}$, $\Theta_{0}(|B^{\bullet}_{1}|^{2-\alpha})<\infty$.
\end{proof}
We conclude this section with the following consequence of Lemma \ref{moment}:
\begin{lem}\label{CantelliH}
For every $\beta_{1}>0$ and $\beta_{2}>2/3$ we have $\Theta_0$-a.s.
\[\inf_{r>0} \frac{|B^{\bullet}_{r}|}{r^4(1+|\log(r)|)^{-\beta_{1}}}>0\]
and
\[\sup_{r>0} \frac{|B^{\bullet}_{r}|}{r^4(1+|\log(r)|)^{\beta_{2}}}<\infty.\]
\end{lem}
\begin{proof}
Fix $\beta_{1}$ and $\beta_{2}$ as in the statement.
By Lemma \ref{moment}, the quantities $\Theta_{0}(|B^{\bullet}_{1}|^{-1/\beta_{1}})$ and $\Theta_{0}(|B^{\bullet}_{1}|^{1/\beta_{2}})$ are finite. This implies by the scaling invariance of $\mathcal{M}_{\infty}$:
\begin{align*}
\sum \limits_{m\in \mathbb{Z}}\Theta_{0}(|B^{\bullet}_{2^{m}}|^{-1}>|m|^{\beta_{1}} 2^{-4m})&\leq 2 \sum \limits_{m=0}^{\infty} \Theta_{0}(|B^{\bullet}_{1}|^{-\frac{1}{\beta_{1}}}>m)<\infty
\end{align*}
and
\begin{align*}
\sum \limits_{m\in \mathbb{Z}}\Theta_{0}(|B^{\bullet}_{2^{m}}|>|m|^{\beta_{2}} 2^{4m})&\leq 2\sum \limits_{m=0}^{\infty}\Theta_{0}(|B^{\bullet}_{1}|^{\frac{1}{\beta_{2}}}>m)<\infty\:.
\end{align*}
The result then follows by the Borel-Cantelli lemma.
\end{proof}

\subsection{Proof of Part $\rm(i)$ of Theorem \ref{iso}}
\noindent The goal of this section is to prove the following slightly more precise form of Theorem \ref{iso} \rm(ii).
\begin{prop}\label{infand0}
Let $f:\mathbb{R}_{+}\to \mathbb{R}_{+}^{*}$ be a positive nondecreasing function such that
$\sum \limits_{m\in \mathbb{N}}f(m)^{-2}=\infty\:.$
 Then,
\begin{equation}\label{infand0:0}
\mathop{\inf_{A\in \mathcal{K}}}_{A\subset B_{1}^{\bullet}} \frac{\Delta(\partial A)}{|A|^{\frac{1}{4}}}f(|\log(|A|)|)=0,\:\:\Theta_{0}\text{-a.s.}
\end{equation}
and
\begin{equation}\label{infand0:inf}
\mathop{\inf_{A\in \mathcal{K}}}_{B_{1}^{\bullet}\subset A} \frac{\Delta(\partial A)}{|A|^{\frac{1}{4}}}f(|\log(|A|)|)=0,\:\:\Theta_{0}\text{-a.s.}
\end{equation}
\end{prop} 

\begin{proof}
Fix $r>2 \tilde{c}_{2}/\tilde{c}_{1}\geq 2$, where $\tilde{c}_{1}$ and $\tilde{c}_{2}$ are as in Proposition \ref{cortail}, and let $f:\mathbb{R}_{+}\to \mathbb{R}_{+}^{*}$ be a positive nondecreasing function such that
$\sum \limits_{m\in \mathbb{N}}f(m)^{-2}=\infty\:.$
We give a detailed proof of the \eqref{infand0:inf} since \eqref{infand0:0}  can be obtained, \textit{mutatis mutandis}, by the same method. Recall the notation $T_{r}:=\inf\{t\geq 0:\: Z_{t}\geq r\}$. Since $Z$ does not have positive jumps, we have $Z_{T_{r}}=r$.
 For every $n\geq 1$, a separating cycle taking values in $B_{T_{r^{n}},T_{r^{n+1}}}^{\bullet}$ bounds a Jordan domain $A\in\mathcal{K}$ such that $B_{T_{r^{n}}}^{\bullet}\subset A$ and for $n$ large enough we also have $B_{1}^{\bullet}\subset A$. Hence,
\[ \mathop{\inf_{A\in \mathcal{K}}}_{ B_{1}^{\bullet}\subset A} \frac{\Delta(\partial A)}{|A|^{\frac{1}{4}}}f(|\log(|A|)|)\leq\liminf \limits_{n\to \infty}\frac{L_{T_{r^{2n+1}},T_{r^{2n+2}}}}{|B_{T_{r^{2n+1}}}^{\bullet}|^{\frac{1}{4}}}f(\log(|B_{T_{r^{2n+2}}}^{\bullet}|)),\:\:\Theta_{0}\text{-a.s.} \]
So to obtain \eqref{infand0:inf} it is enough to show that:
\begin{equation}\label{cut_T_r}
\liminf \limits_{n\to \infty}\frac{L_{T_{r^{2n+1}},T_{r^{2n+2}}}}{|B_{T_{r^{2n+1}}}^{\bullet}|^{\frac{1}{4}}}f(\log(|B_{T_{r^{2n+2}}}^{\bullet}|)) =0,\:\:\Theta_{0}\text{-a.s.}
\end{equation}
Let us study the growth of the sequence $(T_{r^{n}})_{n\in \mathbb{N}}$. First note that:
\[\Theta_{0}(T_{1}\geq u)\leq\Theta_{0}(Z_{u}\leq  1)= \Theta_{0}(Z_{1}\leq  u^{-2})=\frac{3^{\frac{3}{2}}}{\sqrt{2\pi}}\int_{0}^{u^{-2}} \sqrt{x}\exp(-\frac{3}{2}x)\:dx\leq \sqrt{\frac{6}{\pi}}u^{-3}\]
where in the first equality we apply the scaling invariance of the Brownian plane and in the second one we use the density of $Z_{1}$ (see \eqref{eq_Laplace_Z}). In particular, we have $\Theta_{0}(T_{1})<\infty$. So by scaling invariance  we obtain:
\[\sum \limits_{n\in \mathbb{N}^{*}}\Theta_{0}(T_{r^{n}}>n r^{\frac{n}{2}})=\sum \limits_{n\in \mathbb{N}^{*}}\Theta_{0}(T_{1}>n)<\infty.\]
The Borel-Cantelli lemma then implies that $\limsup\limits_{n\to \infty} (nr^{n/2})^{-1}T_{r^{n}}\leq 1$, $\Theta_{0}$-a.s. Since $\lim \limits_{s\to \infty}T_{s}=\infty$,  $\Theta_{0}$-a.s.~, Lemma \ref{CantelliH} gives:
\[\limsup\limits_{n\to \infty}\frac{\log(|B_{T_{r^{2n+2}}}^{\bullet}|)}{4\log(T_{r^{2n+2}})}\leq 1,\:\:\Theta_{0}\text{-a.s.}\]
and then we deduce:
\[\limsup\limits_{n\to \infty}\frac{\log(|B_{T_{r^{2n+2}}}^{\bullet}|)}{2n \log(r)}\leq 1,\:\:\Theta_{0}\text{-a.s.}\]
Fix $h>2\log(r)$. As $f$ is nondecreasing we have:
\begin{equation}\label{dernier**}
\liminf \limits_{n\to \infty}\frac{L_{T_{r^{2n+1}},T_{r^{2n+2}}}}{|B_{T_{r^{2n+1}}}^{\bullet}|^{\frac{1}{4}}}f(\log(|B_{T_{r^{2n+2}}}^{\bullet}|))\leq \liminf \limits_{n\to \infty}\frac{L_{T_{r^{2n+1}},T_{r^{2n+2}}}}{|B_{T_{r^{2n}},T_{r^{2n+1}}}^{\bullet}|^{\frac{1}{4}}}f(h n).
\end{equation} 
We will use the Borel-Cantelli lemma to conclude. By Theorem \ref{spatial_random_level}, under $\Theta_{0}$ :
\[\Big(\:\frac{1}{r^{4n}}|B^{\bullet}_{T_{r^{2n}},T_{r^{2n+1}}}|,\frac{1}{r^{n+\frac{1}{2}}} L_{T_{r^{2n+1}},T_{r^{2n+2}}}\:\Big)_{n\in \mathbb{N}}\:.\]
is an i.i.d. sequence of random variables and for every $n\geq 0$ the variables:
\[\frac{1}{r^{4n}}|B^{\bullet}_{T_{r^{2n}},T_{r^{2n+1}}}|,\:\: \text{and}\:\: \frac{1}{r^{n+\frac{1}{2}}} L_{T_{r^{2n+1}},T_{r^{2n+2}}}\] are independent. Moreover $\frac{1}{r^{4n}}|B^{\bullet}_{T_{r^{2n}},T_{r^{2n+1}}}|$ (resp. $\frac{1}{r^{n+\frac{1}{2}}} L_{T_{r^{2n+1}},T_{r^{2n+2}}}$) is distributed 
under $\Theta_{0}$ as $B_{T_{r}}
^{\bullet}$ (resp. $L_{0,T_{r}}$) under $\Theta_{1}$. Fix $\delta>0$, such that $\Theta_{1}(B_{T_{r}}
^{\bullet}>\delta)>0$ and let $\epsilon>0$. By the previous remark we have:
\begin{align*}
\sum \limits_{n=0}^{\infty}\Theta_{0}\Big(|B^{\bullet}_{T_{r^{2n}},T_{r^{2n+1}}}|>\delta r^{4n},\:  &L_{T_{r^{2n+1}},T_{r^{2n+2}}}<\frac{\epsilon}{f(h n)}r^{n+\frac{1}{2}}\Big)\\
&=\sum \limits_{n=0}^{\infty}\Theta_{1}(|B^{\bullet}_{T_{r}}|>\delta)\: \Theta_{1}\big(L_{0,T_{r}}<\frac{\epsilon}{f(h n)}\big).
\end{align*}
By Proposition \ref{cycleIBD} the right-hand side of the last display is greater than $$\tilde{c}_{3}\:\Theta_{1}(|B^{\bullet}_{T_{r}}|>\delta)\sum \limits_{n=0}^{\infty} \left(\frac{\epsilon^{2}}{f(h n)^{2}}\wedge 1\right)$$ which is infinite since $\sum \limits_{m\in \mathbb{N}}f(m)^{-2}=\infty$.
The  Borel-Cantelli lemma then implies that:
\[ \liminf \limits_{n\to \infty}\frac{L_{T_{r^{2n+1}},T_{r^{2n+2}}}}{|B_{T_{r^{2n}},T_{r^{2n+1}}}^{\bullet}|^{\frac{1}{4}}}f(h n) \leq \epsilon \delta^{-\frac{1}{4}} \sqrt{r},\:\:\Theta_{0}\text{-a.s.}\]
This holds for every $\epsilon>0$, which together with \eqref{dernier**} gives \eqref{cut_T_r}.
\end{proof}

\subsection{Proof of Part $\rm(ii)$ of Theorem \ref{iso}}
We need to show that for any positive nondecreasing function $f$, the condition $\sum \limits_{m\in \mathbb{N}} f(m)^{-2}<\infty$ implies $\inf \limits_{A\in \mathcal{K}} \frac{\Delta(\partial A)}{|A|^{\frac{1}{4}}} f(|\log(|A|)|)>0$, $\Theta_{0}$-a.s.
\\
\\
We begin with a technical lemma.
\begin{lem}\label{lemtech}
Let $\beta\in [0,1)$. There exists a constant $C_{\beta}>0$, which only depends on $\beta$, such that for every $r>0$ and $\epsilon>0$:
\[\Theta_{0}\big(\mathbbm{1}_{L_{1}<\epsilon}|B_{r}^{\bullet}|^{\beta}\big)\leq C_{\beta}\: r^{4\beta}\epsilon^{2}.\]
\end{lem}
\noindent The reason for taking $\beta<1$ is just technical and one can extend the result to $\beta<\frac{3}{2}$ but the proof will be more tedious. At an intuitive level, Lemma \ref{lemtech} states that if we know that there exists a small cycle separating the hull of radius $1$ then the expected volume of the hull of radius $r$  stays at most of order $r^{4}$ (with a uniform control). 
\begin{proof}
Fix $\beta\in(0,1)$ and let $r>0$ and $\epsilon>0$. 
\\
\\
To simplify notation, set $m:=\frac{1}{\epsilon}$, $q:=\frac{1}{\beta}$ and $p:=\frac{q}{q-1}=\frac{1}{1-\beta}$. By the scaling property of $\mathcal{M}_{\infty}$:
\[\Theta_{0}\big(\mathbbm{1}_{L_{1}<\epsilon}|B_{r}^{\bullet}|^{\beta}\big)=\frac{1}{m^{4\beta}}\Theta_{0}\big(\mathbbm{1}_{L_{m}<1}|B_{mr}^{\bullet}|^{\beta}\big) .\]
If $L_{m}<1$, there is a separating cycle of length smaller than $1$ that  is contained in  $\check{B}_{m}^{\circ}$, and necessarily this separating cycle is contained in $B_{m+k,m+k+2}^{\bullet}$ for some integer $k\geq 0$. Hence, 
$$\Theta_{0}\big(\mathbbm{1}_{L_{m}<1}|B_{mr}^{\bullet}|^{\beta}\big)\leq\sum \limits_{k=0}^{\infty}  \Theta_{0}\big(\mathbbm{1}_{L_{m+k,m+k+2}<1}|B_{mr}^{\bullet}|^{\beta}\big).
$$
Applying the conditional version of the H\"older inequality with respect to $Z_{m+k+4}$ we obtain:
\begin{align*}
\Theta_{0}\big(\mathbbm{1}_{L_{m+k,m+k+2}<1}|B_{mr}^{\bullet}|^{\beta}\big)\leq  \Theta_{0}\Big(\:\Theta_{0}\big(L_{m+k,m+k+2}<1\:\big\vert\: Z_{m+k+4}\big)^{\frac{1}{p}}\: \Theta_{0}\big(|B_{mr}^{\bullet}| \:\big\vert\: Z_{m+k+4}\big)^{\frac{1}{q}}\:\Big).
\end{align*}
By Proposition \ref{exp}, there exists $\alpha_{1}>0$ such that $$\Theta_{0}\big(L_{m+k,m+k+2}<1\:\big\vert\: Z_{m+k+4}\big)\leq e\cdot\exp(-\alpha_{1}Z_{m+k+4})$$ for every $k\geq 0$ and thus we get :
\begin{equation}
\label{lemtech5}
\Theta_{0}\big(\mathbbm{1}_{L_{m}<1}|B_{mr}^{\bullet}|^{\beta}\big)\leq e\:\sum \limits_{k=0}^{\infty}  \:\Theta_{0}\Big(\exp(-\alpha Z_{m+k+4})\:\Theta_{0}\big(|B_{mr}^{\bullet}|\:\big\vert\:Z_{m+k+4}\big)^{\frac{1}{q}}\:\Big)
\end{equation}
where $\alpha:=\alpha_{1}/p$.
Then again by the H\"older inequality:
\begin{align}
\label{lemtech6}
\Theta_{0}\Big(\exp(-\alpha Z_{m+k+4})~&\Theta_{0}\big(|B_{mr}^{\bullet}|\:\big\vert\:Z_{m+k+4}\big)^{\frac{1}{q}}\:\Big)\nonumber\\
&=\Theta_{0}\Big(\exp(- \frac{\alpha}{p} Z_{m+k+4}) \exp(- \frac{\alpha}{q} Z_{m+k+4})\Theta_{0}\big(|B_{mr}^{\bullet}|\:\big\vert\:Z_{m+k+4}\big)^{\frac{1}{q}}\:\Big)\nonumber\\
&\leq  \Theta_{0}\big(\exp(-\alpha Z_{m+k+4})\big)^{\frac{1}{p}}\:\Theta_{0}\big(\exp(-\alpha Z_{m+k+4})|B_{mr}^{\bullet}|\big)^{\frac{1}{q}}.
\end{align}
By \eqref{eq_Laplace_Z}, $Z_{m+k+4}$ follows the Gamma distribution with
parameter $\frac{3}{2}$ and mean $(m+k+4)^2$. So we can find $d_{1}(\alpha)>0$ independent of $m$ and $k$ such that:
\begin{equation}
\label{bound-LaplaceZ}
\Theta_{0}(\exp(-\alpha Z_{m+k+4}))\leq \frac{d_{1}(\alpha)}{(m+k+4)^{3}}\,,\quad 
\Theta_{0}(Z_{m+k+4}\exp(-\alpha Z_{m+k+4}))\leq \frac{d_{1}(\alpha)}{(m+k+4)^{3}} .
\end{equation}
Moreover, noting that $|\partial B_{m+k+4}^{\bullet}|=0$, $\Theta_{0}$-a.s.~, we obtain:
\begin{align*}
\Theta_{0}\big( \exp(-\alpha Z_{m+k+4})|B_{mr}^{\bullet}|\big)&=\Theta_{0}\big(\exp(-\alpha Z_{m+k+4})|B_{(m+k+4)\:\wedge\:mr}^{\bullet}|\big)\\
&+ \Theta_{0}\big(\exp(-\alpha Z_{m+k+4})|B_{m+k+4,mr}^{\bullet}|\big)
\end{align*}
where by convention $|B_{s_{1},s_{2}}^{\bullet}|=0$ if $s_{2}\leq s_{1}$. 

Let $0<s_{1}\leq s_{2}$. We observe that $|B^\bullet_{s_1}|$ is independent of $Z_{s_2}$ conditionally on $Z_{s_1}$.
This follows from the special Markov property and the spine independence property, using the fact
that $|B^\bullet_{s_1}|$ is determined by the excursions of $\omega$ below $s_1$ for all atoms 
of $\mathfrak{L}$ and $\mathfrak{R}$ such that $t\geq \tau_{s_1}$, and by the atoms $(t,\omega)$
such that $t\leq \tau_{s_1}$. Thanks to this conditional independence property, we have
\begin{equation}\label{B_s_1_Z_s_2}
\Theta_{0}\Big(|B_{s_{1}}^{\bullet}|\:\big\vert\:Z_{s_{2}}\Big)=\Theta_{0}\Big(\Theta_{0}\big(|B_{s_{1}}^{\bullet}|\:\big\vert\:Z_{s_{1}}\big)\:\big\vert\:Z_{s_{2}}\Big)~.
\end{equation}
By differentiating the right-hand side of $\eqref{eq_B_sachant_Z}$ at $\lambda=0$ we get
\[\Theta_{0}\left(|B_{s_{1}}^{\bullet}|\:\big\vert\:Z_{s_{1}}\right)=\frac{2}{15}s_{1}^{4}+\frac{1}{5}s_{1}^{2} Z_{s_{1}} \]
similarly we have from \eqref{eq_Z_sachant_Z}:
\[\Theta_{0}(Z_{s_{1}}\:\big\vert\:Z_{s_{2}})=\frac{s_{1}^{3}}{s_{2}^{3}}Z_{s_{2}}+\frac{s_{2}-s_{1}}{s_{2}}s_{1}^{2} .\]
So by \eqref{B_s_1_Z_s_2}, for every $0<s_{1}\leq s_{2}$
\[\Theta_{0}\Big(|B_{s_{1}}^{\bullet}|\:\big\vert\:Z_{s_{2}}\Big)=\frac{1}{3} s_1^4 -\frac{s_1^5}{5s_2}+ \frac{s_1^5}{5s_2^3} Z_{s_2}.
\]
Taking $s_{1}=(m+k+4)\wedge (mr)$ and $s_{2}=m+k+4$, we deduce from the last two formulas and
 \eqref{bound-LaplaceZ}  that there exists $d_{2}(\alpha)>0$ independent of $m$ and $k$ such that:
\[\Theta_{0}\big(\exp(-\alpha Z_{m+k+4})|B_{(m+k+4)\:\wedge\:mr}^{\bullet}|\big)\leq  \frac{d_{2}(\alpha)}{(m+k+4)^{3}}r^{4}m^{4}.\]
Suppose that $m+k+4<m r$. Then by the spatial Markov property  and Corollary \ref{lemvolIBD},
\begin{align*}
\Theta_{0}(\: \exp(-\alpha Z_{m+k+4})|B_{m+k+4,mr}^{\bullet}|\:)&=\Theta_{0}\Big( \exp(-\alpha Z_{m+k+4})\:\Theta_{Z_{m+k+4}}\big(|B^{\bullet}_{mr-m-k-4}|\big)\:\Big)\\
&\leq C'\: \Theta_{0}(\exp(-\alpha Z_{m+k+4})(m^{4}r^{4}+Z_{m+k+4}^{2})).
\end{align*}
where $C'=16C$, if $C$ is the constant appearing in Corollary \ref{lemvolIBD}. 
Using again the distribution of $Z_{m+k+4}$ , we get that there exists a constant $d_{3}(\alpha)>0$ independent of $m$ and $k$ such that, if $m+k+4<m r$,
\[\Theta_{0}( \exp(-\alpha Z_{m+k+4})\:|B_{m+k+4,mr}^{\bullet}|\:)\leq  \frac{d_{3}(\alpha)}{(m+k+4)^{3}}\:r^{4} m^{4}. \]
Summarizing, we get  in both cases $m+k+4< mr$ and $m+k+4\geq mr$,
\[\Theta_{0}\big( \exp(-\alpha Z_{m+k+4})|B_{mr}^{\bullet}|\big)\leq \frac{d_{2}(\alpha)+d_{3}(\alpha)}{(m+k+4)^{3}}\:r^{4} m^{4}.\]
Coming back to \eqref{lemtech5} and \eqref{lemtech6}, using \eqref{bound-LaplaceZ} once again, and recalling that $q=1/\beta$ and $m=1/\epsilon$,\newpage
\noindent we can find a constant $d(\alpha)>0$ such that:

\begin{align*}
\Theta_{0}\big(\mathbbm{1}_{L_{1}<\epsilon}|B_{r}^{\bullet}|^{\beta}\big)&=\frac{1}{m^{4\beta}}
\Theta_{0}(\mathbbm{1}_{L_{m}<1}|B_{mr}^{\bullet}|^{\beta})\\
&\leq \frac{c}{m^{4\beta}} \sum \limits_{k=0}^{\infty} \Theta_{0}\big(\exp(-\alpha Z_{m+k+4})\big)^{\frac{1}{p}}\:\Theta_{0}\big(\exp(-\alpha Z_{m+k+4})|B_{mr}^{\bullet}|\big)^{\frac{1}{q}}\\
&\leq  \sum\limits_{k=0}^{\infty} \frac{d(\alpha)}{(m+k+4)^{3}} r^{4 \beta}
\end{align*}
and the lemma follows since $m=\epsilon^{-1}$.
\end{proof}
\vspace{0.4 cm}
\noindent We now use Lemma \ref{lemtech} to prove that for any nondecreasing positive function $f$:
\begin{equation}
\label{second-part-thm}
\sum \limits_{m\in \mathbb{N}} \frac{1}{f(m)^2}<\infty\implies \inf \limits_{A\in \mathcal{K}} \frac{\Delta(\partial A)}{|A|^{\frac{1}{4}}} f(|\log(|A|)|)>0\:,\:\Theta_{0}\text{-a.s.} 
\end{equation}
\begin{proof}
Fix a nondecreasing function $f:\mathbb{R}_{+}\to \mathbb{R}_{+}^{*}$ such that $\sum \limits_{m\in \mathbb{N}}f(m)^{-2}<\infty.$ We begin by showing that for every $A\in \mathcal{K}$ we have:
\begin{equation}\label{inf_iso_A}
\frac{\Delta(\partial A)}{|A|^{\frac{1}{4}}} f(|\log(|A|)|)\geq \Big(\inf\limits_{m\in \mathbb{Z}}\frac{2^{m-1}}{|B_{2^{m+1}}^{\bullet}|^{\frac{1}{4}}}f\big(|\log(|B_{2^{m}}^{\bullet}|)|\big)\Big)\wedge \Big(\inf\limits_{m\in \mathbb{Z}} \frac{L_{2^{m-1}}}{|B_{2^{m+1}}^{\bullet}|^{\frac{1}{4}}}f\big(|\log(|B_{2^{m}}^{\bullet}|)|\big)\Big).
\end{equation}
Let $A\in \mathcal{K}$ and let $m$ be the unique element of $\mathbb{Z}$ such that:
\[|B_{2^{m}}^{\bullet}|< |A|\leq |B_{2^{m+1}}^{\bullet}|.\]
We divide the proof of \eqref{inf_iso_A} in two cases: 
\\
\\
$\bullet$ Case 1: Assume that $\partial A$ intersects  $\partial B^{\bullet}_{2^{m-1}}$. As $|B^{\bullet}_{2^{m}}|<|A|$, there exists $x\in A\setminus B^{\bullet}_{2^{m}}$. Consider a path $p_{\infty}$  connecting $x$ to $\infty$ that does not hit $B^{\bullet}_{2^{m}}$ and let $y$ be the last point of $p_{\infty}$ that belongs to $A$. By construction we have $y\in \partial A$ and $y\notin B^{\bullet}_{2^{m}}$. Finally let $z\in \partial A \cap \partial B^{\bullet}_{2^{m-1}}$. Since $\partial A$ connects $y$ and $z$ we have
\[\Delta(\partial A)\geq \Delta(y,z)\geq 2^{m-1} .\]
This implies that:
\[\frac{\Delta(\partial A)}{|A|^{\frac{1}{4}}} f(|\log(|A|)|)\geq \frac{2^{m-1}}{|B^{\bullet}_{2^{m+1}}|^{\frac{1}{4}}} f(|\log(|B^{\bullet}_{2^{m}}|)|).\]
~
\\
\\
$\bullet$ Case 2: Assume $\partial A$ does not intersect  $\partial B^{\bullet}_{2^{m-1}}$. Since $|A|>|B_{2^{m-1}}^{\bullet}|$, the set $A$ is not contained in $B_{2^{m-1}}^{\bullet}$. This implies that  $\partial A$ separates $B^{\bullet}_{2^{m-1}}$ from infinity, and consequently $\Delta(\partial A)\geq L_{2^{m-1}}$. It follows that:
\[\frac{\Delta(\partial A)}{|A|^{\frac{1}{4}}} f(|\log(|A|)|)\geq\frac{L_{2^{m-1}}}{|B^{\bullet}_{2^{m+1}}|^{\frac{1}{4}}} f(|\log(|B^{\bullet}_{2^{m}}|)|)\]
and this completes the proof of \eqref{inf_iso_A}. 
\begin{figure}[!h]
 \begin{center}
 \includegraphics[height=5.5cm,width=6.5cm]{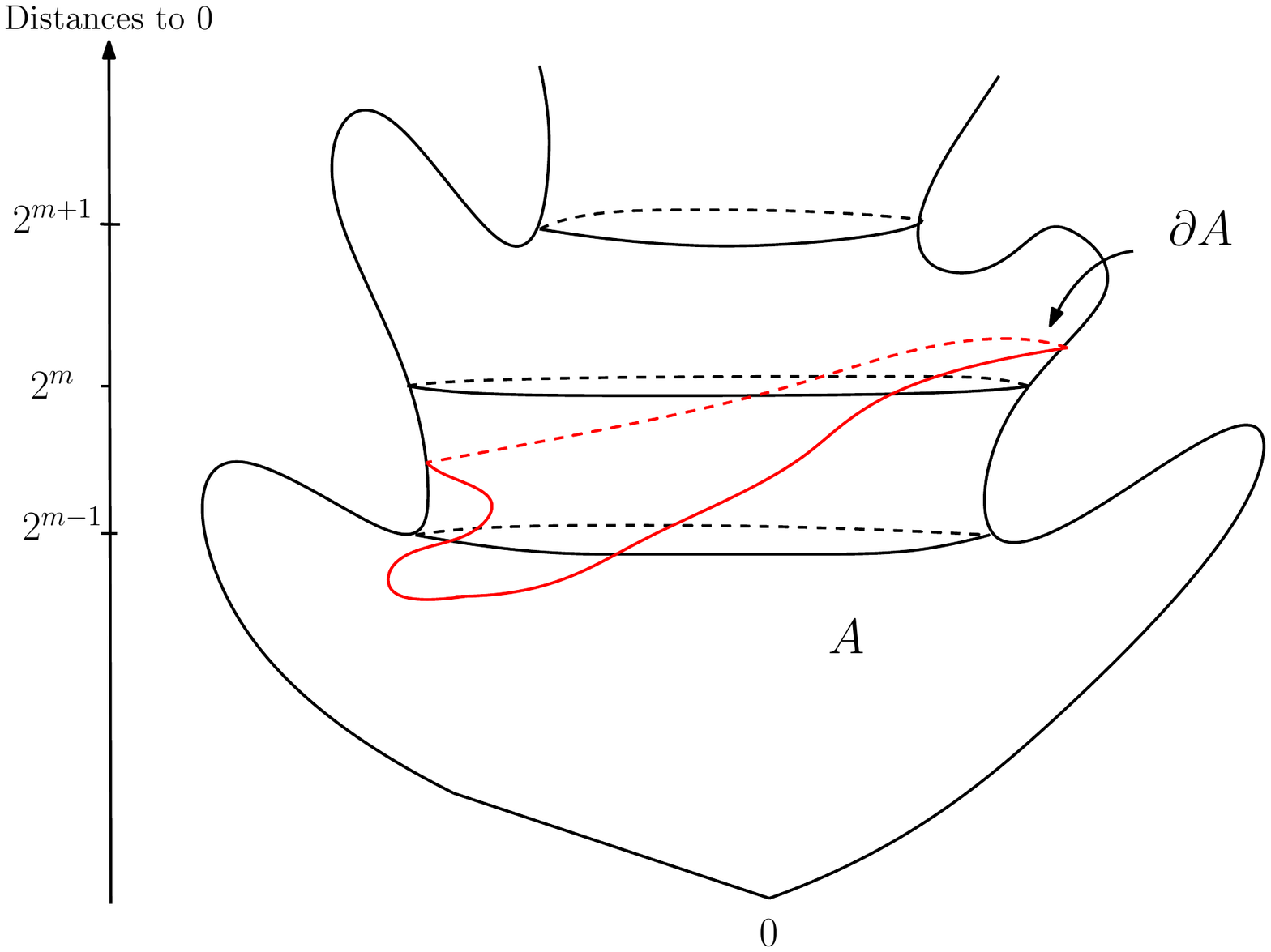}\:\:\:\:\:\:\:\:\:\:\:\:
  \includegraphics[height=5.5cm,width=6.5cm]{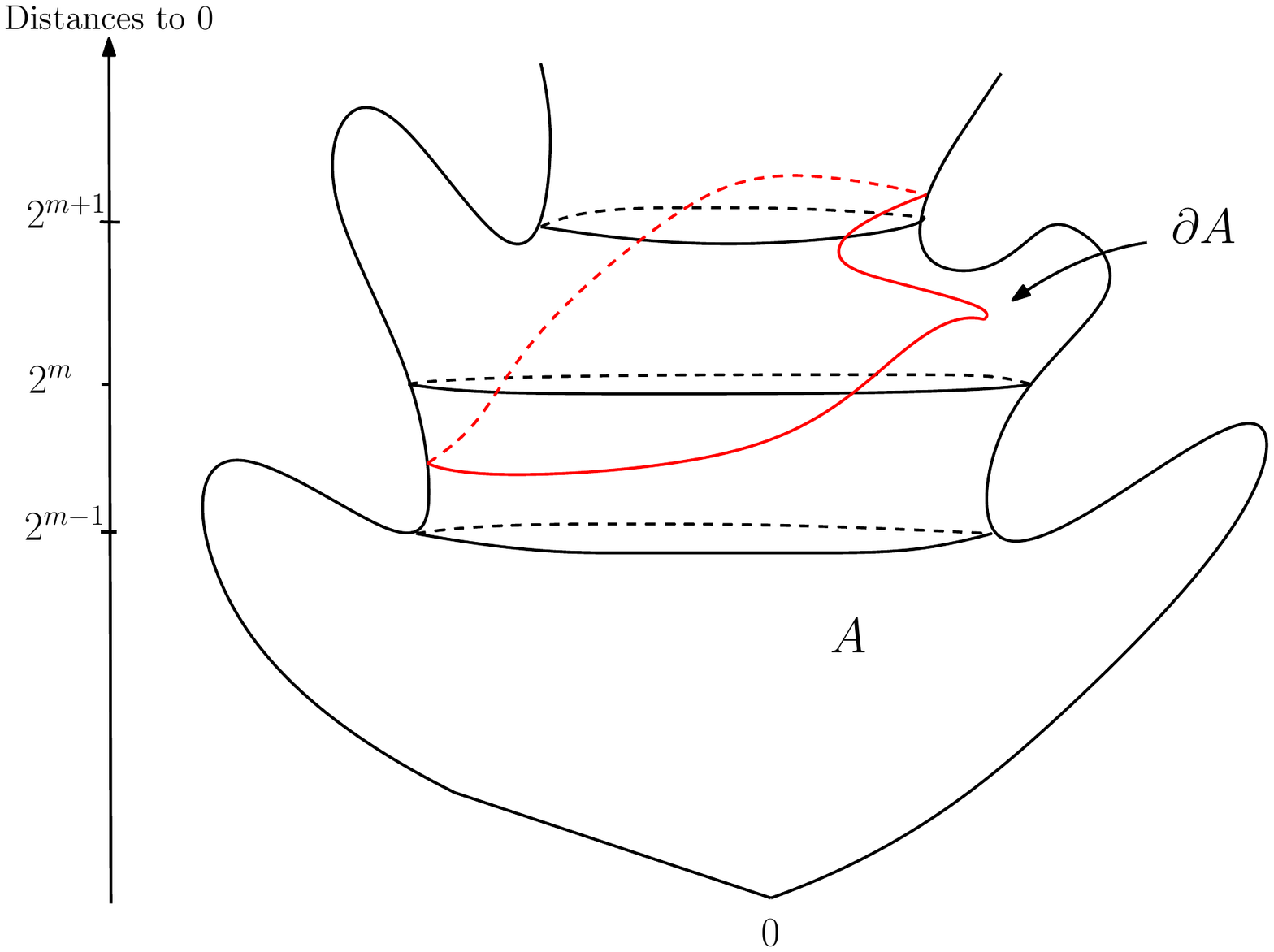}
 \caption{Illustration of \eqref{inf_iso_A}. In red we represent the boundary of $A$. On the left we are in case 1 and we have $\Delta(\partial A)\geq 2^{m-1}$. On the right we are in case 2 and we have $\Delta(\partial A)\geq L_{2^{m-1}}$.}
 \end{center}
 \end{figure}
\newpage
Thanks to \eqref{inf_iso_A}, the proof of \eqref{second-part-thm} will be complete if we can verify that:
\begin{equation}\label{final:1}
\inf\limits_{m\in \mathbb{Z}}\frac{2^{m-1}}{|B_{2^{m+1}}^{\bullet}|^{\frac{1}{4}}}f\big(|\log(|B_{2^{m}}^{\bullet}|)|\big)>0,\:\:\Theta_{0}\text{-a.s.}
\end{equation}
and 
\begin{equation}\label{final:2}
\inf\limits_{m\in \mathbb{Z}}\frac{L_{2^{m-1}}}{|B_{2^{m+1}}^{\bullet}|^{\frac{1}{4}}}f\big(|\log(|B_{2^{m}}^{\bullet}|)|\big)>0,\:\:\Theta_{0}\text{-a.s.}
\end{equation}
Let us start by proving \eqref{final:1}. By Lemma \ref{CantelliH}, $\Theta_{0}\text{-a.s.}$, there is a positive integer $M$ such that for every $m\in \mathbb{Z}$ with $|m|\geq M$:
\begin{equation}\label{control_m}
\frac{1}{|m|}2^{4m}\leq |B_{2^{m}}^{\bullet}|\leq |m|2^{4m}
\end{equation}
In particular, we have:
\[\mathop{\inf_{m\in \mathbb{Z}}}_{|m|>M}\frac{2^{m-1}}{|B_{2^{m+1}}^{\bullet}|^{\frac{1}{4}}}f(|\log(|B_{2^{m}}^{\bullet}|)|)\geq \frac{1}{4}\mathop{\inf_{m\in \mathbb{Z}}}_{|m|>M}\frac{f\big(|4\log(2)m-\log(|m|)|\big)}{(|m|+1)^{\frac{1}{4}}}\:\:\Theta_{0}\text{-a.s.}\]
On the other hand,  by the Cauchy-Schwarz inequality:
\[
\frac{1}{f(n)}\leq\frac{1}{n} \sum \limits_{k=1}^{n} \frac{1}{f(k)} \leq \frac{1}{n^{\frac{1}{2}}} \big(\sum \limits_{k=1}^{n} \frac{1}{f(k)^{2}}\big)^{\frac{1}{2}} 
\leq \frac{1}{n^{\frac{1}{2}}} \big(\sum \limits_{k=1}^{\infty} \frac{1}{f(k)^{2}}\big)^{\frac{1}{2}}.\]
Consequently $\inf \limits_{n\in \mathbb{N}}n^{-\frac{1}{2}}f(n)>0$ and we obtain \eqref{final:1}. Actually here it will be enough to have $\inf \limits_{n\in \mathbb{N}}n^{-\frac{1}{4}}f(n)>0$. 
\\
\\
Let us prove \eqref{final:2}.
By \eqref{control_m}, if we can verify that
\[\sum \limits_{m\in \mathbb{Z}} \Theta_{0}\big(\frac{L_{2^{m-1}}}{|B^{\bullet}_{2^{m+1}}|^{\frac{1}{4}}}<\frac{1}{f(|m|)}\big)<\infty\]
we will conclude by an application of the Borel-Cantelli lemma. 
Fix $\beta\in (\frac{3}{4},1)$. For every  $m\in \mathbb{Z}$, by scaling we have:
\begin{align*}
 \Theta_{0}\big(\frac{L_{2^{m-1}}}{|B^{\bullet}_{2^{m+1}}|^{\frac{1}{4}}}< \frac{1}{f(|m|)}\big)&=   \Theta_{0}\big(\frac{L_{1}}{|B^{\bullet}_{4}|^{\frac{1}{4}}}<\frac{1}{f(|m|)}\big).
\end{align*}
Consequently, we get : 
\begin{align*}
 \Theta_{0}\big(\frac{L_{2^{m-1}}}{|B^{\bullet}_{2^{m+1}}|^{\frac{1}{4}}}< \frac{1}{f(|m|)}\big)&\leq   \Theta_{0}\big(L_{1}<\frac{1}{f(|m|)}\:,\:|B_{4}^{\bullet}|\leq 1\big)\\
 &+ \sum \limits_{n=1}^{\infty}  \Theta_{0}\big(|B^{\bullet}_{4}|^{\frac{1}{4}}\in[n,n+1]\:,\: L_{1}<\frac{n+1}{f(|m|)}\big).
 \end{align*}
 Now remark that the right term of the above display is bounded above by 
\begin{align*} 
 \Theta_{0}\big(L_{1}<\frac{1}{f(|m|)}\big)+ \sum \limits_{n=1}^{\infty}  \Theta_{0}\big(|B^{\bullet}_{4}|^{\beta}\geq n^{4\beta},\: L_{1}<\frac{n+1}{f(|m|)}\big).
 \end{align*}
 We deduce by an application of Markov inequality that
$$ \Theta_{0}\big(\frac{L_{2^{m-1}}}{|B^{\bullet}_{2^{m+1}}|^{\frac{1}{4}}}< \frac{1}{f(|m|)}\big)\leq   \Theta_{0}(L_{1}<\frac{1}{f(|m|)})+ \sum \limits_{n=1}^{\infty}\frac{1}{n^{4\beta}} \Theta_{0}(|B_{4}^{\bullet}|^{\beta} \mathbbm{1}_{L_{1}<\frac{n+1}{f(|m|)}}) $$
 Lemma \ref{lemtech} and Theorem  \ref{tail} imply that there exist two constants $c_{2}\in(0,\infty)$ and $C\in(0,\infty)$ such that:
\[
\sum \limits_{m\in \mathbb{Z}}  \Theta_{0}(\frac{L_{2^{m-1}}}{|B^{\bullet}_{2^{m+1}}|^{\frac{1}{4}}}< \frac{1}{f(|m|)})\leq 2 c_{2} \sum \limits_{m\in \mathbb{N}}^{\infty} \frac{1}{f(m)^{2}}+ C \sum \limits_{n=1}^{\infty}\frac{(n+1)^{2}}{n^{4\beta}} \sum \limits_{m\in \mathbb{N}} \frac{1}{f(m)^{2}}<\infty .
\]
This completes the proof.   
\end{proof}

We observe that the same proof will work mutatis mutandis if we replace $|A|$ by $\Delta(\partial A)$ inside the logarithm in theorem \ref{iso}. 
\\
\\
Recall that $\mathcal{M}_{\infty}^{(z)}$ stands for an infinite volume Brownian disk with perimeter $z$.
With a slight abuse of terminology, we call Jordan domain of $\mathcal{M}_{\infty}^{(z)}$ the closure of the bounded component of the 
complement of an injective cycle of $\mathcal{M}_{\infty}^{(z)}$.
As a direct consequence of Corollary \ref{SpaMarkov_T_z}, Proposition \ref{infand0} $\rm(i)$ and Theorem \ref{iso} we obtain:
\begin{cor}
\label{isocor}~\\
Fix $z>0$. Let $\mathcal{M}_{\infty}^{(z)}$ be the infinite volume Brownian disk with perimeter $z$ defined under 
the probability measure $\Theta_{z}$. Consider the collection 
$\mathcal{K}^{(z)}$ of all Jordan domains of $\mathcal{M}_{\infty}^{(z)}$ whose interior contains the boundary of $\mathcal{M}_{\infty}^{(z)}$. 
For any nondecreasing function $f:\mathbb{R}_{+}\to \mathbb{R}_{+}^{*}$
\\
\\
$\rm(i)$ We have 
\[\inf \limits_{A\in \mathcal{K}^{(z)}} \frac{\Delta\big(\partial A\big)}{|A|^{\frac{1}{4}}} f(|\log(|A|)|)=0 \:,\:\Theta_{z}\text{-a.s.}\:,\:\text{if}\:\: \sum \limits_{m\in \mathbb{N}}\frac{1}{f(m)^{2}}=\infty.\]
\\
\\
$\rm(ii)$ We have
\[\inf \limits_{A\in \mathcal{K}^{(z)}} \frac{\Delta\big(\partial A\big)}{|A|^{\frac{1}{4}}} f(|\log(|A|)|)>0 \:,\:\Theta_{z}\text{-a.s.}\:,\:\text{if} \:\:\sum \limits_{m\in \mathbb{N}}\frac{1}{f(m)^{2}}<\infty.\]

\end{cor}

\subsection*{Appendix: Proof of Lemma  \ref{lem:appen}}

This appendix is devoted to the proof of Lemma \ref{lem:appen}, which relies on  \cite[Proposition 8]{Growth}. We use the notation of Subsection \ref{subsub-Plane}.

\smallskip
{\sc Proof of Lemma \ref{lem:appen}}.

 First fix $0<r_1<r_2<\infty$. The lemma will follow if we prove that, $\Theta_0$-a.s.~, for every $r_1\leq r\leq r_2$, we have:
\begin{equation*}
Z_{r}=\lim \limits_{\epsilon \downarrow 0} \frac{1}{\epsilon^{2}} |\check{B}^{\circ}_{r}\cap B_{r+\epsilon}|.
\end{equation*}
In order to prove this, we introduce the event $A(s):=\{Z_{r_2}^{s,\infty}=0\}$, for every $s>r_2$. In particular, under the event $A(s)$, we have $Z_{r}=Z_{r}^{r,s}$ for every   $r\leq r_2$.

 Moreover, by Proposition  \ref{corZabc0} (taking the limit when $t\to \infty$) we also have:
$$\Theta(A(s))=\big(\frac{s-r}{s}\big)^{3},$$ 
which converges to $1$ when $s\to \infty$. Consequently, to obtain the desired result it is sufficient to show that, for every $s>r_2$, under $A(s)$  we have:
\begin{equation*}
Z_{r}^{r,s}=\lim \limits_{\epsilon \downarrow 0} \frac{1}{\epsilon^{2}} |\check{B}^{\circ}_{r}\cap B_{r+\epsilon}|,
\end{equation*}
for every $r\in [r_1,r_2]$.
Let us now introduce, for every $r\in \mathbb{R}$ and $\omega\in \mathcal{S}$ with $\omega_0>r$, the quantity:
\[\mathcal{Z}_{r}^{\epsilon}(\omega):=\frac{1}{\epsilon^{2}}\int_{0}^{\sigma}ds\mathbbm{1}_{\text{hit}_{r}(\omega_{s})=\infty,~\widehat{\omega}_{s}<r+\epsilon}.\] 
In particular, note that $\mathcal{Z}_{r}(\omega)=\liminf_{\epsilon \to 0}\mathcal{Z}_{r}^{\epsilon}(\omega)$. We set:
\[ Z_{r}^{r,s}(\epsilon):=\int \mathcal{Z}_{r}^{\epsilon}(\omega)\mathfrak{R}^{r,s}(d\ell d\omega)+\int \mathcal{Z}_{r}^{\epsilon}(\omega)\mathfrak{L}^{r,s}(d\ell d\omega).\]
Under $A(s)$, all the labels appearing after the point $\tau_s$ of the spine are greater than $r_2$. This implies that, under $A(s)$,  the quantity $\epsilon^{2} Z_{r}^{r,s}(\epsilon)$ is exactly $|\check{B}^{\circ}_{r}\cap B_{r+\epsilon}|$ (since $|\cdot|$ is the pushforward of Lebesgue measure under $\Pi\circ \mathcal{E}$). To conclude, we are going to use \cite[Proposition 8]{Growth}, which states that,  for every $s>0$ and $\beta>0$, we have
\begin{equation}\label{int_Z^eps}
\sup_{r\in(-\infty,s-\beta]} |\mathcal{Z}_{r}^{\epsilon}-\mathcal{Z}_{r}| ~\underset{\epsilon \to 0} {\longrightarrow}~0\;,\qquad \mathbb{N}_s\hbox{-a.e.}
\end{equation}
To translate \eqref{int_Z^eps} in terms of $Z_{r}^{r,s}$ and  $Z_{r}^{r,s}(\epsilon)$, we recall that $(X_{(\tau_s-\ell)\vee 0})_{\ell\geq 0}$ is a Bessel process of dimension $-5$ started from $s$, and that, conditionally on $(X_{(\tau_s-\ell)\vee 0})_{\ell\geq 0}$, the measures $\mathfrak{R}^{0,s}$ and $\mathfrak{L}^{0,s}$ are two independent  Poisson point measures on $\mathbb{R}_{+}\times \mathcal{S}$ with intensity:
\[2\mathbbm{1}_{[0,\tau_s]}(\ell)~d\ell~
\mathbb{N}_{X_{\ell}}(d\omega\cap\{\omega_{*}>0\}).\]
 In particular, the distribution $(X_{(\tau_s-\ell)\vee 0})_{0\leq \ell\leq \tau_s-\tau_{r_1}}$ is absolute continuous with respect to the distribution of a Brownian motion started from $s$ and stopped when it hits $r_1$. We can now apply  \cite[Proposition 2]{Infinite_Spine} to deduce that the distribution of $(Z_{r}^{r,s},Z_{r}^{r,s}(\epsilon))_{r\in [r_1,r_2]}$ is absolute continuous with respect to the distribution $(\mathcal{Z}_{r-r_1},\mathcal{Z}^{\epsilon}_{r-r_1})_{r\in [r_1,r_2]}$ under $\mathbb{N}_{s-r_1}$. Consequently  \eqref{int_Z^eps} gives that:
\begin{equation*}
\sup_{r\in[r_1,r_2]} |Z_{r}^{r,s}(\epsilon)-Z_{r}^{r,s}| ~\underset{\epsilon \to 0} {\longrightarrow}~0\;,\qquad \Theta_0\hbox{-a.s.}
\end{equation*}
which completes the proof. \hfill$\square$

\section*{Acknowledgements}
I warmly thank Jean-François Le Gall for all his suggestions and for carefully reading earlier versions of this manuscript.  The present work was supported by the ERC Advanced Grant 740943 GEOBROWN.

\end{document}